\documentclass[10pt,twoside,a4paper,english,leqno]{article}

\title{Scattering and Localization Properties of Highly Oscillatory Potentials}
\author{V. Duch\^ene\footnotemark[1], I. Vuki{\'c}evi{\'c}\footnotemark[1] \& M.I. Weinstein\thanks{Department of Applied Physics and Applied Mathematics, Columbia University, 200 S.W. Mudd Building, Mail Code 4701, New York, NY 10027, USA}}
\date{\today}

\usepackage{babel}   
\usepackage[utf8]{inputenc} 
\usepackage[T1]{fontenc}  
\usepackage[ec]{aeguill}  
\usepackage{amsmath, amsthm} 
\usepackage{amsfonts,amssymb}
\usepackage{float}   
\usepackage{graphicx}  
\usepackage{fancyhdr}  
\usepackage{url}    
\usepackage{verbatim}  
\usepackage{wrapfig}   
\usepackage[hang,centerlast]{subfigure} 
\numberwithin{equation}{section}

\usepackage{enumerate}

\newcommand{\RR}{\mathbb{R}}
\newcommand{\CC}{\mathbb{C}}

\newcommand{\NN}{\mathbb{N}}
\renewcommand{\t}{\tilde}
\newcommand{\m}{\mathfrak{m}}
\renewcommand{\L}{\mathcal{L}}
\renewcommand{\O}{\mathcal{O}}
\newcommand{\eps}{\epsilon}
\newcommand{\dd}{\ \text{d}}

\newcommand{\seff}{{\sigma_{\text{eff}}^\eps}}
\newcommand{\nn}{\nonumber}

\newtheorem{theorem}{Theorem}[section]
\newtheorem{Definition}[theorem]{Definition}
\newtheorem{Proposition}[theorem]{Proposition}
\newtheorem{Corollary}[theorem]{Corollary}
\newtheorem{Lemma}[theorem]{Lemma}
\newtheorem{Remark}[theorem]{Remark}

\makeatletter
\let\Title\@title
\let\Author\@author
\makeatother
 
\fancypagestyle{thestyle}{%
	\fancyfoot{}
	\fancyhead[CE]{\Title} 
	\fancyhead[CO]{\scriptsize V. Duch\^ene, I. Vuki{\'c}evi{\'c} \& M.I. Weinstein} 
	\fancyhead[RE,LO]{} 
	\fancyhead[LE,RO]{\scriptsize\thepage}
	}
\fancypagestyle{plain}{\fancyhead{} } 

\begin{document}

\maketitle

\begin{abstract} 
 We investigate scattering, localization and dispersive time-decay properties for the one-dimensional Schr\"odinger equation with a rapidly oscillating and spatially localized potential, $q_\eps=q(x,x/\eps)$, where $q(x,y)$ is periodic and mean zero with respect to $y$. Such potentials model a microstructured medium. Homogenization theory fails to capture the correct low-energy ($k$ small) behavior of scattering quantities, {\it e.g.} the transmission coefficient, $t^{q_\eps}(k)$, as $\eps$ tends to zero. We derive an {\it effective potential well}, $\seff(x)=-\eps^2\Lambda_{\rm eff}(x)$, such that $t^{q_\eps}(k)-t^\seff(k)$ is small, uniformly for $k\in\mathbb{R}$ as well as in any bounded subset of a suitable complex strip. Within such a bounded subset, the scaled limit of the transmission coefficient has a universal form, depending on a single parameter, which is computable from the effective potential. A consequence is that if $\eps$, the scale of oscillation of the microstructure potential, is sufficiently small, then there is a pole of the transmission coefficient (and hence of the resolvent) in the upper half plane, on the imaginary axis at a distance of order $\eps^2$ from zero. It follows that the Schr\"odinger operator $H_{q_\eps}=-\partial_x^2+q_\eps(x)$ has an 
 $L^2$ bound state with negative energy situated
 a distance $\mathcal{O}(\eps^4)$ from the edge of the continuous spectrum. Finally, we use this detailed information to prove the local energy time-decay estimate: ${\big|(1+|\cdot|)^{-3}e^{-i t H_{q_\eps}} P_c \psi_0\big|_{L^\infty} \le C\ t^{-1/2}\ 
 \left(1+\eps^4\ \big(\int_\mathbb{R}\Lambda_{\rm eff}\right)^2 t\ \big)^{-1} \big|(1+|\cdot|^3)\psi_0\big|_{L^1}}$\ ,  where $P_c$ denotes the projection onto the continuous spectral part of $H_{q_\eps}$.
\end{abstract}

\maketitle 






\section{Introduction}

 We investigate scattering and localization phenomena for the one-dimensional 
 Schr\"odinger equation, $i\partial_t\psi=(-\partial_x^2+V(x))\psi$, where $V$ denotes a real-valued, rapidly oscillating and spatially localized potential. This equation governs the behavior of a quantum particle or, in the paraxial  approximation of electromagnetics, waves in a medium with strong and rapidly varying inhomogeneities. We find interesting and subtle low energy behavior and study its consequences for scattering, localization and dispersive time-decay. Our results imply the existence of  waveguide  modes which display very short length-scale localization of light in photonic microstructures \cite{DDEW-inprogress}.

The scattering problem for the Schr\"odinger equation
 \begin{align}
&\left( H_V\ -\ k^2 \right) u\ =\ 0,\ \ \ \ H_V\ \equiv\ -\partial_x^2\ +\ V(x),
\label{Hq-def}
\end{align}
 is the question of the scattered field in response to an incoming plane wave, $e^{ikx}$:
\begin{align}
u(x;k)\ =\ 
\left\{\begin{array}{ll} e^{ikx}\ +\ r^V(k)e^{-ikx},\ \ &x\to-\infty\ ,\\
 t^V(k) e^{ikx},\ \ &x\to+\infty\ .\end{array}\right.
 \label{scattering} \end{align}
$t^V(k)$ and $r^V(k)$ are called reflection and transmission coefficients for the potential~$V$; see section~\ref{sec:quicksum}. Considered as a function of a complex variable, $k$, the transmission coefficient, $t^V(k)$, is meromorphic in the upper half $k-$plane, having possibly simple poles located on the positive imaginary axis. If $i\rho,\ \rho>0$, is a pole of $t^V$ then $E=-\rho^2$ is a discrete eigenvalue of $H_V$ of multiplicity one.

In this paper, we are interested in the case where $V(x)$ is spatially localized and highly oscillatory. A class of potentials to which our results apply are potentials of the form:
\begin{equation}
V_\eps(x) \ = \ q_{\rm av}(x)+q(x,x/\eps),\ \ \eps\ll1\ .\label{Vdef}
\end{equation}
Here, $q_{\rm av}(x)$ denotes a spatially localized background average potential and $q(x,y)$ 
a potential which is spatially localized on the slow scale, $x$, and periodic and mean zero on the fast scale $y$:
\begin{equation}
q(x,y+1)\ =\ q(x,y), \qquad \text{and} \qquad 
\int_0^1q(x,y)\ \dd y\ =\ 0.
\label{1periodicmean0}\end{equation}
Thus,
\begin{equation}
q(x,y)\ =\ \sum_{j\ne0}\ q_j(x)\ e^{2\pi i j y}.
\label{q-eps-per}
\end{equation}
More generally, our theory admits potentials which are aperiodic. For example, we allow for real-valued potentials:
\begin{equation}
q(x,y)\ =\ \sum_{j\ne0}\ q_j(x)\ e^{2\pi i\lambda_j y},
\label{q-eps-gen}
\end{equation}
where $\{\lambda_j\}_{j\in\mathbb{Z}\setminus\{0\}}$ is a sequence of non-zero distinct frequencies for which
 there is a constant $\theta>0$ such that
\begin{equation}
\inf_{j\ne k} |\lambda_j-\lambda_k|\ge\theta>0, \quad \inf_{j\in\mathbb{Z}} |\lambda_j|\ge\theta>0.
\label{theta-lb}
\end{equation}
That $q$ is real-valued is imposed by:
\begin{align}
\overline{q_j(x)}\ &=\ q_{-j}(x),\ \qquad\ \ \lambda_{-j}\ =\ -\lambda_j,\ \ \ \ \ j\in\mathbb{Z}\setminus\{0\}\ .
\end{align}
\noindent We ask the following:
 
\noindent {\it Question:\ What are the characteristics of solutions to the scattering problem~\eqref{Hq-def}, \eqref{scattering} in the limit as $\eps$ tends to zero?}\smallskip
 
 For fixed $k\ne0$, this is the regime where averaging or homogenization theory applies; the leading order behavior in $\eps$ is governed by the average of $V_\eps$ over its fast variations. To simplify the present motivating discussion we consider the case where $V_\eps$ is periodic on the fast scale with vanishing mean,
 satisfying~\eqref{1periodicmean0}.
 Then, for any fixed $k\ne0$, as $\eps\to0$, we have
\[
 t^{V_\eps}(k) \to t^{0}(k)\equiv 1,\ \ r^{V_\eps}(k)\to r^{0}(k)\equiv0\ ;
\]
 see~\cite{DucheneWeinstein:11}, which contains very detailed asymptotic expansions of $t^{V_\eps}(k)$ for a general class of $V_\eps$, admitting singularities.
Very generally, as $k$ tends to infinity, $t^V(k)\to1$; the large $k$ transmission behavior of $V_\eps(x)$ 
and its average, $q_{\rm av}(x)$, agree. 

However, the low energy, $k\approx0$, comparison between the scattering behavior for $q_{\rm av}(x)\equiv0$ and $V_\eps(x)$ is far more subtle. First of all, the potential $V(x)\equiv0$ has non-generic low energy behavior!\ Indeed, for {\it generic} localized potentials, $V$, $ \lim_{k\to0} t^V(k)=0$; see the discussion of and references to genericity in Section~\ref{sec:quicksum}.
Thus we expect (and our analysis implies for small and non-zero $\eps$) that $t^{V_\eps}(k)\to0$ as $k\to0$; see Corollary~\ref{cor:qeps-is-generic}. 
%

 It follows that the convergence of $t^{V_\eps}(k)$, as $\eps$ tends to zero, to the {\it homogenized transmission coefficient} $t^{q_{\rm av}}(k)\equiv t^{0}(k)\equiv 1$ is non-uniform in a neighborhood of $k=0$. Figure~\ref{fig:subfigc} displays plots of $t^{V_\eps}(k)$ for several successively smaller values of $\eps$. 
{\it Underlying this non-uniformity is a subtle behavior of $t^{V_\eps}(k)$ in the complex plane and an 
interesting localization phenomenon, which we now explain.}

 \begin{figure}
 \subfigure[$V_\eps(x) =\mathbf{1}_{[-1;1]}(x)10e^{\frac{-x^2}{1-x^2}}\cos(2\pi x/\eps)$]{
\includegraphics[width=0.48\textwidth]{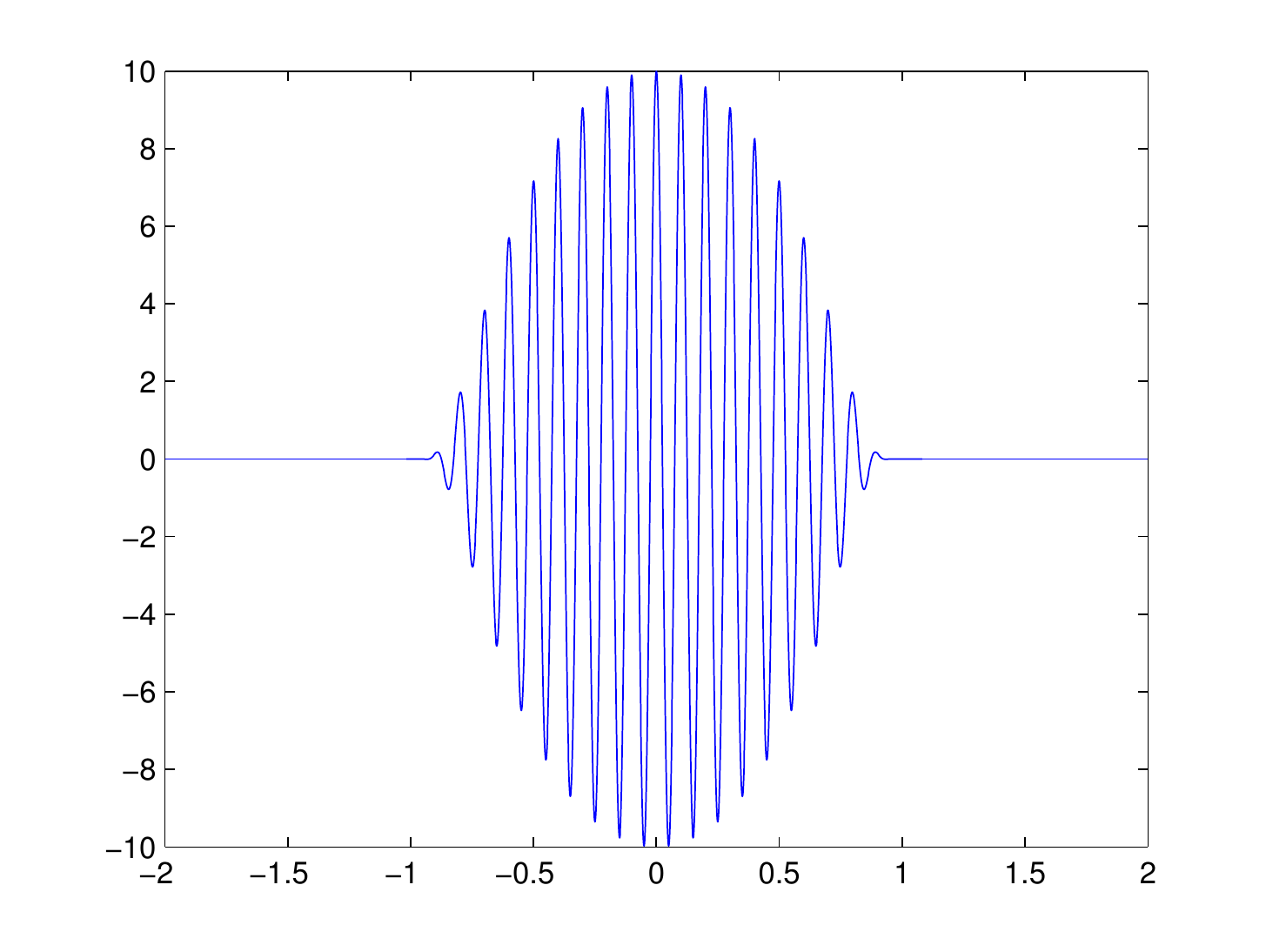}
\label{fig:subfiga}
}
\subfigure[$\seff(x) =\frac{-\eps^2}{8\pi^2}\mathbf{1}_{[-1;1]}(x)\Big(10e^{\frac{-x^2}{1-x^2}}\cos\big(\frac{2\pi x}\eps\big)\Big)^2$]{
\includegraphics[width=0.48\textwidth]{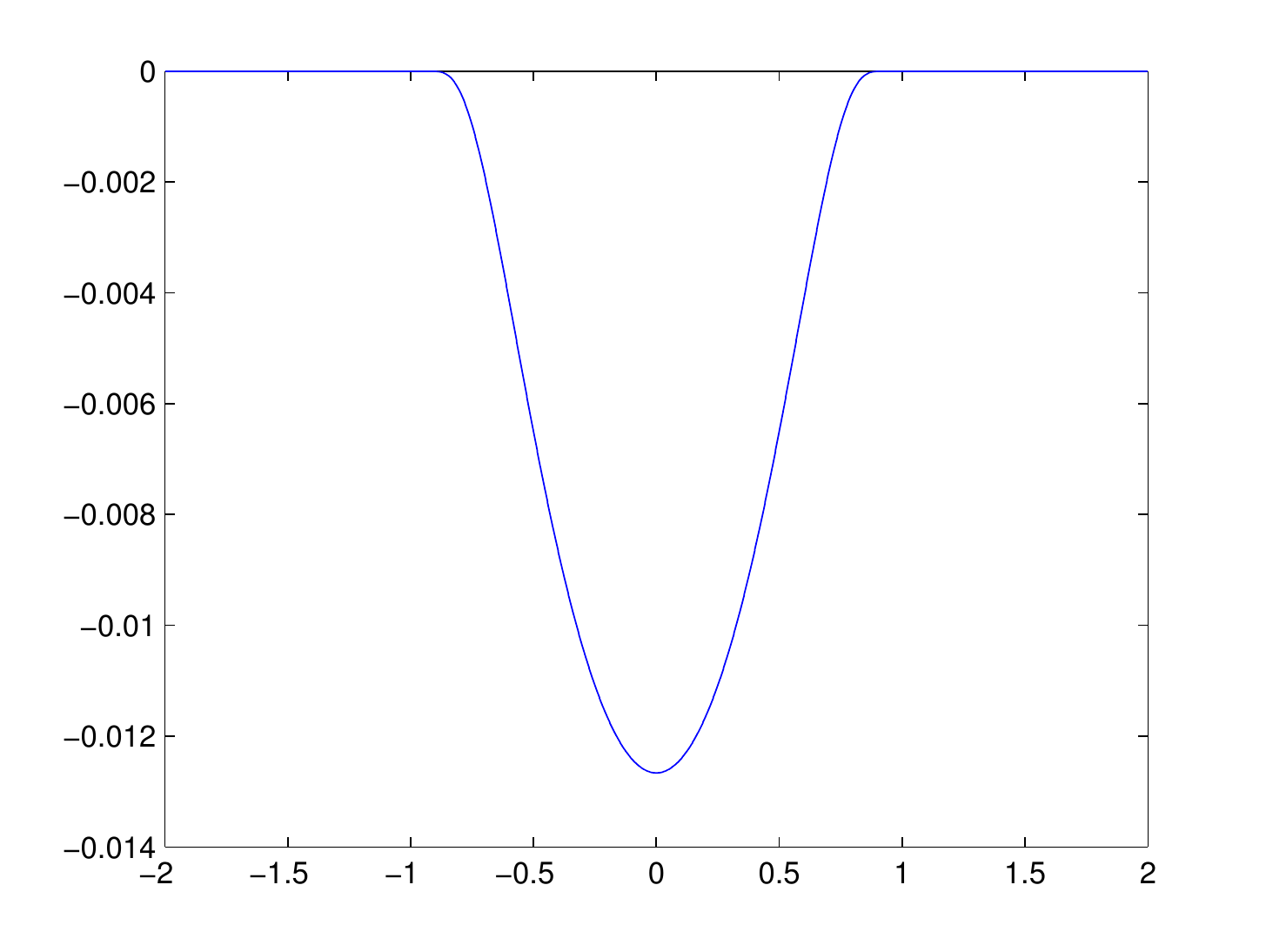}
\label{fig:subfigb}
}
\subfigure[$|t^{V_\eps}(k)|$, $k \in (0;0.1)$, $\eps$ varying]{
\includegraphics[width=0.48\textwidth]{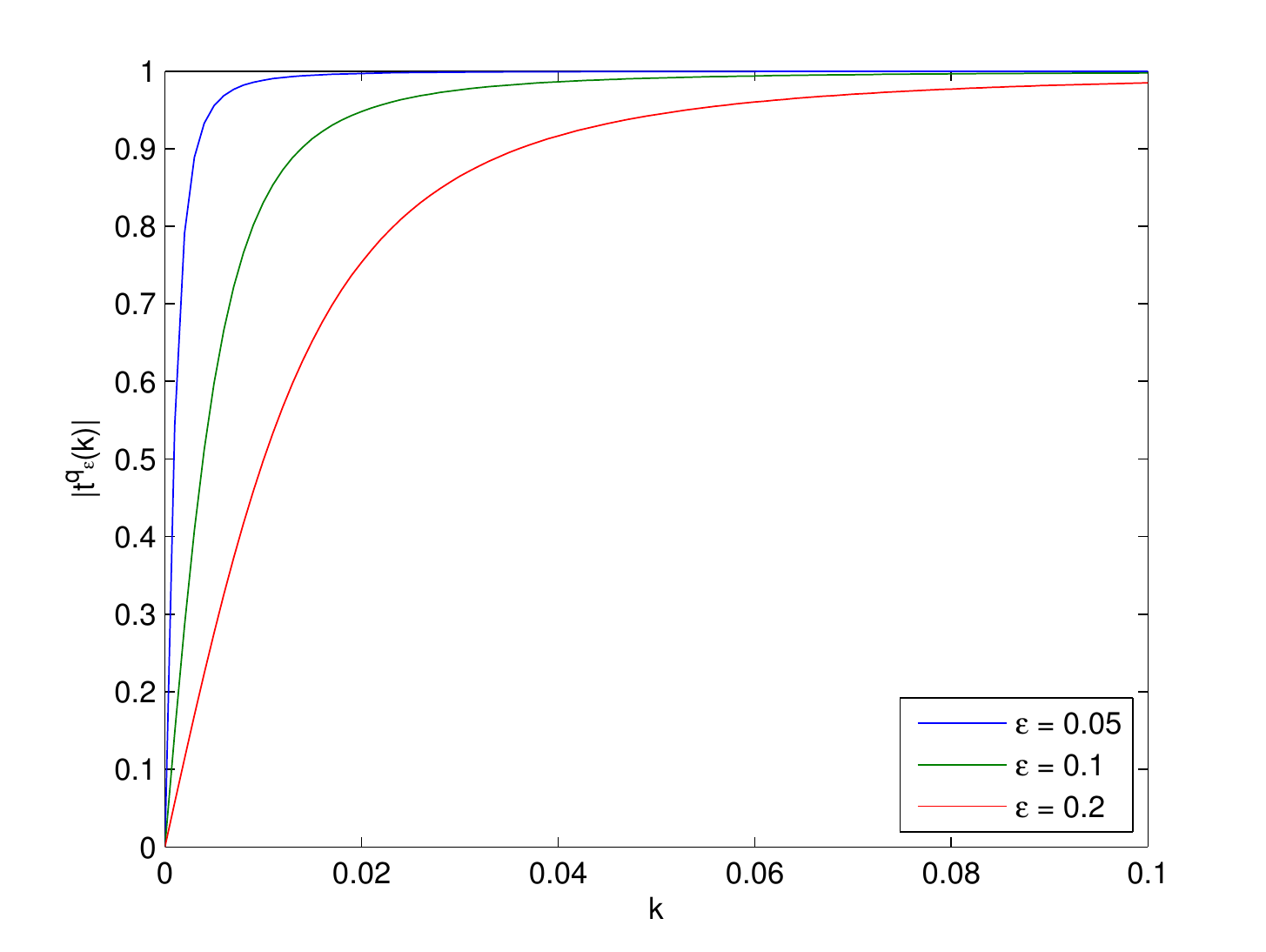}
\label{fig:subfigc}
}
\subfigure[$\big\vert t^{\seff}(k)\big\vert$, $k \in (0;0.1)$, $\eps$ varying]{
\includegraphics[width=0.48\textwidth]{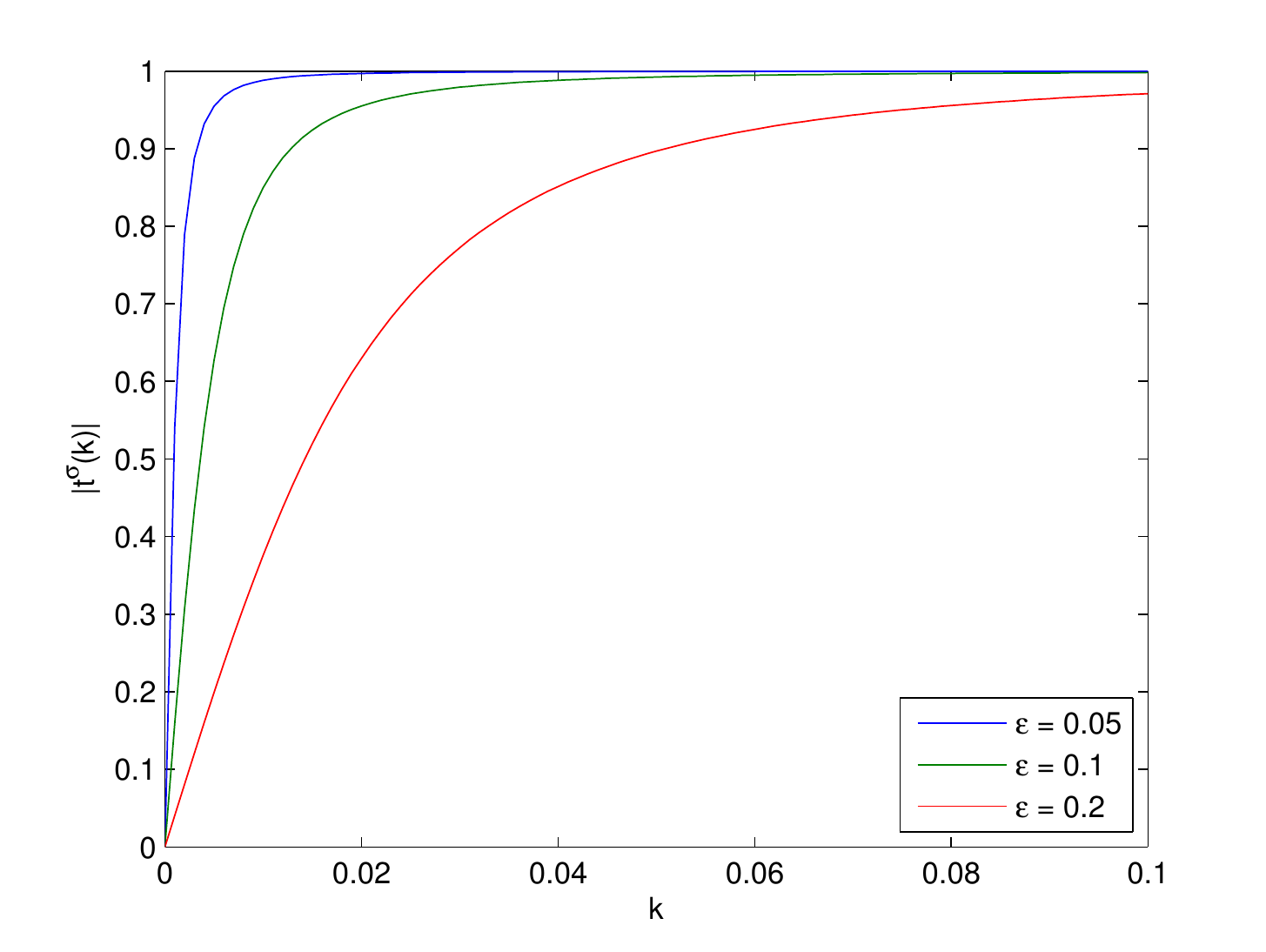}
\label{fig:subfigd}
}
\subfigure[$|t^{V_\eps}(\eps^2\kappa)|$, $\kappa \in (0;2)$, $\eps$ varying]{
\includegraphics[width=0.48\textwidth]{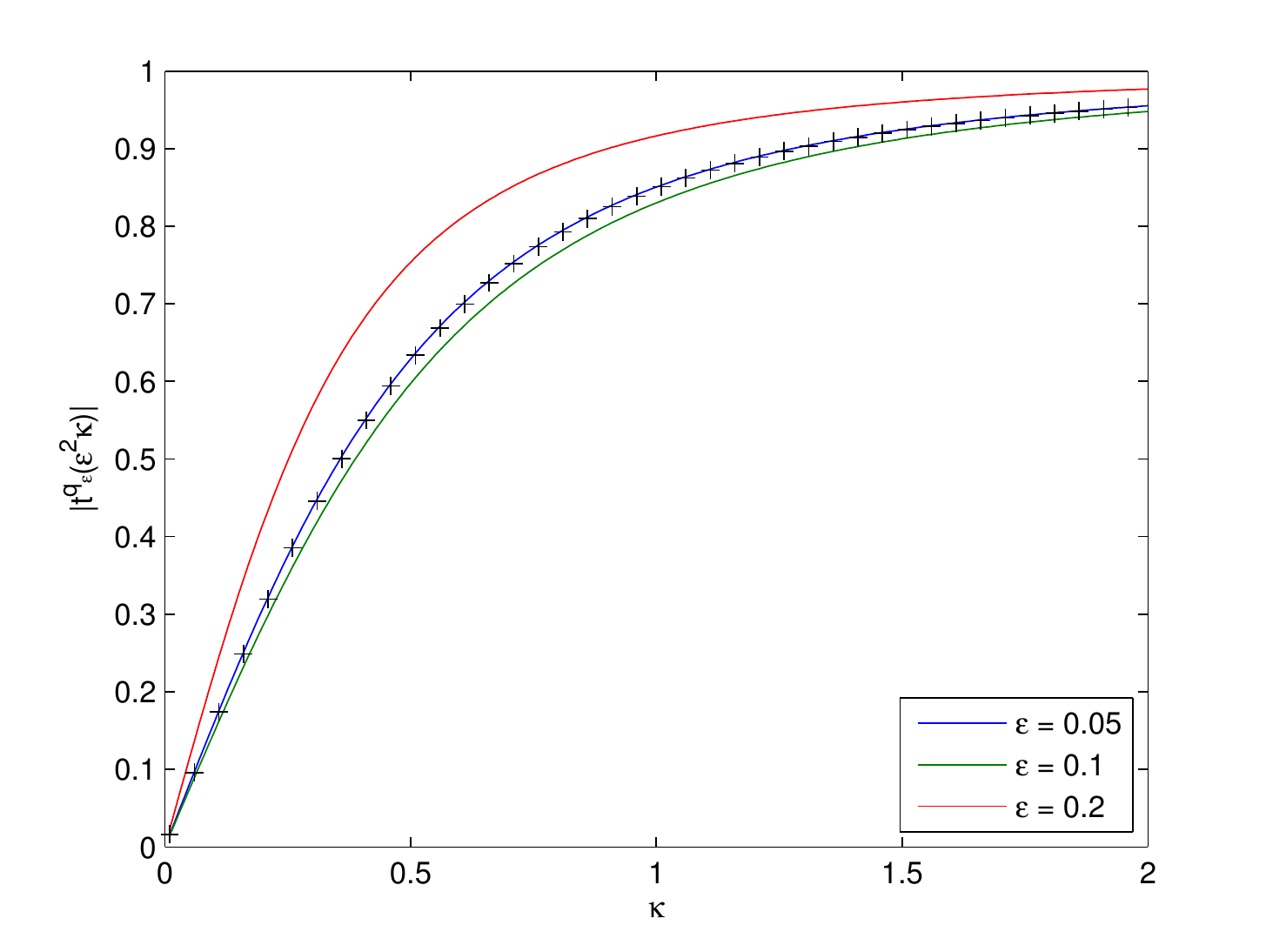}
\label{fig:subfige}
}
\subfigure[$|t^{\seff}(\eps^2\kappa)|$, $\kappa \in (0;2)$, $\eps$ varying]{
\includegraphics[width=0.48\textwidth]{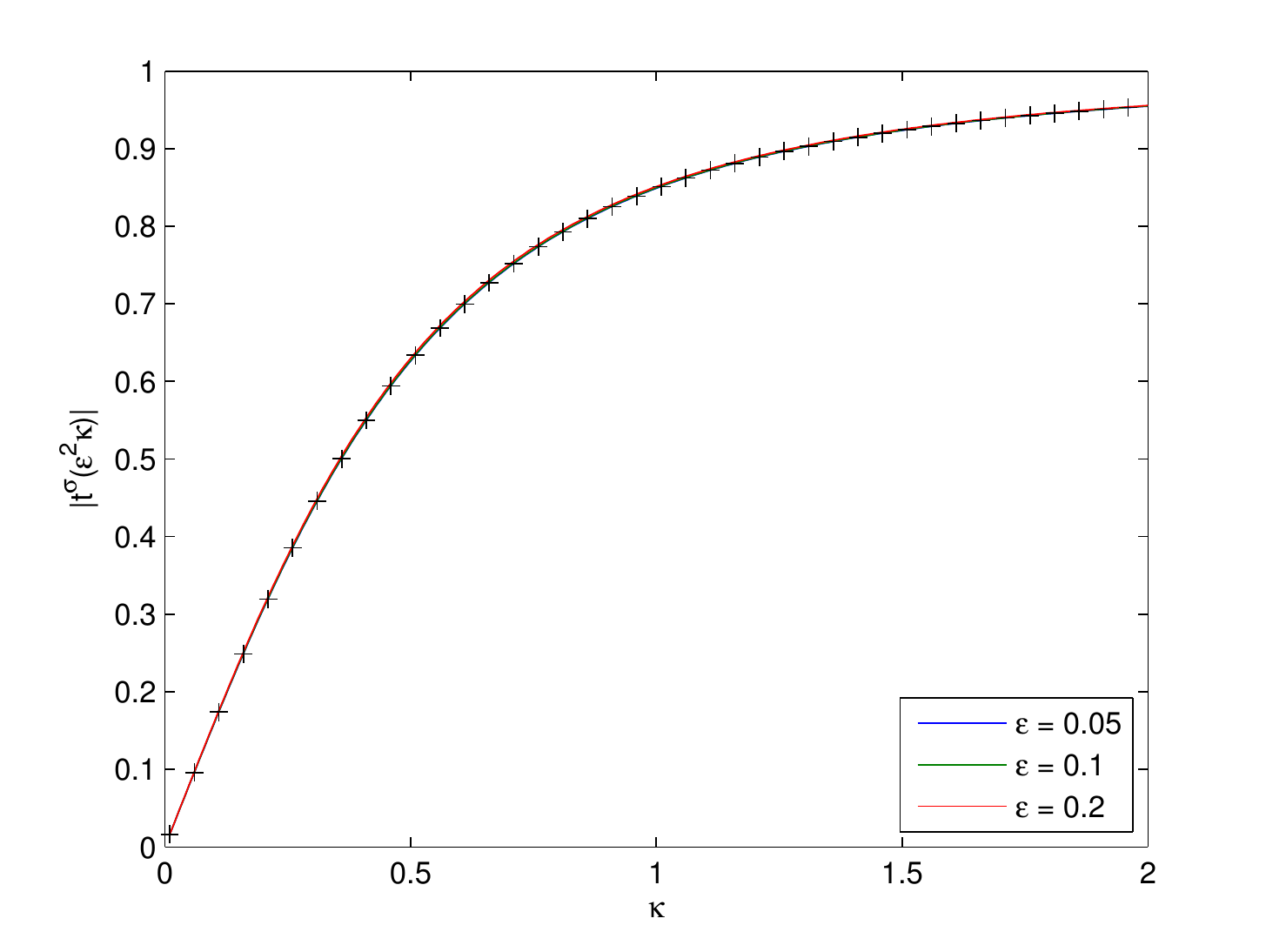}
\label{fig:subfigf}
}
\caption{Plots of potentials $V_\eps(x)$, (a), and the corresponding effective potential $\seff(x)$, (b). Transmission coefficients $t^{V_\eps}(k)$, (c), and $t^{\seff}(k)$, (d). Plots (e) and (f) show convergence of scaled transmission coefficients $t^{V_\eps}(\eps^2\kappa)$ and $t^{\seff}(\eps^2\kappa)$ to the transmission coefficent
$t^{\rm Dirac}(\kappa)=\frac{\kappa}{\kappa-\frac{i}{2}\int\Lambda_{\rm eff}}$ associated with the 
Dirac delta potential well of mass $\int\Lambda_{\rm eff}$. The cross markers in plots (e) and (f) correspond to values of $t^{\rm Dirac}(\kappa)$.}
\label{fig:T-of-k}
\end{figure}
To fix ideas, stick with the case $q_{\rm av}(x)\equiv0$ and thus, $H_{V_\eps}=H_{q_\eps}$, with $q_\eps(x)\equiv q(x,x/\eps)$. We comment below on the case where $q_{\rm av}$ is non-zero. We clarify the nature of low energy scattering by proving that there is an {\it effective potential well}:
\begin{equation}
\seff(x)\ =\ -\eps^2\Lambda_{\rm eff}(x), \label{sigma-eff-1}
\end{equation}
 such that
\begin{equation}
 t^{q_\eps}(k)\ - t^{\seff}(k) \to \ 0\ \ {\rm as}\ \eps\to0,\ {\rm uniformly\ in}\ k\in\RR ;\label{compare-ts1}
\end{equation}
see Theorem~\ref{thm:t-conv-real}, Corollary~\ref{cor:k-real}, and Theorem~\ref{thm:t-conv}, proved by a ``normal form'' type analysis in section~\ref{sec:hard}. Here, $\Lambda_{\rm eff}(x)$ is a positive and localized function defined in terms of the Fourier expansion of the 2-scale potential, $q(x,y)$:
\begin{equation}
\Lambda_{\rm eff}(x)\ =\ -\frac{1}{(2\pi)^2} \sum_{j\neq 0} \frac{|q_j(x)|^2}{\lambda_j^2}.
\label{Lambda-def-aper}
\end{equation}
For the periodic case, $q(x,y+1)=q(x,y)$, $\lambda_j=j,\ j\ne0$ and $\Lambda_{\rm eff}$ is given by:
\begin{equation}
\Lambda_{\rm eff}(x)\ =\ -\frac{1}{(2\pi)^2} \sum_{j\neq 0} \frac{|q_j(x)|^2}{j^2}\ =\ 
- \left\langle-\partial_y^{-2}q(x,y),q(x,y)\right\rangle_{L^2(S^1_y)}\ \ .
\label{Lambda-def}
\end{equation}
This particular choice of effective potential well is anticipated by a formal two-scale homogenization expansion. 
%
%
An example of a mean zero potential $V_\eps(x)=q_\eps(x)=q(x,x/\eps)$ and the associated effective potential is displayed in Figures~\ref{fig:subfiga} and ~\ref{fig:subfigb}. 
%
A clue to the source of non-uniformity in $k$ is offered by a result of Simon~\cite{Simon76}, applied to $\seff$, which implies that for $\eps$ small, the operator $H_\seff$, has a single negative eigenvalue:
\begin{equation} 
 E^{\seff}\ =\ \ -\frac{\eps^4}4 \left(\int_\RR \Lambda_{\rm eff}\ \right)^2\ +\ \mathcal{O}(\eps^6) .
\label{edge-eig1} \end{equation}
Since the eigenvalues of $H_V$ are associated with poles of $t^V(k)$ located on the positive imaginary axis (section~\ref{sec:quicksum}), 
 the eigenvalue $E^\seff$ is associated with a pole at
 \begin{equation}
 k^\seff(\eps)\ = i\frac{\eps^2}{2}\ \left(\ \int_\RR \Lambda_{\rm eff} \ \right)\ 
 +\ \mathcal{O}(\eps^4)
 \label{poleofseff}
 \end{equation}
The estimates of Theorem~\ref{thm:t-conv} and 
Corollary~\ref{cor:compare-t-qav-is-0}, comparing $t^{q_\eps}(k)$ to $t^{\seff}(k)$, in a complex neighborhood of $k=0$ for small $\eps$, enable us to conclude, via Rouch\'e's Theorem, that $t^{q_\eps}(k)$ has a pole $k^{q_\eps}(\eps)\approx k^{\seff}(\eps)$. It follows that $H_{q_\eps}$ has a bound state, $u_{E_{q_\eps}}(x)$, with energy
\begin{equation}
E^{q_\eps}\ =\ -\frac{\eps^4}4 \left(\int_\RR \Lambda_{\rm eff}\ \right)^2\ +\ \mathcal{O}(\eps^5) \ .
\label{small-eig}\end{equation}
 Moreover, $u_{E_{q_\eps}}(x)=\mathcal{O}\left(e^{-\sqrt{|E_{q_\eps}|}\ |x|}\right)$ as $|x|\to\infty$ (Corollary~\ref{cor:edge}).
Furthermore, by Corollary~\ref{cor:transm}, there is a universal scaled limit depending on a single parameter, $\int_\RR\Lambda_{\rm eff}$:
\[
 t^{q_\eps}(\eps^2\kappa)\ \to\ t^\star\left(\kappa;\int_\RR\Lambda_{\rm eff}\right) \ \equiv \ \frac{\kappa}{\ \kappa - \frac{i}{2} \int_\RR\Lambda_{\rm eff} \ }\ 
 {\rm as}\ \ \eps\to0 \quad \text{for } \kappa\ne \frac{i}{2}\ \int_\RR\Lambda_{\rm eff}\ .
\]
  Note that $t^\star\left(\kappa;\int_\RR\Lambda_{\rm eff}\right)$ is the transmission coefficient for the Schr{\"o}dinger operator with a Dirac-distribution potential well of total mass $\int_\RR \Lambda_{\rm eff}>0$:
 \[ H^\star \ \equiv \ -\partial_x^2-\left(\int_\RR \Lambda_{\rm eff}(\zeta) d\zeta\right)\times \delta(x).\]
 Figures~\ref{fig:subfige} and~\ref{fig:subfigf}, as well as Figure~\ref{fig:T-of-kappa}, illustrate this behavior.
 \begin{figure}
 \subfigure[$V_{1,\eps}(x) =\mathbf{1}_{[-1;1]}(x)10e^{\frac{-x^2}{1-x^2}}\cos\big(\frac{2\pi x}{\eps}\big)$, and effective potential, $\sigma_{1,\eps}\equiv -\eps^2\Lambda_1$, $\eps=0.1$]{
\includegraphics[width=0.48\textwidth]{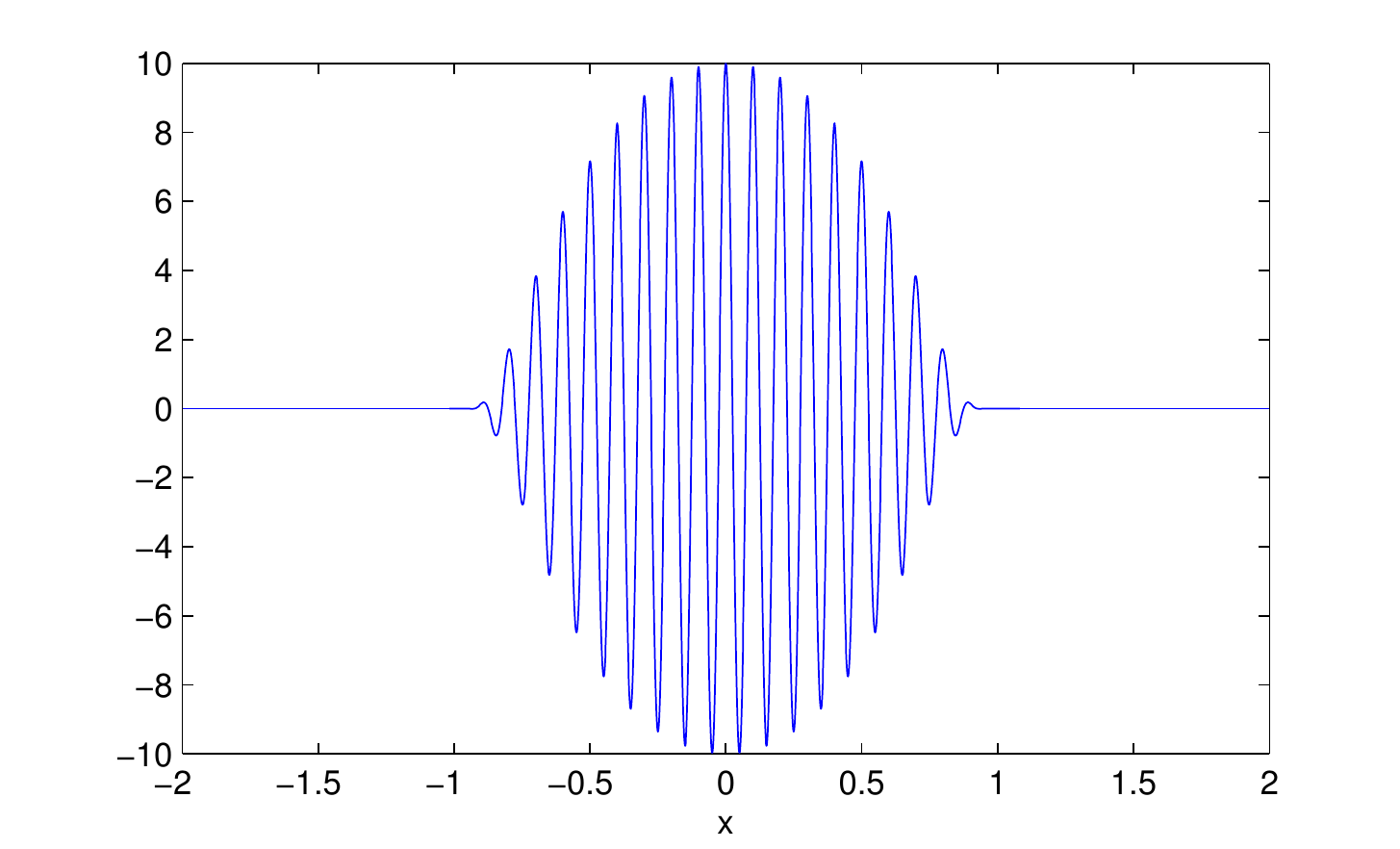}
\includegraphics[width=0.48\textwidth]{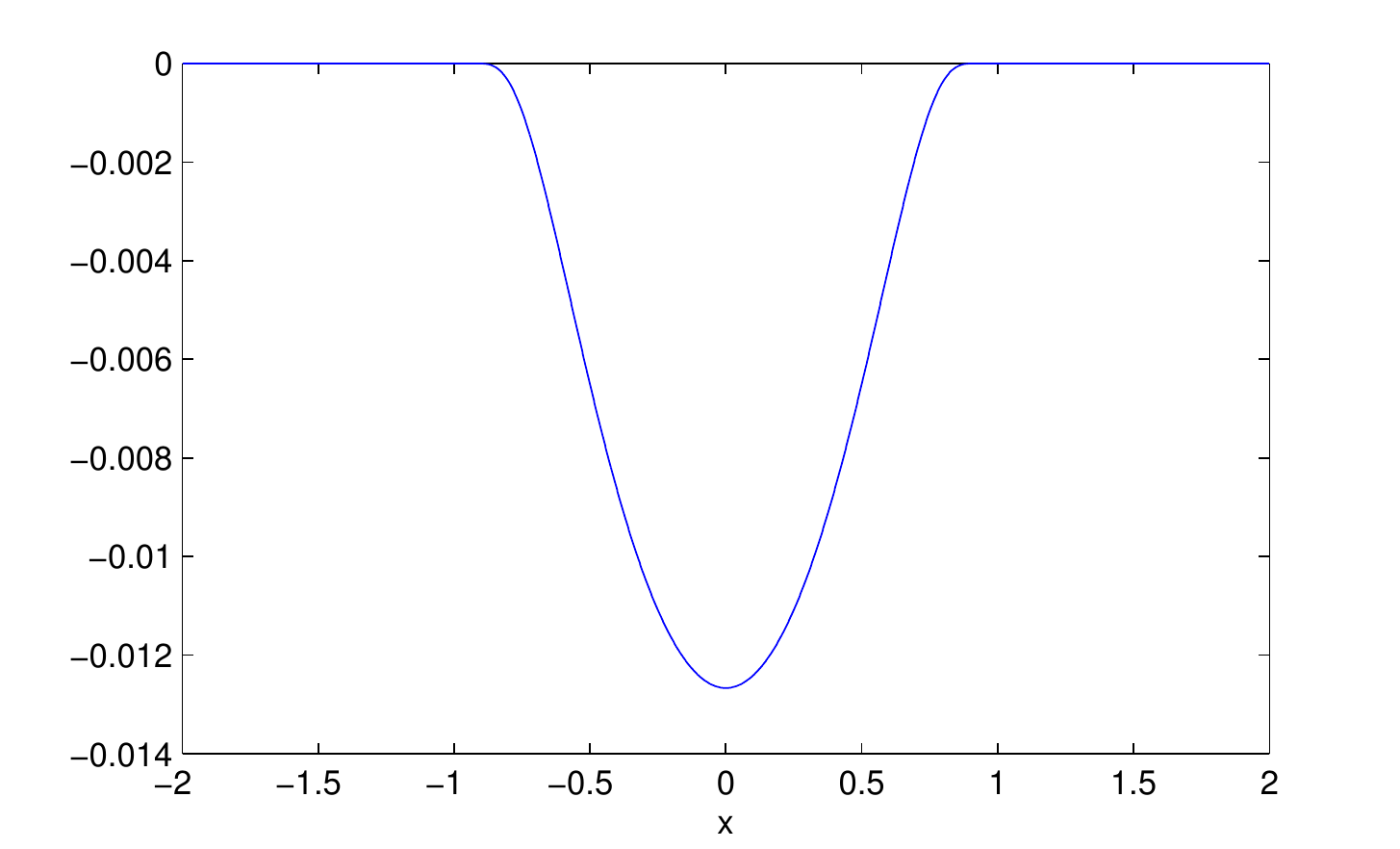}
\label{fig:subfig2a}
}
\subfigure[$V_{2,\eps}(x) \! =\! 10 \Big(\mathbf{1}_{[-1;0]}(x)e^{\frac{-(2x+1)^2}{1-(2x+1)^2}}\! +\! \mathbf{1}_{[0;1]}(x)e^{\frac{-(2x-1)^2}{1-(2x-1)^2}}\Big) \cos\big(\frac{2\pi x}{\eps}\big)$ and effective potential ${\sigma_{2,\eps}\equiv-\eps^2\Lambda_2}$, ${\eps=0.1}$]{
\includegraphics[width=0.48\textwidth]{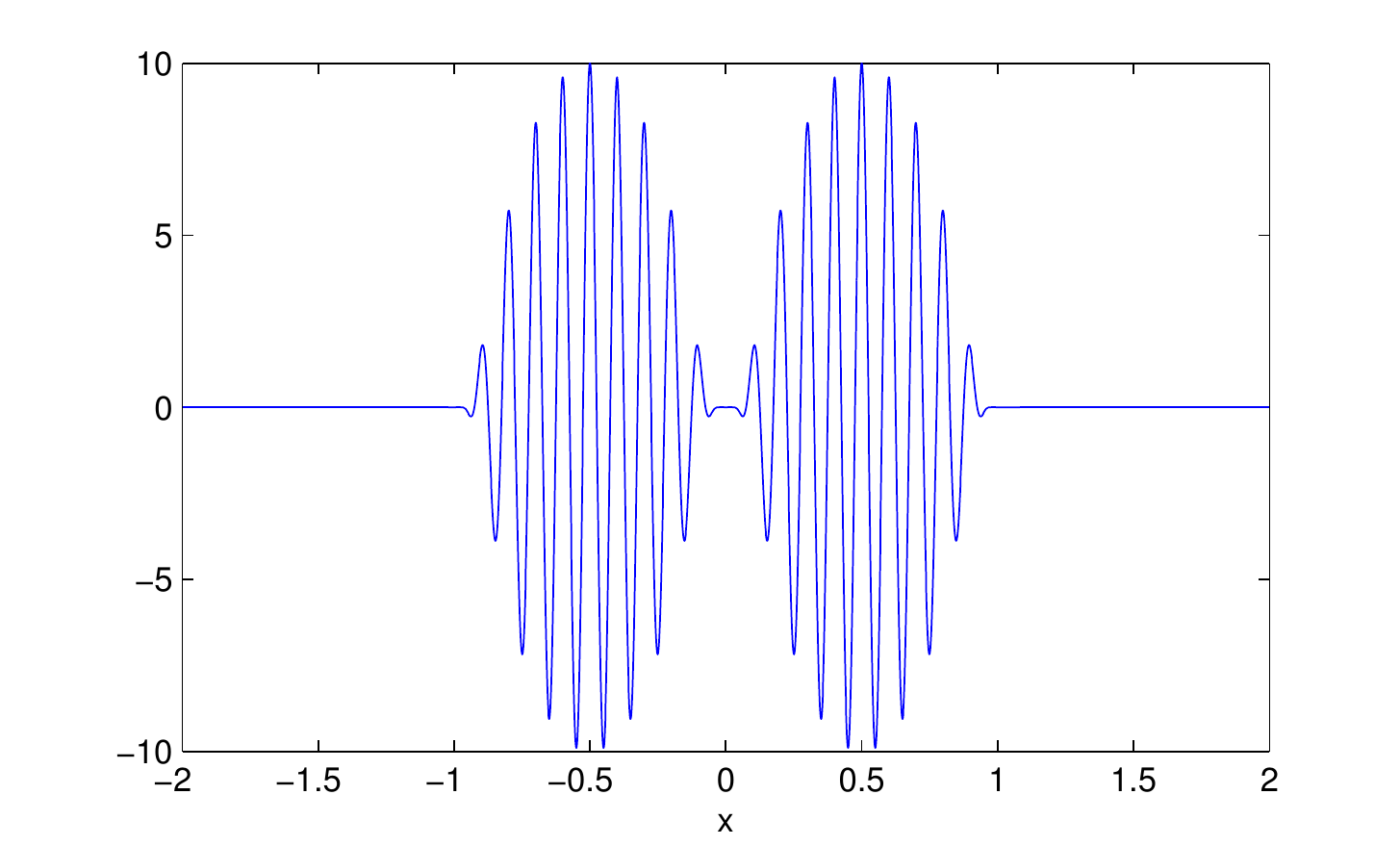}
\includegraphics[width=0.48\textwidth]{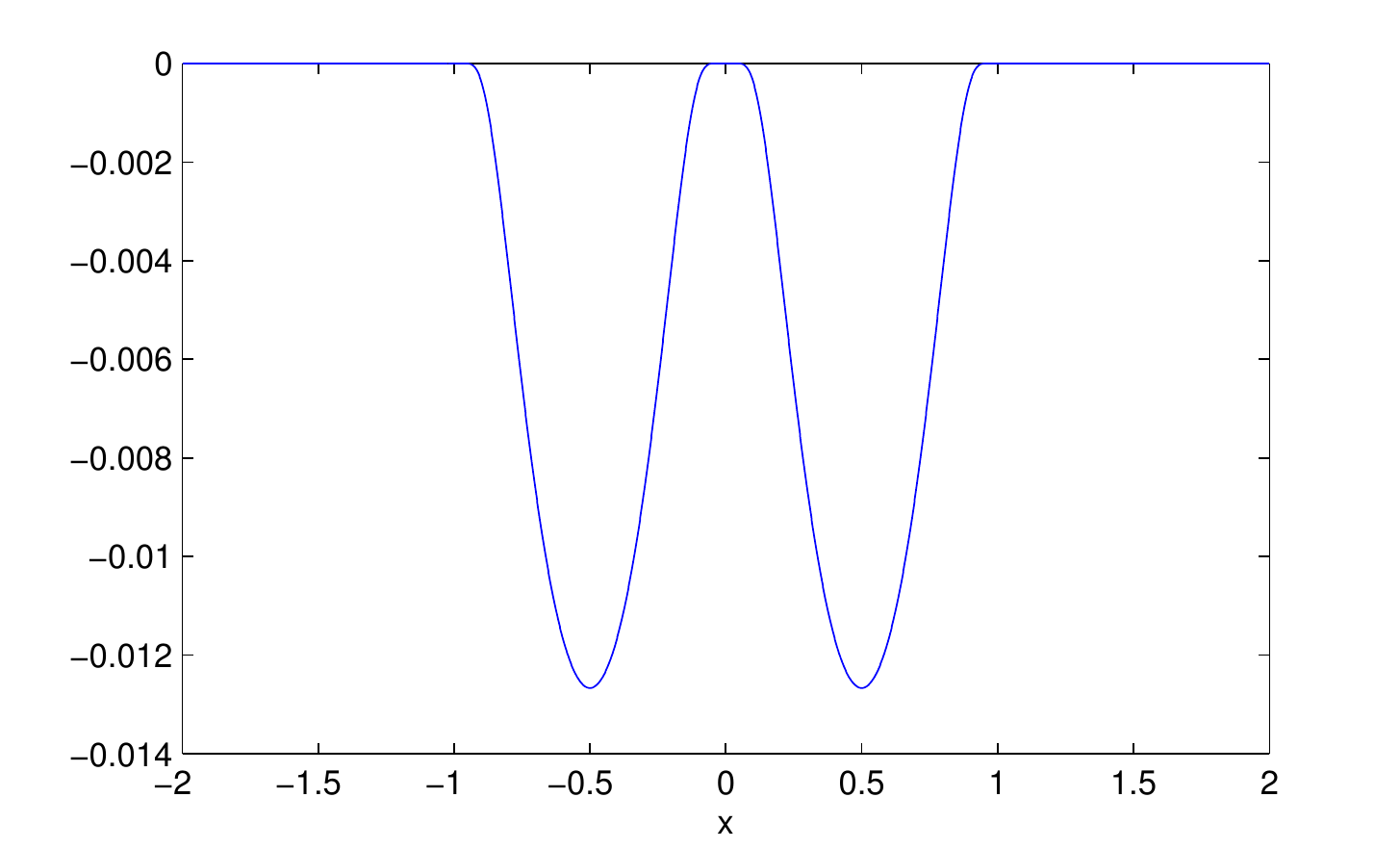}
\label{fig:subfig2b}
}
\subfigure[$|t^{V_{1,\eps}}(\eps^2\kappa)|$, $\kappa \in (0;2)$, $\eps$ varying]{
\includegraphics[width=0.48\textwidth]{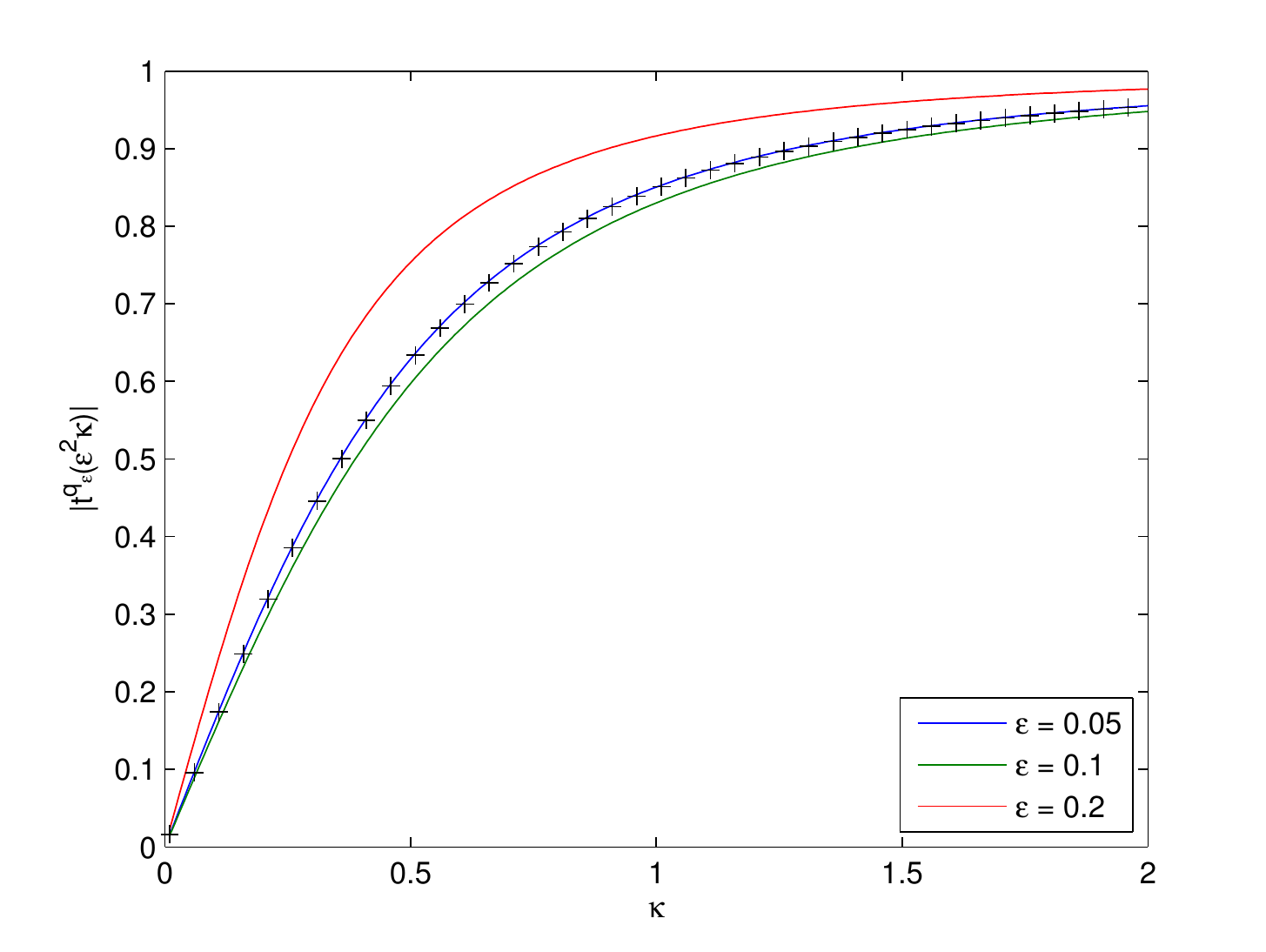}
\label{fig:subfig2c}
}
\subfigure[$|t^{V_{2,\eps}}(\eps^2\kappa)|$, $\kappa \in (0;2)$, $\eps$ varying]{
\includegraphics[width=0.48\textwidth]{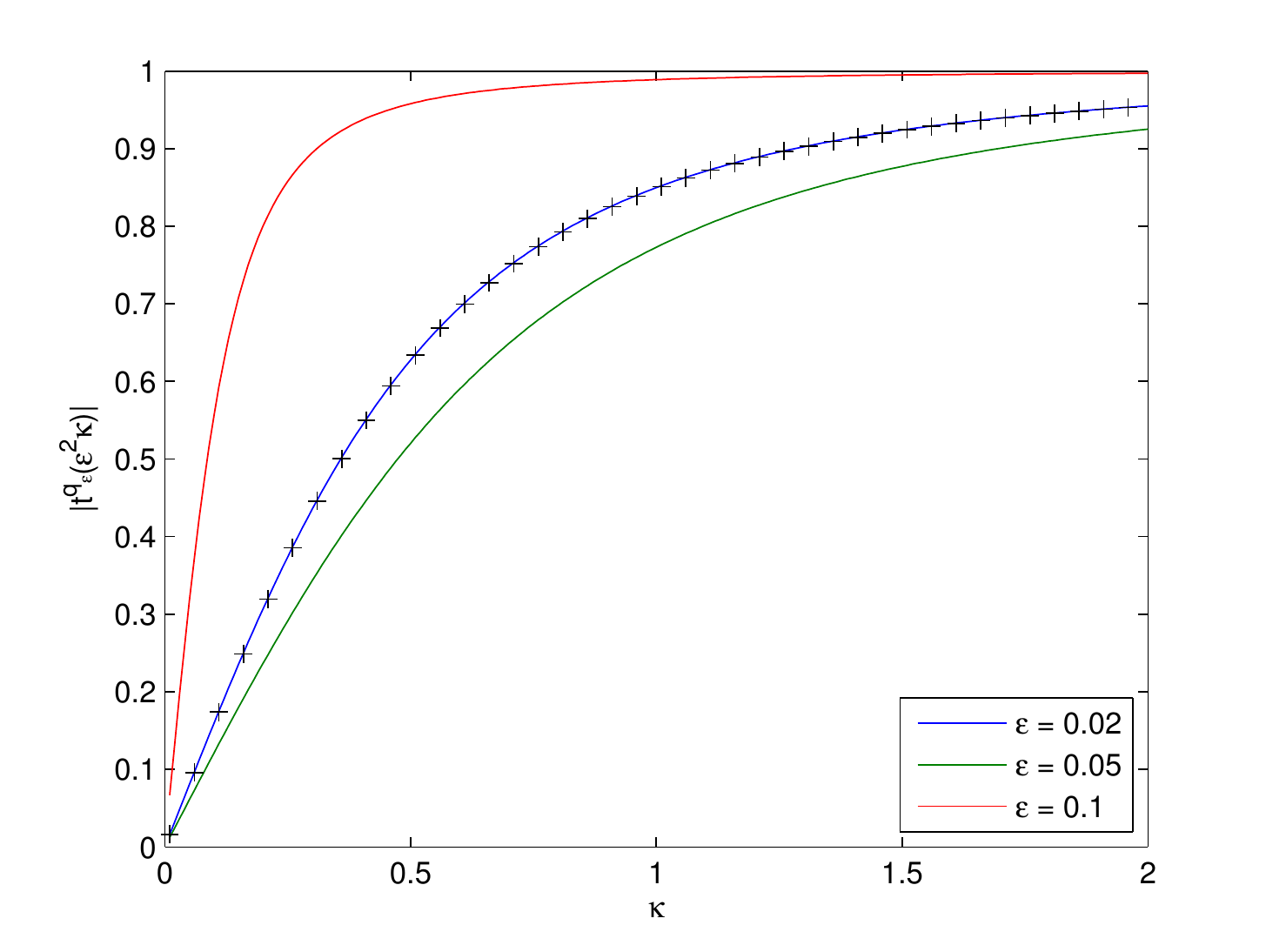}
\label{fig:subfig2d}
}
\caption{Plots (a) and (b) are of two mean zero potentials, $V_{1,\eps}$ and $V_{2,\eps}$ (left), 
 and effective potentials $\sigma^\eps_{1,{\rm eff}}$ and $\sigma^\eps_{2,{\rm eff}}$ (right).
 Potentials chosen so that: $\int\Lambda_{1,{\rm eff}}=\int\Lambda_{2,{\rm eff}}$. Plots (c) and (d) illustrate universality of scaled limits: $t^{V_\eps}(\eps^2\kappa)$ and $t^{\seff}(\eps^2\kappa)$. 
The cross markers correspond to the scaled limit: 
 $t^\star(\kappa)=\frac{\kappa}{\kappa-\frac{i}{2}\int\Lambda_{1,{\rm eff} } }=\frac{\kappa}{\kappa-\frac{i}{2}\int\Lambda_{2,{\rm eff} } }$}
\label{fig:T-of-kappa}
\end{figure} 
 
 A further consequence concerns the large-time dispersive character of solutions to the time-dependent Schr\"odinger equation:
 \begin{equation}\label{eq:schrod0}
i \partial_t \psi \ =\ -\ \partial_x^2 \psi \ + \ q(x,x/\eps)\psi,\ \ \psi(0,x) \ = \ \psi_0\ .
\end{equation}
 We have the following time-decay estimate (Theorem~\ref{cor:schrod}) for sufficiently localized initial conditions, $\psi_0$ , in the continuous spectral part of $H_{q_\eps}$, {\it i.e.}
$ u_{E_{q_\eps}}\perp_{L^2}\psi_0$:
\begin{equation}
 \left(1+|x|^3\right)^{-1}\ \left| \psi(x,t)\ \right|\ \le\ 
 \frac{C}{t^{1/2}}\ \frac{1}{1 + \eps^4 \left(\int_\mathbb{R} \Lambda_{\rm eff}\right)^2 t\ }\ 
 \int_\mathbb{R} \left(1+|\zeta|^3\right) \left|\psi_0(\zeta)\right|\ \dd y \ .
\label{ld-estimate}
\end{equation}
Therefore the effect of the oscillatory perturbation on the rate of dispersion is only seen on the time scale $t\gtrsim\eps^{-4}$.

The above results follow from the non-generic low energy behavior of the average potential $V\equiv0$. Thus we ask:

\noindent {\it Question:\ Are there non-trivial potentials, $V(x)\equiv q_{\rm av}(x)$, with low energy behavior analogous to $V\equiv0$, 
such that $V_\eps = q_{\rm av}(x) + q_\eps(x)$ exhibits similar behavior?}

\noindent The answer is yes! Such examples need to exhibit the behavior 
 \[|t^{q_{\rm av}}(k)|\to |t^{q_{\rm av}}(0)|\ne0\ \ {\rm as}\ \ k\to0.\]
 How such non-generic behavior arises is discussed in section~\ref{sec:refless}. 
 The class of {\it reflectionless potentials}, for which one has $|t(k)|\equiv 1$ for all $k\in\RR$, is a large family of such examples.
 %
%
 Our main Theorem~\ref{thm:t-conv} holds for general $q_{\rm av}$, and shows that the low energy behavior is determined by the effective potential: 
 \begin{equation}
 q_{\rm av}(x)\ +\ \seff(x)\ =\ q_{\rm av}(x)\ -\ \eps^2\Lambda_{\rm eff}(x)\ .
 \nn\end{equation}
Therefore, if $q_{\rm av}$ is a reflectionless potential, then $t^{q_{\rm av}+\seff}(k)$ has a pole, $k^{q_{\rm av}+\seff}(\eps)$, situated on the positive imaginary axis, and of size $\O(\eps^2)$. An application of Rouch\'e's Theorem yields that $t^{q_{\rm av}+q_\eps}(k)$, has a pole near $k^{q_{\rm av}+\seff}(\eps)$ and a bound state
\[ E^{q_{\rm av}+q_\eps}(\eps)\approx\ E^{q_{\rm av}+\seff}(\eps)\ =\ \left[k^{q_{\rm av}+\seff}(\eps)\right]^2\ <0;\ \  \textrm{see Corollary~\ref{cor:compare-t-qav-is-reflectionless}.}\]

 \subsection{Outline of the paper}
 
 In section~\ref{sec:quicksum} we review the prerequisite one-dimensional scattering theory.\ Section~\ref{sec:main-results} contains statements of
 our main results and is structured as follows:
 \begin{enumerate}[(1)]
 \item  Detailed hypotheses on the class of potentials: $V_\eps(x)=q_{\rm av}(x)+q(x,x/\eps)$ are given in Hypotheses {\bf (V)} at the beginning of section~\ref{sec:main-results}.
\item We consider the case where $q_{\rm av}$ is generic and the case where $q_{\rm av}$ is non-generic. As indicated above, the non-generic case, {\it i.e.} $q_{\rm av}\equiv0$, is of greatest interest and we emphasize this case.
\item For non-generic $q_{\rm av}$, Theorem~\ref{thm:t-conv} and Corollary~\ref{cor:compare-t-qav-is-0} give precise estimates on the difference $t^{q_{\rm av}+q_\eps}(k)-t^{q_{\rm av}+\seff}(k)$, for $k$ in a complex neighborhood of zero, and $\eps\to0$. 
 \item For $q_{\rm av}=0$, Corollary~\ref{cor:transm} gives a universal form of the scaled limit of $t^{q_{\rm av}+q_\eps}(\eps^2\kappa)$ as $\eps\to0$. This limit depends on a single parameter, given by the integral of the effective potential.
 \item For $q_{\rm av}=0$, Corollary~\ref{cor:edge} states the potential $q_{\rm av}+q_\eps$, has a bound state with negative energy $\approx\mathcal{O}(\eps^4)$, near the edge of the continuous spectrum.
 \item In subsection~\ref{sec:refless} we present non-trivial (non-indentically zero) potentials, $q_{\rm av}$, which are non-generic, for which the above results for $q_{\rm av}\equiv0$ also apply. We work out the details for ``one-soliton'' potentials $q_{\rm av,\rho}(x)=-2\rho^2{\rm sech}^2(\rho(x-x_0))$, for which $H_{q_{\rm av,\rho}}$ has exactly one negative eigenvalue at $E_0(\rho)=-\rho^2$ and continuous spectrum extending from zero to positive infinity. In this example, our result shows that $H_{q_{\rm av,\rho}+q_\eps}$ has an eigenvalue of order $\mathcal{O}(\eps^4)$, which bifurcates from the edge of the continuous spectrum.
 Specifically,
 \begin{equation}
E^{q_{\rm av}+q_\eps} \ \approx \ -\frac{\eps^4}4 \left(\int_\RR \tanh^2(y)\ \Lambda_{\rm eff}(y)\ \dd y\ \right)^2;
\label{edge-eig-soliton}
\end{equation}
compare with~\eqref{small-eig}.
A second eigenvalue is $\mathcal{O}(\eps^2)$ distant from $E_0(\rho)$.
\item In subsection~\ref{sec:genericcase} we deal with the relatively simple case of highly oscillatory perturbations of a generic potential, $q_{\rm av}$.
 \end{enumerate}
 In section~\ref{sec:k-real}, we combine our precise analysis for bounded $k$ with the relatively simple analysis when $k\in\RR$ is bounded away from zero, and obtain control on the difference $t^{q_\eps}(k)-t^{\seff}(k)$, uniformly for $k\in\RR$.\\ 
 In section~\ref{sec:time-decay} our results on the high and low energy behavior of $t^{q_\eps}(k)$ are used to prove the local energy time-decay estimate~\eqref{ld-estimate}; Theorem~\ref{cor:schrod}.\\
 The proof of Theorem~\ref{thm:t-conv}, and the emergence of the effective potential, $\seff(x)$, are presented in section~\ref{sec:hard}. \
Appendix~\ref{sec:jost-tools} contains detailed estimates on Jost solutions for general localized potentials in an appropriate domain in the complex plane. Appendix~\ref{sec:seff-tools}
 presents a discussion of the potential $q_{\rm av}(x)+\seff(x)=q_{\rm av}(x)-\eps^2\Lambda_{\rm eff}(x)$.
 
 \subsection{Remarks on related work}

\begin{enumerate}[(1)]
\item Detailed and rigorous asymptotic expansions of $t^{q_{\rm av}+q_\eps}(k)$ were
 derived in~\cite{DucheneWeinstein:11} by a method developed in~\cite{GolowichWeinstein05}.
In this work, singular potentials were also admitted. Potentials with singularities, {\it e.g.} jump discontinuities, Dirac delta singularities, give rise to interface-effects which require the inclusion of interface correctors, not captured by standard bulk homogenization theory, in the expansions. For generic potentials \ these expansions hold for any fixed $k\in\RR$ and $\eps\downarrow0$.
\item As discussed, our results are related to those of Simon~\cite{Simon76} on shallow depth potentials with negative or zero average. Our results can be viewed as a generalization to a larger class of perturbations, admitting high-contrast and rapidly oscillatory potentials, {\em i.e} potentials which converge only weakly to their mean.
\item We conjecture, motivated by~\cite{Simon76}, that in dimension 2, there is a discrete eigenvalue which is exponentially small in $\eps$; and that in dimension 3, there exists no bound state for $\eps$ sufficiently small.
\item {\it E. Schr\"odinger meets P. Kapitza:}\ 
There is an interesting connection between our results and
a phenomenon in Mechanics known as the {\it Kapitza Pendulum}. Very generally, this refers to the stabilization of an unstable equilibrium of a dynamical system through time-dependent parametric forcing, {\it i.e.} the stabilization of the classical inverted pendulum 
\cite{Landau-Lifshitz:76,Levi-Weckesser:95}.
\end{enumerate}

\subsection{Notation, norms and function spaces}\label{sec:def}

Various norms are introduced in the analysis of the transmission coefficient, Jost solutions {\it etc}. These norms involve spatial weights of the potential which are algebraic, when we analyze scattering properties 
 for $k\in \mathbb{R}$, and exponential, when we consider these properties for $k\in\mathbb{C}$. Our convention throughout is that spaces with algebraic spatial weights are denoted with calligraphic upper-case letters, {\it e.g.} $\mathcal{W}^{k,p}_{\gamma}$, and spaces with exponential spatial weights are denoted with ordinary upper-case Roman letters, {\it e.g.} $W^{k,p}_\beta$. The parameters $\gamma$ and $\beta$ define the spatial weight. 
 
We denote by $ \mathcal{L}^1_{\gamma} (\RR)$ the space of measurable functions $g$ such that
\[
\big| g \big|_{\mathcal{L}^1_{\gamma}} \ = \ \int_\RR | g (x) | (1 + |x|)^\gamma \dd x\ <\ \infty .
\]
The space of functions, $g$, whose derivatives up to order $n$ are in $ \mathcal{L}^1_{\gamma} $ is denoted $\mathcal{W}^{n,1}_{\gamma}$ and the associated norm is
\[
\big| g \big|_{\mathcal{W}^{n,1}_{\gamma}} \ \equiv \ \sum_{l=0}^n \big| \partial^l g\big|_{\mathcal{L}^1_{\gamma}}.
\]

For a fixed $\beta>0$, we denote by $L^{\infty}_\beta$ the space of measurable functions $g$ defined on $\RR$ such that 
\[
\big| g \big|_{L^{\infty}_{\beta}} \ \equiv \ \big|e^{\beta \cdot} g\big|_{L^\infty}\ \equiv\
 {\rm ess\ sup}_{x\in\RR}\ e^{\beta x} \left| g(x) \right| \ < \ \infty.
\]
$W^{n,\infty}_{\beta}$ denotes the space of the functions $g$ defined on $\RR$, whose derivatives up to order $n$ are in $L^{\infty}_{\beta}$ with associated norm 
\[
\big| g \big|_{W^{n,\infty}_{\beta}} \ \equiv \ \sum_{l=0}^n \big| \partial^l g\big|_{L^\infty_{\beta}}.
\]
For a function, $V$, of the form
\[
V(x,y)\ = \ q_{\rm av}(x)+ q(x,y) \ = \ q_{\rm av}(x) \ + \ \sum_{j\in\mathbb{Z}\setminus\{0\}}\ q_j(x)\ e^{2\pi i\lambda_j y},
\]
we introduce the following norms:
\begin{align*}
& \textrm{exponentially weighted:}\ & \big\bracevert V \big\bracevert \ \ \, &\equiv \ \big|q_{\rm av}\big|_{W^{1,\infty}_{\beta}} \ + \ \sum_{j\in\mathbb{Z}\setminus\{0\}} \big|q_j\big|_{W^{3,\infty}_{\beta}} \ ;\\
& \textrm{algebraically weighted:}\ & \big\bracevert\hspace{-6pt}\big\bracevert V \big\bracevert\hspace{-6pt}\big\bracevert \ & \equiv \ \big|q_{\rm av}\big|_{\mathcal{W}^{1,1}_2} \ + \ \sum_{j\in\mathbb{Z}\setminus\{0\}} \big|q_j\big|_{\mathcal{W}^{3,1}_3}\ .
\end{align*}

\section{Review of 1d scattering theory}\label{sec:quicksum}

In this section we briefly review some of the basics of scattering theory
 for the one-dimensional Schr\"odinger equation:
 \begin{equation}
\left(-\frac{d^2}{dx^2} \ + \ V(x) \ - \ k^2\right) \ u(x;k) \ = \ 0,
\label{schrod-sigma}
\end{equation}
for localized potentials, $V(x)$, assumed to satisfy 
\[V\in \mathcal{L}^1_{2}(\RR)\ =\ \{V: (1+|x|)^2 V(x)\in L^1(\RR)\} .\]
 In particular, in section~\ref{sec:frt} we discuss the Jost solutions, $f^V_\pm(x;k)$, and the reflection and transmission coefficients, $r^V_\pm(k)$ and $t^V(k)$.
An extensive discussion can be found in~\cite{DeiftTrubowitz79},~\cite{ReedSimonIII},~\cite{Newton86}. 
Section~\ref{sec:generic} explains what is meant by a {\it generic potential}.
 Finally, in section~\ref{sec:volterra} we introduce some important tools enabling us to compare the transmission coefficients of two different potentials. This is based on the Volterra integral equation for the Jost solution for a potential, $V$, viewed as a perturbation of a second potential, $W$.

\subsection{The Jost solutions, and reflection and transmission coefficients}\label{sec:frt}

For $k\in\RR$, introduce $f^{V}_\pm(x;k)$, the unique solutions of~\eqref{schrod-sigma} 
with
\begin{equation}
 f^{V}_\pm(x;k) \ \sim \ e^{\pm ikx}, \quad {\rm as}\ \ \ x\to \pm \infty.\label{jost-bc}
 \end{equation}

Observe from the asymptotics as $x\to\infty$, we have $\mathcal{W}[f^V_+(\cdot;k),f^V_+(\cdot;-k)]=2ik$, where $\mathcal{W}[h_1,h_2]$ denotes the Wronskian of functions $h_1(x)$ and $h_2(x)$:
\begin{equation}
\mathcal{W}[h_1,h_2]= h_1(x)h_2'(x)-h_2(x)h_1'(x).
\label{wronskian}
\end{equation}
Therefore, for $k\in\RR\setminus\{0\}$, the set $\{f^V_+(x;k),f^V_+(x;-k)\}$ is a linearly independent set of solutions of~\eqref{schrod-sigma}.
 
The transmission coefficients, $t^V_\pm(k)$, and the reflection coefficients $r_\pm^V(k)$ are defined via the algebraic relations, among the Jost solutions $f^V_\pm(x;k)$:
\begin{align}
\label{eq:defrt+} f^V_+(x;k) \ &\equiv \ \frac{r^V_+(k)}{t^V_+(k)}f^V_-(x;k)\ +\ \frac{1}{t^V_+(k)}f^V_-(x;-k),\\
\label{eq:defrt-} f^V_-(x;k) \ &\equiv \ \frac{r^V_-(k)}{t^V_-(k)}f^V_+(x;k)\ +\ \frac{1}{t^V_-(k)}f^V_+(x;-k).
\end{align}
One can check that $\mathcal{W}[f_+^V,f_-^V] \ = \ -2ik[t^V_-(k)]^{-1} \ = \ -2ik[t^V_+(k)]^{-1}$, 
and therefore we write
\begin{equation}\label{eq:Wronskian-vs-transmission}
\mathcal{W}[f_+^V,f_-^V] \ = \ -\frac{2ik}{t^V(k)}\ ,
\end{equation}
with $t^V(k) \ \equiv \ t^V_-(k) \ = \ t^V_+(k)$. Furthermore, one has
\begin{equation}
\left| t^V(k) \right|^2\ +\ \left| r^V_\pm(k) \right|^2\ =\ 1,\ \ k\in\RR.
\label{tr-energy}
\end{equation}

The Jost solutions, $f_\pm^V$,  and scattering coefficients, $t^V$ and $r^V_\pm$, can be analytically extended into the upper-half complex $k-$plane. Note that if $k_1$ is a pole of $t^V(k)$, with $\Im(k_1)>0$, then ${\mathcal W}[f_+^V,f_-^V](k_1)=0$.
In this case, $f^V_+(x;k_1)$ and $f^V_-(x;k_1)$ are proportional and therefore decay exponentially as $x\to\pm\infty$. Thus, $k_1^2$ is an $L^2-$eigenvalue of $H_V$. 

If the potential $V(x)$ is exponentially decaying as $x$ tends to infinity, then the Jost solutions can be analytically extended into the lower half complex $k-$plane. More precisely,
 if $V\in L^\infty_\beta$ (see Section~\ref{sec:def}), then $f^V_\pm(x;k)$ are defined for $\Im(k)>-\beta/2$ as the unique solutions of the Volterra integral equations
\begin{align}\label{eq:Volterra}
f^{V}_+(x;k) \ &= \ e^{ikx} \ + \ \int_x^\infty \frac{\sin(k(y-x))}{k} V(y) f^{V}_+(y;k) \dd y, \nn \\
f^{V}_-(x;k) \ &= \ e^{-ikx} \ - \ \int_{-\infty}^x \frac{\sin(k(y-x))}{k} V(y) f^{V}_-(y;k) \dd y.
\end{align}
Detailed bounds on $f^V_\pm(x;k)$ and their derivatives are presented in Appendix~\ref{sec:jost-tools}. 

\medskip

Finally, note the following consequences of $V(x)$ being real-valued, the uniqueness of the Jost solutions as defined above, and~\eqref{eq:defrt+}--\eqref{eq:defrt-}:
\begin{equation}\label{reality-condition}
f^{V}_\pm(x;-\overline{k}) \ = \ \overline{f^{V}_\pm(x;k)}, \qquad t^{V}(-\overline{k}) \ = \ \overline{t^{V}(k)}, \qquad r^{V}_\pm(-\overline{k}) \ = \ \overline{r^{V}_\pm(k)}.
\end{equation}
In particular, $f^V_\pm(x;0)$, $t^{V}(0)$, $r^{V}_\pm(0)$ are real. 

\subsection{Generic and non-generic potentials}\label{sec:generic}

Using the decay hypotheses of potential $V\in L^\infty_\beta$ and the method of~\cite{DeiftTrubowitz79}, page~145, one can check that the transmission and reflection coefficients are well-defined by~\eqref{eq:defrt+}--\eqref{eq:defrt-} for $|\Im(k)|<\beta/2$, and satisfy the important relations, which follow from~\eqref{eq:Wronskian-vs-transmission}
 and~\eqref{eq:Volterra}:
\[
\frac{1}{t^V(k)} \ = \ 1 \ - \ \frac{1}{2ik}I^V(k), \qquad \text{ thus} \quad
\mathcal{W}[f_+^V,f_-^V](k)\ =\  -2ik\ +\ I^V(k),\]
where  $ I^V(k) \ \equiv \ \int_{-\infty}^\infty V(y)e^{-iky}f^V_+(y;k)\dd y$.
Equivalently, one has
\begin{equation}
t^V(k) \ =\ -\frac{2ik}{\mathcal{W}[f_+^V,f_-^V](k)}\ =\ \frac{2ik}{2ik-I^V(k)}.
\label{tsig-Isig}
\end{equation}

Recall that if $V(x)\equiv 0$, then $t^V(k)\equiv 1$. Moreover, if 
\begin{equation}
\label{generic-def}
 I^V(0) \ = \ \mathcal{W}[f_+^V,f_-^V](0)\ =\ \int_{-\infty}^\infty V(y)f^V_+(y;0)\dd y\ne0, \
 \end{equation}
 then by continuity of $t^V(k)$ and~\eqref{tsig-Isig}, one has 
 \begin{equation}
 t^V(0) = \lim_{k\to0} t^V(k)\ =\ 0.
 \label{lim-t}\end{equation}
The case where~\eqref{generic-def} and therefore~\eqref{lim-t} holds is typical. Indeed, 
 it has been shown in Appendix~2 of~\cite{Weder00} that for a dense subset of $\mathcal{L}^1_{1}$, one has $I^V(0)\ne0$; see also~\cite{DeiftTrubowitz79} and~\cite{Newton86}. Thus we say that~\eqref{generic-def} and~\eqref{lim-t} holds {\it generically in the space of potentials}.
 \begin{Definition}[Generic potentials] \label{def:generic}
 We say that a potential, $V$, is {\em generic} if one has $ t^V(0) =0$. Equivalently, $V$ is generic if and only if
 \[ \frac{k}{t^V(k)} \longrightarrow \frac{I^V(0)}{2i} \neq 0,\quad \text{ as } k\to 0.\] 
 \end{Definition}
 Note that in the non-generic case, where $\mathcal{W}[f_+^V,f_-^V](0)=0$, we have that Jost solutions $f^V_\pm(x;k)$ satisfy
 $f^V_\pm(x;0)\sim1$ as $x\to\pm\infty$ and are multiples of one another. Thus, non-genericity is equivalent to the existence of a globally bounded solution of the Schr\"odinger equation at zero energy. Such states are sometimes referred to as zero energy resonances. The simplest example is $V\equiv0$
 where $f^0_\pm(x;k)=e^{\pm i kx}$ and $f^0_\pm(x;0)\equiv 1$.

\subsection{Relations between $f_\pm^V$ and $f_\pm^W$ for general $V$ and $W$}
\label{sec:volterra}
Our approach is based on associating with $V_\eps(x)=q_{\rm av}(x)+q_\eps(x)$ a more accurate (than $q_{\rm av}$) 
 minimal model or {\it normal form}, $V_{\eps,{\rm eff}}(x)=q_{\rm av}(x)+\seff(x)$, of the asymptotic scattering properties for $k$ bounded and $\eps\to0$. An important tool will then be to compare the Jost solutions associated with the potential, $V=V_\eps$, with those of some family of potentials, $W=q_{\rm av}+\sigma$, parametrized by $\sigma$, which is to be determined. This section develops the necessary tools for this comparison.

In the Volterra equation~\eqref{eq:Volterra} we write $f^V_\pm(x;k)$ as a perturbation of the states $e^{\pm ikx}$, which lie in the kernel of $-\partial_x^2-k^2$. In the following proposition, we generalize this formula by viewing $f^V_\pm(x;k)$ as a perturbation of the Jost solutions $f^{W}_\pm(x;k)$ for the problem:
\[ \left(-\frac{d^2}{dx^2} \ + \ W \ - \ k^2\right) \ u \ = \ 0.\]

\begin{Proposition}\label{prop:q2s}
Let $V,W\in L^\infty_\beta$ and and let $f^{V}_\pm,f^{W}_\pm$ denote the associated Jost solutions. Then for $|\Im(k)|<\beta/2$, one has
\begin{align}
f^{V}_+(x;k) \ &= \ \alpha_+[V,W] \ f^W_+(x;k) \ + \ \beta_+[V,W] \ f^W_-(x;k) \nn \\
f^{V}_-(x;k) \ &= \ \alpha_-[V,W] \ f^W_+(x;k) \ + \ \beta_-[V,W] \ f^W_-(x;k),\label{VW-relate}
\end{align}
with $\alpha_\pm[V,W](x;k)$ and $\beta_\pm[V,W](x;k)$ defined by
\begin{align}
\alpha_+[V,W] \ &\equiv \ 1 + \int_x^\infty \frac{f^W_- (V -W) f^{V}_+}{\mathcal{W}[f^W_+ ,f^W_-]} \dd y , & 
\beta_+[V,W] \ &\equiv \ -\int_x^\infty \frac{f^W_+ (V -W) f^{V}_+}{\mathcal{W}[f^W_+ ,f^W_-]} \dd y,\label{eq:alphabeta+} \\
\alpha_-[V,W] \ &\equiv \ -\int_{-\infty}^x \frac{f^W_- (V -W) f^{V}_-}{\mathcal{W}[f^W_+ ,f^W_-]} \dd y, & 
\beta_-[V,W] \ &\equiv \ 1 + \int_{-\infty}^x \frac{f^W_+ (V -W) f^{V}_-}{\mathcal{W}[f^W_+ ,f^W_-]} \dd y. \label{eq:alphabeta-}
\end{align}
Equivalently, one has the Volterra equation
\begin{align}\label{eq:Volterra1}
f^{V}_+(x;k)  &=  f^W_+(x;k)  +  \int_x^\infty \frac{f^W_+(x;k)f^W_-(y;k)-f^W_-(x;k)f^W_+(y;k)}{\mathcal{W}[f^W_+ ,f^W_-]} (V -W) f^{V}_+(y;k) \dd y, \\
f^{V}_-(x;k)  &=  f^W_-(x;k)  -  \int_{-\infty}^x \frac{f^W_+(x;k)f^W_-(y;k)-f^W_-(x;k)f^W_+(y;k)}{\mathcal{W}[f^W_+ ,f^W_-]} (V -W) f^{V}_-(y;k) \dd y. \nn
\end{align}
\end{Proposition}

A very useful consequence is:
\begin{Corollary}\label{cor:relateW} Let $V,W\in L^\infty_\beta$ and and let $f^{V}_\pm,f^{W}_\pm$ denote their respective associated Jost solutions. Then for $|\Im(k)|<\beta/2$, one has
\begin{equation}
\mathcal{W}[f^{V}_+ ,f^{V}_-](k) \ = \ \mathcal{M}[V,W](k) \ \mathcal{W}[f^W_+ ,f^W_-](k)\ ,
\label{relateW}
\end{equation}
where $\mathcal{M}[V,W](x;k)$ is constant in $x$, and given by
\begin{equation}
\mathcal{M}[V,W](k) \ \equiv \ \alpha_+[V,W](x;k)\beta_-[V,W](x;k) \ - \ \alpha_-[V,W](x;k)\beta_+[V,W](x;k).
\label{Mdef}
\end{equation}
By~\eqref{eq:Wronskian-vs-transmission}, and taking the limit as $x\to -\infty$ of~\eqref{eq:alphabeta+} and~\eqref{eq:alphabeta-} in~\eqref{Mdef}, one has 
\begin{equation}
\frac{k}{t^V(k)}  =  \frac{k}{t^W(k)}  -  \frac{ I^{[V,W]}(k)}{2i},\  \text{with } \ I^{[V,W]}(k) \equiv \int_{-\infty}^\infty f^W_-(y;k)(V-W)(y)f^V_+(y;k)\dd y.
\label{IVW-def}
\end{equation}
\begin{Remark}
The relation~\eqref{IVW-def}, applied for $V=V_\eps$ and a judicious choice of $W$, is the point of departure for the proofs of our main results.
\end{Remark}
\end{Corollary}
\begin{proof}[Proof of Corollary~\ref{cor:relateW}.]
Equation~\eqref{relateW} follows from substituting the expressions \eqref{VW-relate} into the definition of $\mathcal{W}[f^{V}_+ ,f^{V}_-]$, and using that $\alpha_+[V,W],\beta_+[V,W]$ satisfy the identity:\ $ (\alpha_\pm)' f^W_+ \ + \ (\beta_\pm)' f^W_- \ = \ 0$; see~\eqref{sys-alphabeta} below.

To prove~\eqref{IVW-def}, we begin by making use of relation~\eqref{eq:Wronskian-vs-transmission}. One has
\[
\frac{k}{t^V(k)} \ = \ -\frac{\mathcal{W}[f^{V}_+ ,f^{V}_-](k) }{2i} \]
We next relate $\mathcal{W}[f^{V}_+ ,f^{V}_-]$ to $\mathcal{W}[f^{W}_+ ,f^{W}_-]$ by 
substitution of the expressions~\eqref{VW-relate} into the definition of $\mathcal{W}[f^{V}_+ ,f^{V}_-]$ 
 and using~\eqref{eq:alphabeta+} and~\eqref{eq:alphabeta-} to obtain
\[ \frac{k}{t^V(k)} = \
 -\mathcal{M}[V,W](x,k)\frac{ \mathcal{W}[f^W_+ ,f^W_-](k)}{2i} \ = \ \mathcal{M}[V,W](x,k)\frac{k}{t^W(k)}.\]

Now, since $V,W\in L^\infty_\beta$, the estimates of Lemma~\ref{Lem:f-gen} yield
\[\lim\limits_{x\to -\infty} \beta_+[V,W](x) \ < \ \infty, \quad
\lim\limits_{x\to -\infty} \alpha_-[V,W](x) \ = \ 0 \quad \text{and} \quad 
\lim\limits_{x\to -\infty} \beta_-[V,W](x) \ = \ 1. \]
Therefore,
\[ \mathcal{M}[V,W](k)\ =\ \lim_{x\to-\infty}\alpha_+[V,W](x).\]

Therefore, one deduces from Proposition~\ref{prop:q2s} that
\[\frac{k}{t^V(k)} \ = \ \frac{k}{t^W(k)}\lim\limits_{x\to -\infty}\alpha_+[V,W] \ = \ \frac{k}{t^W(k)}\left(1+ \frac{ I^{[V,W]}(k)}{\mathcal{W}[f^W_+ ,f^W_-](k)}\right) \ = \ \frac{k}{t^W(k)} \ - \ \frac{ I^{[V,W]}(k)}{2i},\]
where $I^{[V,W]}(k)$ is given in~\eqref{IVW-def}. 
The proof of Corollary~\ref{cor:relateW} is complete. \end{proof}

\medskip
\begin{proof}[Proof of Proposition~\ref{prop:q2s}.]
The integral equation governing a Jost solution for the potential $V$ may be written relative to the potential $W$ as follows. Start with the equation for $u_\pm=f^V_\pm$ written in the form:
\begin{equation}
 (H_W-k^2)\ u\ =\ \left(-\frac{d^2}{dx^2} \ + \ W \ - \ k^2\right) \ u \ = \ (W-V)u.\label{pert}\end{equation}
Treating the right hand side of~\eqref{pert} as an inhomogeneous term, we now derive an equivalent integral equations for the Jost solutions . Thus, we seek solutions $u_\pm$ of~\eqref{pert}, such that 
$u_\pm(x;k)\ \ \sim\ \ f_\pm^V(x;k),\ \ x\to\pm\infty$ 
of the form
\[u(x,k) \ \equiv\ \alpha(x,k) f^W_+(x,k) \ + \ \beta(x,k) f^W_-(x,k),\ \ \text{with}\ \ \alpha' f^W_+ \ + \ \beta' f^W_-\ = \ 0.\]
 We obtain
$ u' \ = \ \alpha {f^W_+}' \ + \ \beta {f^W_-}',\ u'' \ = 
 \alpha' {f^W_+}' \ + \ \beta' {f^W_-}' \ + \ (W-k^2)u$ 
and eventually the following system for $(\alpha',\beta')$:
\begin{equation}\label{sys-alphabeta}
\left\{ \begin{array}{l}
\alpha' f^W_+ \ + \ \beta' f^W_- \ = \ 0 \\
 \alpha' {f^W_+}' \ + \ \beta' {f^W_-}' \ = \ -\left(-\partial_x^2 \ + \ W \ - \ k^2\right) \ u \ = \ (V -W) u 
\end{array}\right. \end{equation}

Solving for $\alpha'$ and $\beta'$ we have:
\[ \alpha' \ = \ \frac{-f^W_-(x,k) (V(x) -W(x)) u(x,k) }{\mathcal{W}[f^W_+ ,f^W_-](k)} \quad \text{and} \quad \beta'\ = \ \frac{f^W_+(x,k) (V(x) -W(x)) u(x,k) }{\mathcal{W}[f^W_+ ,f^W_-](k)}.\]

The expressions for $\alpha_\pm$ and $\beta_\pm$ in~\eqref{eq:alphabeta+} and 
\eqref{eq:alphabeta-} follow by integrating and imposing the asymptotic behavior of $u_\pm\sim f^{V}_\pm$ as $x\to\pm\infty$. In particular, one has $f^{V}_+(x;k)\sim f^{W}_+(x;k)\sim e^{ikx}$ when $x\to\infty$, and $f^{V}_-(x;k) \sim f^{W}_-(x;k) \sim e^{-ikx}$ when $x\to-\infty$. This completes the proof of Proposition~\ref{prop:q2s}. \end{proof}

\section{Convergence of $t^{q_\eps}(k)$ for $k\in\mathbb{C}$ and bifurcation of eigenvalues from the edge of the continuous spectrum}\label{sec:main-results}

In this section we state our main results for the Schr\"odinger equation~\eqref{Hq-def}
 with potential of the form:
 \begin{equation}
V_\eps(x)\ =\ V(x,x/\eps).
\label{formofpot}
\end{equation}

Recall the exponentially weighted norms $\big| g \big|_{W^{n,\infty}_{\beta}}$ introduced in section~\ref{sec:def}.
The potential $V(x,y)$ is assumed to satisfy the following precise hypotheses:

\noindent {\bf Hypotheses (V):} {\em
$V(x,y)$ is real-valued and of the form:
\begin{equation}
V(x,y)\ = \ q_{\rm av}(x)+ q(x,y) \ = \ q_{\rm av}(x) \ + \ \sum_{j\ne0}\ q_j(x)\ e^{2\pi i\lambda_j y}\ .
\label{q-expand}
\end{equation}
There exist positive constants $\theta>0$ and $\beta>0$ such that the sequence of non-zero (distinct) frequencies $\{\lambda_j\}_{j\in\mathbb{Z}\setminus\{0\}}$ satisfies
\begin{equation}\label{hyp-lambda}
\inf_{j\ne k} |\lambda_j-\lambda_k|\ge\theta>0, \quad \inf_{j\in\mathbb{Z}\setminus\{0\}} |\lambda_j|\ge\theta>0 \ ,
\end{equation}
and the coefficients $\{q_j(x)\}_{j\in\mathbb{Z}}$, satisfy the decay and regularity assumptions
 \begin{equation}
\big\bracevert V \big\bracevert \ \equiv \ \big|q_{\rm av}\big|_{W^{1,\infty}_{\beta}} \ + \ \sum_{j\in\mathbb{Z}\setminus\{0\}} \big|q_j\big|_{W^{3,\infty}_{\beta}}\ <\ \infty.
 \label{qj-decay-reg}
 \end{equation}
}
\begin{Remark}
If $V$ satisfies Hypotheses {\bf (V)}, and $\seff$ is defined in~\eqref{sigma-eff-1},\eqref{Lambda-def-aper}, then $V_\eps \in L^{\infty}_\beta$, $q_{\rm av}+\seff\in W^{1,\infty}_\beta$ and $\seff\in W^{3,\infty}_\beta$, and there exists $C(\big\bracevert V \big\bracevert)$, independent of $\eps$, such that
\[\big|V_\eps \big|_{L^{\infty}_\beta} \ \leq \ C(\big\bracevert V \big\bracevert) , \quad \big|q_{\rm av}+\seff \big|_{W^{1,\infty}_\beta} \ \leq \ C(\big\bracevert V \big\bracevert), \quad \big|\seff \big|_{W^{3,\infty}_\beta} \ \leq \ \eps^2 C(\big\bracevert V \big\bracevert).\]
\end{Remark}

Our approach is to study the Jost solutions, $f^{V_\eps}(x;k)$, and scattering coefficients, 
$t^{V_\eps}(k),\ r^{V_\eps}_\pm(k)$, for $\eps$ sufficiently small $\eps\in[0,\eps_0)$, and for $k$ in a complex neighborhood of zero. More precisely, we assume
\smallskip

\noindent {\bf Hypotheses (K):} {\em
We assume that the wave number, $k$, varies in $K$, a compact subset
of $\CC$ such that
 \begin{itemize}
\item $K\subset \{k,\ |\Im(k)|<\alpha\}$, with $0<\alpha<\beta/2$, and $\beta$ is as in Hypotheses~{\bf (V)};
\item $K$ does not contain any pole of the transmission coefficient, $t^{q_{\rm av}}(k)$. 
\end{itemize}
}
\noindent It follows that $t^{q_{\rm av}}(k)$ is bounded, uniformly for $k\in K$, and we define
\begin{equation}\label{def-Malpha}
M_K \ \equiv \ \max\big(1\ ,\ \sup_{k\in K} |t^{q_{\rm av}}(k)|\big) \ < \ \infty .
\end{equation}
Moreover, if $K\subset \RR$, then $M_K= 1$; see~\eqref{tr-energy}.


%

\begin{Remark}
We can relax the spatial decay assumptions of Hypotheses {\bf (V)}, if we restrict Hypotheses {\bf (K) }to the upper-half plane $\Im(k)\geq0$. Our methods apply and only require sufficient algebraic decay of $V(x)$. Results of this kind for $k\in\RR$ are presented in Section~\ref{sec:k-real}.
\end{Remark}

We now state our main theorem and its important consequences.

\begin{theorem}[Convergence of the transmission coefficient] \label{thm:t-conv} ~\\
Assume $V_\eps(x)=V(x,x/\eps)$ satisfies Hypotheses~{\bf (V)}, and $k\in K$ satisfies Hypotheses~{\bf (K)}. Then there exists $\eps_0>0$ such that for all $|\eps|<\eps_0$, $t^{q_{\rm av}+q_\eps}(k)$, the transmission coefficient of the scattering problem~\eqref{Hq-def}-\eqref{scattering} with
\begin{equation}
V_\eps(x)\ =\ \ q_{\rm av}(x)+q_\eps(x)\ =\ q_{\rm av}(x)+q(x,x/\eps),
\nn\end{equation}
 is uniformly approximated by the transmission coefficient, $t^{q_{\rm av}+\seff}(k)$, for 
 \begin{equation}
 V^\eps_{\rm eff}(x)\ =\ q_{\rm av}(x)\ +\ \seff(x), 
 \nn\end{equation}
 where $\seff(x)$ denotes the 
\underline{{\it effective potential well}}, 
 \begin{equation}
 \seff(x)\ \equiv\ -\eps^2\ \Lambda_{\rm eff}(x)\ \equiv\ -\frac{\eps^2}{(2\pi)^2} \sum_{j\in\mathbb{Z}\setminus\{0\}} \frac{|q_j(x)|^2}{{\lambda_j}^2}.
\label{sigma-eff}
%
\end{equation}
Specifically, we have the estimate
\begin{equation}
\sup_{k\in K} \left|\ \dfrac{k}{t^{q_{\rm av}+\seff}(k) } \ - \ \dfrac{k}{ t^{q_{\rm av}+q_\eps}(k) }\ \right| \ \le\ \eps^3\ M_K\ C(\big\bracevert V \big\bracevert,\sup_{k\in K}|k|),
 \label{t-diff-smallk}
\end{equation}
with $C(\big\bracevert V \big\bracevert) $ a constant, independent of $\eps$.
\end{theorem}
%
\noindent The proof of Theorem~\ref{thm:t-conv} is given in section~\ref{sec:hard}; we first present its consequences. A simple outcome of~\eqref{t-diff-smallk} and the genericity of $q_{\rm av}+\seff$ for $\epsilon$ sufficiently small (which holds for $q_{\rm av}$ generic {\em and} non-generic; see Corollary~\ref{cor:tseff}\footnote{Note that in the non-generic case, the condition $\int_{\RR}\Lambda_{\rm eff}(y)(f^{q_{\rm av}}_-(y;0))^2\dd y\neq 0$ is always satisfied. Indeed, $f^{q_{\rm av}}_-(\cdot;0)\in\RR$ by~\eqref{reality-condition}, and is non-zero almost everywhere on the support of $\Lambda_{\rm eff}$.}) is:
\begin{Corollary}\label{cor:qeps-is-generic}
Assume $V_\eps=q_{\rm av}+q_\eps$ satisfies Hypotheses {\bf (V)}. We allow $q_{\rm av}$ to be either generic or non-generic in the sense of Definition~\ref{def:generic}. Then, there exists $\eps_0>0$ such that for any $0<\eps<\eps_0$,  $V_\eps$ is generic.
\end{Corollary}

Theorem~\ref{thm:t-conv} holds for both generic and non-generic potentials, $q_{\rm av}$. In the following section we explore consequences for the non-generic potential, ${q_{av}(x)\equiv0}$, {\it i.e.} $V_\eps(x)=q(x,x/\eps),\ $ with $\int_0^1q(x,y)\dd y=0$. In particular, we explain the non-uniformity localization phenomenon discussed in the Introduction. Results for non-trivial $q_{\rm av}(x)$ are developed in sections~\ref{sec:refless} and~\ref{sec:genericcase}.

\subsection{Results for mean-zero oscillatory potentials: $q_{\rm av}(x)\equiv0$}
The following corollary, comparing $t^{q_\eps}(k)$ and $t^\seff(k)$, is a consequence of Theorem~\ref{thm:t-conv}, and Lemma~\ref{Lem:tsigma}.
\begin{Corollary}\label{cor:compare-t-qav-is-0}
 Let $q_{\rm av}\equiv0$, so that $V_\eps(x)=q_\eps(x)=q(x,x/\eps)$. Let $K$ denote the compact set of Hypotheses~{\bf (K)}. There exists $\eps_0>0$ such that if
 \begin{equation}\label{eq:condition1}
 \left| k - \frac{i\eps^2}{2}\int_{-\infty}^\infty \Lambda_{\rm eff}\right|\geq C \eps^\tau,\ \ \tau<3,\ \ k\in K, \ \ 0<\eps<\eps_0,
 \end{equation}
then one has for $0<\eps<\eps_0$,
 \begin{equation}
\left| \frac{t^{\seff}(k)}{t^{q_\eps}(k)}-1 \right| \ = \ \O\big(\eps^{3-\tau}\big).
\label{compare-tsq0iszero1}
 \end{equation}
If in addition to~\eqref{eq:condition1}, the following condition holds:
 \begin{equation}\label{eq:condition2}
 \left| k - \frac{i\eps^2}{2}\int_{-\infty}^\infty \Lambda_{\rm eff}\right|\geq C |k|, \ \ k\in K, \ \ 0<\eps<\eps_0\nn
 \end{equation}
 then one has for $0<\eps<\eps_0$,
\[
\left| t^{\seff}(k) - t^{q_\eps}(k) \right| \ = \ \O\big(\eps^{3-\tau}\big), \quad \text{and}\quad\left| t^{q_\eps}(k) \ - \ \frac{k}{k-\frac{i\eps^2}{2}\int_{-\infty}^\infty \Lambda_{\rm eff}} \right| \ = \ \O\big(\eps^{3-\tau}\big).
\]
In particular, if $k = \eps^{2}\kappa$, with $\kappa\neq \kappa^\star\equiv-\frac{1}{2i} \int \Lambda_{\rm eff}$, then for $0<\eps<\eps_0$,
 \begin{equation}
 \left| t^{\seff}(\eps^{2}\kappa) - t^{q_\eps}(\eps^{2}\kappa) \right| \ = \ \O\left(\frac{\eps\ |\kappa|}{|\kappa-\kappa^\star|^2}\right) \ = \ \O(\eps),\ \ \left| t^{q_\eps}(\eps^{2}\kappa) - \frac{\kappa}{\kappa-\frac{i}{2}\int_{-\infty}^\infty \Lambda_{\rm eff}} \right| \ = \ \O\big(\eps\big).
\label{compare-tsq0iszero3}
\end{equation}
\end{Corollary}
\begin{proof} Corollary~\ref{cor:tseff} of appendix~\ref{sec:seff-tools} gives
\begin{equation}
\frac{k}{t^\seff(k)} \ = \ \ k\ -\ 
\frac{i\eps^2}{2}\int_{-\infty}^\infty\ \Lambda_{\rm eff}(y)\ \dd y\ + \ \O\big(\eps^4\big), \ \ \eps\to 0\ ,
\label{eqnfork}
\end{equation}
uniformly for $k\in K$. By Theorem~\ref{thm:t-conv}, one has
\begin{equation}
\frac{k}{t^{q_\eps}(k)} \ = \ k \ -\ \frac{i\eps^2}{2}\int_{-\infty}^\infty \ \Lambda_{\rm eff}(y)\ \dd y\ + \ \mathcal{O}\left(\eps^3\right), \textrm{ uniformly for } k\in K\ .
\label{eqnfork1}
\end{equation}
Expansions~\eqref{eqnfork} and~\eqref{eqnfork1} imply straightforwardly~\eqref{compare-tsq0iszero1}--\eqref{compare-tsq0iszero3}. \end{proof}

A direct consequence of Corollary~\ref{cor:compare-t-qav-is-0} and the expansion of $t^{\seff}$ implied by Lemma~\ref{Lem:tsigma}, is the following result showing a universal scaled limit of $t^{q_\eps}$, depending on the single parameter, $\int_\RR\Lambda_{\rm eff}$.
\begin{Corollary}[Scaled limit of $t^{q_\eps}$]\label{cor:transm} Let $k=\eps^2\kappa$, with $\kappa \neq\frac{i}{2}\ \int_\RR\Lambda_{\rm eff}$. Then one has
\begin{equation}
 t^{q_\eps}(\eps^2\kappa)\ \to\ t^\star\Big(\kappa;\int_\RR\Lambda_{\rm eff}\Big)\equiv \frac{\kappa}{\ \ \kappa -\frac{i}{2}\ \int_\RR\Lambda_{\rm eff}\ \ }\ 
 {\rm as}\ \ \eps\to0,
 \label{scaled-transm2}
 \end{equation}
 where $t^\star\left(\kappa;m\right)$ is the transmission coefficient associated with the Schr\"odinger operator with attractive $\delta-$function potential well of total mass $m>0$: 
 \[H_{-m\delta}=-\partial_X^2 -m\delta(X).\]
 \end{Corollary}

As observed in section~\ref{sec:quicksum}, the poles of the transmission coefficient in the upper half $k-$plane, which must lie on the imaginary axis, correspond to the $L^2$ point eigenvalues.
From our estimates on the transmission coefficient, $t^{q_\eps}(k)$, we further deduce the existence
of a discrete eigenvalue near the edge of the continuous spectrum.
 
\begin{Corollary}[Edge bifurcation of point spectrum from the continuum] \label{cor:edge} ~\\
If $\epsilon$ if sufficiently small, then the transmission coefficient, $t^{q_\eps}(k)$ has a pole in the upper half plane at
\begin{equation}
k_\eps\ = \ i\ \frac{\eps^2}{2}\ \left(\int_\RR\Lambda_{\rm eff}\right) + \mathcal{O}(\eps^3)\ ,\ \ \eps\to 0\ , 
\nn\end{equation}
and therefore 
 $H_{q_\eps}$ has the simple eigenpair
\begin{gather*}
 E_\eps= k_\eps^2= -\frac{\eps^4}{4}\ \left(\int_\RR \Lambda_{\rm eff}\ \right)^2\ +\ \mathcal{O}(\eps^5) ,\ \ \eps\to 0\ , \\
 u_{E_{q_\eps}}(x)\ =\ \mathcal{O}\left(e^{-\sqrt{|E_{q_\eps}|}\ |x|}\right),\ \ |x|\gg1\ .
\end{gather*}
\end{Corollary}

\begin{proof}[Proof of Corollary~\ref{cor:edge}:]
Let us recall Rouch\'e's Theorem: Let $f$ and $g$ denote analytic functions, defined on an open set $A\subset\CC$. Let $\gamma$ denote a simple loop within A, which is homotopic to a point. If 
 $|g(k)-f(k)|<|f(k)|$ for all $k\in\gamma$, then $f$ and $g$ have the same number of roots inside $\gamma$.

Now let
\[f(k)\ \equiv \ k \ -\ \frac{i\eps^2}{2}\int_{-\infty}^\infty\Lambda_{\rm eff}(y)\dd y\ ,\]
\[ \ g_1(k) \ = \ \frac{k}{t^{\seff}(k)} , \qquad g_2(k) \ = \ \frac{k}{t^{q_\eps}(k)},\] 
and $\gamma \ = \ \{k:\ |k-\frac{i\eps^2}{2}\int\Lambda_{\rm eff}|=C\eps^3\}\subset K.$ 
 These functions are analytic in $k$; see~\cite{DeiftTrubowitz79} and our previous discussion. Theorem~\ref{thm:t-conv} and Corollary~\ref{cor:tseff} imply, respectively,
\[ g_2(k) = f(k) + \mathcal{O}(\eps^3)\ \ {\rm and}\ \ g_1(k) = f(k) +\ \mathcal{O}(\eps^4).\]
Therefore, there exist constants $a_K,\ b_K$, such that for $k\in\gamma$:
\[ \big\vert f(k) \ - \ g_1(k)\big\vert \ \le \ a_K\eps^4 , \qquad \big\vert f(k) \ - \ g_2(k)\big\vert \ \le \ b_K\eps^3 , \qquad \text{and} \qquad \vert f(k) \vert = C\eps^3. \]
Taking $\eps$ sufficiently small and choosing $C$ sufficiently large, Rouch\'e's Theorem implies that both $g_1$ and $g_2$, have unique roots, poles of $t^{\seff}$ and $t^{q_\eps}$, in the set $\{k:\ |k-\frac{i\eps^2}{2}\int\Lambda_{\rm eff}|\leq C\eps^3\}$. By self-adjointness, these poles lie on the positive imaginary axis. Corollary~\ref{cor:edge} now follows. 
\end{proof}

\subsection{Non-generic and non-zero $q_{\rm av}$; example of an oscillatory perturbation of a reflectionless potential}
\label{sec:refless}

As seen above, for the case where $q_{\rm av}\equiv0$ the transmission coefficient 
$t^{q_\eps}(k)$, does not converge to $t^{0}(k)\equiv1$ uniformly in a neighborhood of $k=0$ and the 
obstruction to uniform convergence is the approach, as $\eps\to0$, of a pole of $t^{q_\eps}(k)$ toward $k=0$. Such non-uniform convergence will occur whenever $t^{q_{\rm av}}(0)\ne0$. By~\eqref{tsig-Isig},~\eqref{generic-def}, we can have $t^{q_{\rm av}}(0)\ne0$ if and only if 
 $\mathcal{W}[f^{q_{\rm av}}_+,f^{q_{\rm av}}_-](0)=0$, the case where $q_{\rm av}$ is {\rm non-generic}; see section~\ref{sec:generic}.
 
 One may construct non-generic potentials as follows. Let $v(x)$ denote a potential well, $v(x)\le0$, say a square well, having one eigenstate and which is generic, {\it i.e.} ${\mathcal{W}[f^{v}_+,f^{v}_-](0)\ne0}$ and therefore $t^{v}(0)=0$. Consider the one-parameter family of Schr\"odinger operators defined as ${h_g=-\partial_x^2+gv(x),\ g\ge1}$. As $g$ increases, new eigenvalues of $h_g$ appear as $g$ tranverses discrete values $g_1<g_2<\cdots$. These eigenvalues appear via the crossing of a pole of $t^{gv}(k)$ in the lower half $k-$plane, for $g<g_N$, into the upper half plane for $g>g_N$. 
For $g$ equal to one of these transition values, $g_N$, one has $t^{g_N v}(0)\ne0$. Thus, $g_Nv(x)$ is a non-generic potential. Our analysis gives, for $q_{\rm av}$ taken to be any such non-generic potential, a precise description of the motion of the pole of $t^{q_{\rm av}+q_\eps}$ as it approaches $k=0$ for $\eps$ small. 

The following corollary, the analogue of Corollaries~\ref{cor:compare-t-qav-is-0} 
 and~\ref{cor:transm}, follows as in the case $q_{\rm av}\equiv0$ from Theorem~\ref{thm:t-conv} and Lemma~\ref{Lem:tsigma}.
\begin{Corollary}[Oscillatory perturbation of a reflectionless potential]\label{cor:compare-t-qav-is-reflectionless} ~\\ 
 Let $V_\eps(x)=q_{\rm av}+q_\eps(x)=q_{\rm av}+q(x,x/\eps)$ satisfy Hypotheses {\bf (V)}, let $q_{\rm av}$ be reflectionless, and finally let $k\in K$ satisfy Hypotheses {\bf (K)}. Assume in addition that the following condition holds,
\begin{equation}\label{conditionkeps}
 \left| \frac{k}{t^{q_{\rm av}}(k)} - \frac{i\eps^2}{2}\int_{-\infty}^\infty f^{q_{\rm av}}_-(y;k)\ \Lambda_{\rm eff}(y)\ f^{q_{\rm av}}_+(y;k)\dd y\right|\geq C\min(|k|, \eps^\tau),\ \tau<3,
\end{equation}
then one has for $\eps$ sufficiently small 
 \begin{equation}
\left| t^{q_{\rm av}+\seff}(k) - t^{q_{\rm av}+q_\eps}(k) \right| \ = \ \O\big(\eps^{3-\tau}\big).
\label{compare-tsq0isreflectionless}
\end{equation}
In particular, $k=\eps^2\kappa$ satisfies~\eqref{conditionkeps}, and therefore, by Lemma~\ref{Lem:tsigma}, there is a universal scaled limit of $t^{q_{\rm av}+q_\eps}(\eps^2\kappa)$:
\begin{align}
 t^{q_{\rm av}+q_\eps}(\eps^2\kappa)\ &\to\ \frac{ t^{q_{\rm av}}(0) \ \kappa}{\kappa -\frac{i}2 t^{q_{\rm av}}(0)\ \int_\RR f^{q_{\rm av}}_-(y;0)\ \Lambda_{\rm eff}(y)\ f^{q_{\rm av}}_+(y;0)\dd y\ \ }\nn\\
 &\ \ =\ \ \frac{ t^{q_{\rm av}}(0)\ \kappa}{\kappa -\frac{i}2 (1+r_-^{q_{\rm av}}(0))\ 
 \int_\RR (f^{q_{\rm av}}_-(y;0))^2\ \Lambda_{\rm eff}(y)\ \dd y\ \ },\ \ \ 
 {\rm as}\ \ \eps\to0
 \label{universal}\end{align}
 provided $\kappa\neq \kappa^\star \equiv \frac{i}2 t^{q_{\rm av}}(0)\ \int_\RR f^{q_{\rm av}}_-(y;0)\ \Lambda_{\rm eff}(y)\ f^{q_{\rm av}}_+(y;0)\dd y$. 
\footnote{Note that $\kappa^\star$ lies in the positive imaginary axis. Indeed, $f^{q_{\rm av}}_-(\cdot;0)\in\RR$ and $r_-(0)\in\RR$ by~\eqref{reality-condition}, and one has $r_-(0)+1\geq0$, since $|r_-(0)|\leq1$;  see~\eqref{tr-energy}.}
The last equality in~\eqref{universal} follows from~\eqref{eq:defrt+}.
 
The transmission coefficient, $t^{q_{\rm av}+\seff}(k)$ has a pole in the upper half plane at $k_{q_{\rm av}+\seff}$ the solution of the implicit equation: 
\begin{equation}
 k \ =\ i \frac{\eps^2}{2}\ t^{q_{\rm av}}(k)\ \int_{-\infty}^\infty f^{q_{\rm av}}_-(y;k)\ \Lambda_{\rm eff}(y)\ f^{q_{\rm av}}_+(y;k)\dd y\ +\ \mathcal{O}(\eps^4). 
\label{keps-gen}
\end{equation}
It follows that $H_{q_{\rm av}+\seff}$ has an eigenvalue at $E^\seff=(k_{q_{\rm av}+\seff}(\eps))^2<0$. 
Finally, Lemma~\ref{Lem:tsigma} and an application of Rouch\'e's Theorem imply that 
$t^{q_{\rm av}+q_\eps}(k)$, has a pole near $k^{q_{\rm av}+\seff}(\eps)$, on the positive imaginary axis, and a bound state 
\[E^{q_{\rm av}+q_\eps}(\eps)\approx\ E^{q_{\rm av}+\seff}(\eps)\ =\ \left[k^{q_{\rm av}+\seff}(\eps)\right]^2\ <\ 0.\] 
\end{Corollary}
\noindent We now consider this result in the context of a particular family of potentials. Consider the family of operators
$ h(g)= -\partial_x^2-g\ {\rm sech}^2(x)$.
 Let $g_N=N(N+1),\ N=0,1,2,\dots$. For ${g_N\le g<g_{N+1}}$, the operator $h(g)$ has precisely $N$- bound states. At the transition values, $h(g_N)$ has a zero energy resonance and $t^{h(g_N)}(0)\ne0$.
 The family of potentials for the values $g_N,\ N=0,1,2,\dots$, are called {\it reflectionless potentials} for which
$ |t(k)| \ \equiv \ 1$ and $ r_\pm(k) \ \equiv \ 0, \ \ k\in\RR$;
see~\cite{AblowitzSegur}. 
These potentials are well-known for their role as soliton solutions of the Korteweg-de Vries equation.

Consider the case of the one-soliton potential, corresponding to $N=1$ in the above discussion. Here, 
\[V_1(x) \ = \ -2\rho^2\text{sech}^2(\rho(x-x_0)), \quad \text{where $x_0$ satisfies } C =2\rho \exp(2\rho x_0).\]
In this case, the transmission coefficient satisfies
\[ \frac1{t^{V_1}(k)} \ = \ \lim\limits_{x\to -\infty} f^{V_1}_+(x;k)e^{-ikx} \ = \ \frac{k-i\rho}{k+i\rho}.\]
As for the Jost solutions, one has (setting $x_0=0$ for simplicity)
\[ f^{V_1}_+(x;k)=e^{ikx}\left(1- \frac{2i\rho}{k+i\rho}\frac{ e^{-x} }{ e^{x}+ e^{-x}} \right).\]
Since the $V_1$ is reflectionless, one has by~\eqref{eq:defrt-}
\[ f^{V_1}_-(x;k)=0+\frac{1}{t^{V_1}(k) }f^{V_1}_+(x;-k) \ = \ \frac{1}{t^{V_1}(k) }e^{-ikx}\left(1- \frac{2i\rho}{-k+i\rho}\frac{ e^{-x} }{ e^{x}+ e^{-x}} \right).\]

In this case, there exists a pole of $t^{{V_1}+\seff}(k)$, and similarly a pole of $t^{{V_1}+q_\eps}(k)$, located around
\begin{align*}
k \ & = \ i\frac{\eps^2}{2}\int_{-\infty}^\infty t^{V_1}(0) f^{V_1}_-(y;0)\Lambda_{\rm eff}(y)f^{V_1}_+(y;0)\dd y \ + \ \O(\eps^3)\ , \\
& = \ i\frac{\eps^2}{2}\int_{-\infty}^\infty \tanh^2(y)\Lambda_{\rm eff}(y)\dd y \ + \ \O(\eps^3)\ , \ \ \eps\to 0.
\end{align*}
Finally, $H_{{V_1}+q_\eps}$ and $H_{{V_1}+\seff}$ have a bound state with energy
\begin{equation}
E \ = \ -\frac{\eps^4}4 \left(\int_\RR \tanh^2(y) \Lambda_{\rm eff}(y)\ \dd y\ \right)^2 \ + \ \O(\eps^5)\ , \ \ \eps\to 0.
\nn\end{equation}

\subsection{Results for generic potentials, $q_{\rm av}$, and their highly oscillatory perturbations }
\label{sec:genericcase}

In this section, we study the case where $q_{\rm av}$ is a generic potential in the sense of section~\ref{sec:quicksum}. In this case $t^{q_{\rm av}+q_\eps}(k)$ converges uniformly to $t^{q_{\rm av}}(k)$ in a neighborhood
of $k=0$~\cite{DucheneWeinstein:11}. More precise information is contained in the following Corollary, a direct consequence of Lemma~\ref{Lem:tsigma}, and Theorem~\ref{thm:t-conv}.
\begin{Corollary}\label{Cor:generic}
 Let $V_\eps(x)=q_{\rm av}(x)+q_\eps(x)=q_{\rm av}(x)+q(x,x/\eps)$ satisfy Hypotheses~{\bf (V)} with $q_{\rm av}$ generic, and $k\in K$ satisfy Hypotheses~{\bf (K)}. Then for $k$ and $\eps$ small enough, one has
 \begin{align}
 \big\vert t^{q_{\rm av}+\seff}(k) \big\vert \ &\leq \ C_0 |k|, \label{prop-generic-1} \\
 \big\vert t^{q_{\rm av}+q_\eps}(k) \big\vert \ &\leq \ C_0 |k|, \label{prop-generic-2}\\
 \big\vert t^{q_{\rm av}+q_\eps}(k) -t^{q_{\rm av}+\seff}(k) \big\vert \ &\leq \ C_0\eps^3 |k|,\label{prop-generic-3}
 \end{align}
 with $C_0=C_0(M_K)$, $M_K=\max(1,\sup_{k\in K}|t^{q_{\rm av}}(k)|)$.
\end{Corollary}
\begin{proof}
In the case of generic potentials, $q_{\rm av}$, we know from~\cite{DeiftTrubowitz79} that there exists a constant $a_{q_{\rm av}}$ such that
\[ t^{q_{\rm av}}(k) \ = \ a_{q_{\rm av}} k \ + \ o(k), \quad \text{ as } k\to 0.\]
 It follows that for $k$ sufficiently small, there exists a constant $C_0$ such that 
$ \left|k\ (t^{q_{\rm av}}(k))^{-1} \right| \ \geq \ C_0>0$. 
Estimate~\eqref{prop-generic-1} follows then straightforwardly from Lemma~\ref{Lem:tsigma}, when $\eps$ is sufficiently small.
 Now, applying Theorem~\ref{thm:t-conv}, one has
\begin{align*}
 \big\vert t^{q_{\rm av}+\seff}(k)-t^{q_{\rm av}+q_\eps}(k) \big\vert 
&=  \left| \dfrac{k}{t^{q_{\rm av}+\seff}(k) }  -  \dfrac{k}{ t^{q_{\rm av}+q_\eps}(k) } \right|\left| \dfrac{t^{q_{\rm av}+\seff}(k)\ t^{q_{\rm av}+q_\eps}(k) }{k} \right| \\
 &\leq \ C_0\eps^3 \big\vert t^{q_{\rm av}+q_\eps}(k) \big\vert.\end{align*}
Estimate~\eqref{prop-generic-2} and then ~\eqref{prop-generic-3} follow easily. This concludes the proof. \end{proof}

\section{ Behavior of the transmission coefficient, uniformly in ${k\in\RR}$}
\label{sec:k-real}

In this section we focus on the properties of $t^{q_\eps}(k)$, which hold uniformly in $k\in\RR$. 
The results presented in section~\ref{sec:quicksum} are valid for $k\in\RR$, and under the less stringent condition: 
 $V\in \mathcal{L}^{1}_2(\RR)\ =\ \{V: (1+|x|)^2 V(x)\in L^1(\RR)\}$. 
Most of these results can be found in~\cite{DeiftTrubowitz79}. Our required bounds on the Jost solutions, $f_\pm^V$ are given in Lemma~\ref{Lem:f-gen-k-real}.

Since $k$ is constrained to the real axis, we find that we can relax the assumption of exponential decay on the potential $V_\eps=V(x,x/\eps)$. 

\noindent {\bf Hypotheses (V'):} {\em $V(x,y)$ is a real-valued potential of the form
\[V(x,y)\ = \ q_{\rm av}(x)+ q(x,y) \ = \ q_{\rm av}(x) \ + \ \sum_{j\ne0}\ q_j(x)\ e^{2\pi i\lambda_j y}, \]
such that the sequence of non-zero (distinct) frequencies $\{\lambda_j\}_{j\in\mathbb{Z}\setminus\{0\}}$ satisfies~\eqref{hyp-lambda}, and the coefficients $\{q_j(x)\}_{j\in\mathbb{Z}}$, satisfy the decay and regularity assumptions
 \begin{equation}
\big\bracevert\hspace{-6pt}\big\bracevert V \big\bracevert\hspace{-6pt}\big\bracevert \ \equiv \ \big|q_{\rm av}\big|_{\mathcal{W}^{1,1}_2} \ + \ \sum_{j\in\mathbb{Z}\setminus\{0\}} \big|q_j\big|_{\mathcal{W}^{3,1}_3}\ <\ \infty.
 \label{qj-decay-reg-real}
 \end{equation}
}
 We first investigate the difference between the transmission coefficients $t^{q_{\rm av}+q_\eps}(k)$ and $t^{q_{\rm av}+\seff}(k)$, where $\seff$ is defined as in Theorem~\ref{thm:t-conv}. 
 %
 %
 The proof of the following theorem is analogous to that of Theorem~\ref{thm:t-conv} (section~\ref{sec:hard}). We omit the proof for the sake of brevity.
 
 \begin{theorem}[Transmission coefficient, $t^{V_\eps}(k)$, for $k\in\RR$]\label{thm:t-conv-real}
 Assume $V_\eps(x)=V(x,x/\eps)$ satisfies Hypotheses {\bf (V')}. Assume $k\in \RR$, $|k|\leq1$. Then, the following holds:
\begin{enumerate}[(1)]
\item There exists $\eps_0>0$ such that for all $|\eps|<\eps_0$, $t^{q_{\rm av}+q_\eps}(k)$
 is uniformly approximated by the transmission coefficient, $t^{q_{\rm av}+\seff}(k)$, for $H_{q_{\rm av}+\seff}$. Here $\seff(x)$ denotes the 
\underline{{\it effective potential well}} defined in~\eqref{sigma-eff}.
%

\noindent Moreover, there is a constant $C(\big\bracevert\hspace{-6pt}\big\bracevert V \big\bracevert\hspace{-6pt}\big\bracevert) $, independent of $\eps$ and $k$, such that
\begin{equation}
\sup_{k\in \RR, \ |k|\le1} \left| \dfrac{k}{t^{q_{\rm av}+\seff}(k) }  -  \dfrac{k}{ t^{q_{\rm av}+q_\eps}(k) }\ \right|  \le \eps^3\ C(\big\bracevert\hspace{-6pt}\big\bracevert V \big\bracevert\hspace{-6pt}\big\bracevert)\ \max(1,\sup_{k\in K}|t^{q_{\rm av}}(k)|) \le \eps^3\ C(\big\bracevert\hspace{-6pt}\big\bracevert V \big\bracevert\hspace{-6pt}\big\bracevert).
 \label{t-diff-smallk-real}
\end{equation}
\item Assume $q_{av}\equiv0$, so that $H_{V_\eps}=-\partial_x^2+q(x,x/\eps)$, where $y\mapsto q(x,y)$ has mean zero. Then, applying~\eqref{t-diff-smallk-real} and Corollary~\ref{cor:tseff} we have
\begin{equation}
t^{q_\eps}(k)\ =\ \frac{k}{k-\frac{i}{2}\ \eps^2\ \int_\mathbb{R} \Lambda_{\rm eff}\ +\ \mathcal{O}(\eps^3)}
\label{teps-real}
\end{equation}

\end{enumerate}
\end{theorem}

In the following, we are able to control the difference between $t^{q_{\rm av}+q_\eps}(k)$ and $t^{q_{\rm av}+\seff}(k)$, for large real wave number, $|k|\geq1$. This allows, in particular, control of the difference between $t^{q_{\rm av}+q_\eps}(k)$ and $t^{q_{\rm av}+\seff}(k)$, when the averaged potential $q_{\rm av}\equiv0$, {\em uniformly in $k\in\RR$}.

\begin{Proposition}\label{prop:large-k1}
Let $V_\eps\equiv V(x,x/\eps)\equiv q_{\rm av}+q_\eps$ with $V$ satisfying Hypotheses {\bf (V')}, and $\sigma^\eps(x)$ denote any potential for which 
\begin{equation}
 \int|\sigma^\eps(y)|(1+|y|) \dd y \le \eps^2\ C_\sigma
 \nn\end{equation}
 Then, for $k\in\RR\setminus\{0\}$, one has
 \begin{equation}
 \left|\ t^{q_{\rm av}+q_\eps}(k)-t^{q_{\rm av}+\sigma^\eps}(k)\ \right|\ \le\ \eps^2\ |k|^{-1}\ C(\big\bracevert\hspace{-6pt}\big\bracevert V \big\bracevert\hspace{-6pt}\big\bracevert,C_\sigma)\ ,
 \label{large-k-bound}
 \end{equation}
where $\big\bracevert\hspace{-6pt}\big\bracevert V \big\bracevert\hspace{-6pt}\big\bracevert$ is defined in~\eqref{qj-decay-reg-real}. 
\end{Proposition}
\begin{Remark}
We shall apply this proposition to $\sigma^\eps(x)=\seff(x)$, for which ${C_\sigma=\mathcal{O}(\big\bracevert\hspace{-6pt}\big\bracevert V \big\bracevert\hspace{-6pt}\big\bracevert)}$. 
 \end{Remark}
\begin{proof}
Recall the identity~\eqref{IVW-def}, relating the transmission coefficients of {\em any potentials} $V,W\in \mathcal{L}^1_{2}$:
\begin{equation}
\frac{k}{t^V(k)}  =  \frac{k}{t^W(k)}  -  \frac{ I^{[V,W]}(k)}{2i},\  \text{with } \ I^{[V,W]}(k) \equiv \int_{-\infty}^\infty f^W_-(y;k)(V-W)(y)f^V_+(y;k)\dd y.
\label{tdiff}
\end{equation}
Since $t^{ q_{\rm av}+q_\eps } - t^{q_{\rm av}+\sigma^\eps}=\left[t^{ q_{\rm av}+q_\eps }-t^{q_{\rm av}}\right]\ +\ \left[t^{q_{\rm av}}- t^{q_{\rm av}+\sigma^\eps}\right] $, we estimate the two bracketed terms independently. We begin by comparing the transmission coefficients for $W\equiv q_{\rm av}$ and $V \ \equiv \ q_{\rm av}+\sigma^\eps$. We have by~\eqref{tdiff}
\begin{equation}
\frac{k}{t^{q_{\rm av}+\sigma^\eps}(k)}  -  \frac{k}{t^{q_{\rm av}}(k)}  =  -\frac{1}{2i} I^{[q_{\rm av}+\sigma^\eps,q_{\rm av}]}(k) = -\frac{1}{2i} \int_{-\infty}^\infty f^{q_{\rm av}}_-(y;k)\sigma^\eps(y)f^{q_{\rm av}+\sigma^\eps}_+(y;k)\dd y.
\label{recip-diff}
\end{equation}
Using the estimates of Lemma~\ref{Lem:f-gen}, we obtain
\begin{align}
\left|\ \int_{-\infty}^\infty f^{q_{\rm av}}_-(y;k)\ \sigma^\eps(y)\ f^{q_{\rm av}+\sigma^\eps}_+(y;k)\dd y\ \right| 
&\le \eps^2 C_\sigma. \label{sandwich}
\end{align}
From~\eqref{recip-diff} and~\eqref{sandwich} we have 
\begin{equation}
\left|\ t^{q_{\rm av}+\sigma^\eps}(k) - t^{q_{\rm av}}(k)\ \right|\ \le\ \eps^2\ |k|^{-1}\ C_\sigma\ \left| t^{q_{\rm av}}(k)\ t^{q_{\rm av}+\sigma^\eps}(k) \right|\ .
\label{tdiffs}\end{equation}
Using the general relation $|t^V(k)|\le1$, for any $k\in\RR$, (see~\eqref{tr-energy}), we obtain
\[
\left|\ t^{q_{\rm av}+\sigma^\eps}(k)-t^{q_{\rm av}}(k)\ \right|\ \le\ \eps^2\ |k|^{-1}\ C_\sigma\ .\]

We now turn to the comparison of the transmission coefficients of $V \ \equiv \ q_{\rm av}+q_\eps$
 and $W\ \equiv q_{\rm av}$. Proceeding similarly, we have
\begin{align}
&\frac{k}{t^{q_{\rm av}+q_\eps}(k)} \ - \ \frac{k}{t^{q_{\rm av}}(k)} \ = \ -\frac{1}{2i} I^{[q_{\rm av},q_{\rm av}+q_\eps]}(k),\ \ {\rm where}\ \nn\\
 & I^{[q_{\rm av},q_{\rm av}+q_\eps]}(k)\ \equiv\ \int_{-\infty}^\infty f^{q_{\rm av}}_-(y;k)\ q_\eps(y)\ f^{q_{\rm av}+q_\eps}_+(y;k)\dd y\ .\label{Idef}
 \end{align}
Two integrations by parts yield
\begin{align*}
I^{[q_{\rm av},q_{\rm av}+q_\eps]}(k) \ &= \ \sum_{j\neq0}\int_{-\infty}^\infty q_j(y)e^{2i\pi {\lambda_j}y/\eps}f^{q_{\rm av}}_-(y;k) f^{q_{\rm av}+q_\eps}_+(y;k)\dd y \\
&= \sum_{j\neq0} \left(\frac{-\eps}{2i\pi {\lambda_j}}\right)^2 \int_{-\infty}^\infty \partial_y^2 \big(q_j(y)f^{q_{\rm av}}_-(y;k) f^{q_{\rm av}+q_\eps}_+(y;k)\big)e^{2i\pi {\lambda_j}y/\eps}\dd y. 
\end{align*}
Using the estimates of Lemma~\ref{Lem:f-gen-k-real} and Hypotheses {\bf (V')}, one sees that the integrand is bounded. Indeed, one has
\begin{align*} 
&\left|\ I^{[q_{\rm av},q_{\rm av}+q_\eps]}(k)\ \right|\ \le\  \sum_{j\neq0} \left(\frac{\eps}{2\pi {\lambda_j}}\right)^2  \ \int_{-\infty}^\infty\big| \partial_y^2 \big(q_j(y)f^{q_{\rm av}}_-(y;k) f^{q_{\rm av}+q_\eps}_+(y;k)\big)\big| \dd y  \nn\\
 &\quad \leq \ \eps^2  C\big(\big|q_{\rm av}\big|_{\mathcal{L}^{1}_2}\big) \sum_{j\neq0} \left[ \int_{-\infty}^\infty\big|\partial_y^2 q_j(y)\big|\frac{(1+|y|)^2}{(1+|k|)^{2}} \dd y\ + \ \int_{-\infty}^\infty\big|\partial_y q_j(y)\big|\frac{(1+|y|)^2}{1+|k|} \dd y  \right.\\
 & \qquad \left. + \   \int_{-\infty}^\infty\big| q_j(y)\big|(1+|y|)^2 \dd y\right]\   \leq \ \eps^2 C\big(\big|q_{\rm av}\big|_{\mathcal{L}^{1}_2}\big)\sum_{j\neq0} \big|q_j\big|_{\mathcal{W}^{1,1}_2}.
\end{align*}
Arguing as in~\eqref{tdiffs}, we deduce
\[\left|\ t^{q_{\rm av}+q_\eps}(k)-t^{q_{\rm av}}(k)\ \right| \ \le\ \eps^2\ |k|^{-1}\ C(\big\bracevert\hspace{-6pt}\big\bracevert V \big\bracevert\hspace{-6pt}\big\bracevert)\ \left| t^{q_{\rm av}}(k)\ t^{q_{\rm av}+q_\eps}(k) \right| \ \le\ \eps^2\ |k|^{-1} \ C(\big\bracevert\hspace{-6pt}\big\bracevert V \big\bracevert\hspace{-6pt}\big\bracevert).\]

This completes the proof of Proposition~\ref{prop:large-k1}.
\end{proof}

The following corollary follows from Theorem~\ref{thm:t-conv-real} and Proposition~\ref{prop:large-k1}.

\begin{Corollary}\label{cor:k-real}
Let $V_\eps=q_\eps=q(x,x/\eps)$ ($q_{\rm}=0$) satisfy Hypotheses {\bf (V')}. Then
 \begin{equation}
\sup_{k\in \RR} \left| t^{\seff}(k) - t^{q_\eps}(k) \right| \ = \ \O(\eps), \quad \eps\to 0.
\label{compare-tsq0iszero4}
\end{equation}
\end{Corollary}
\begin{proof}
The behavior for $k$ small is controlled as in Corollary~\ref{cor:compare-t-qav-is-0}. Conditions~\eqref{eq:condition1} and~\eqref{eq:condition2} hold in particular when we restrict to real wave numbers, $k\in \RR$. Therefore, one sees from~\eqref{eqnfork} and~\eqref{eqnfork1} that the difference between $t^{q_\eps}(k)$ and $t^{\seff}(k)$ is small, uniformly for $|k|\leq1$, $k\in\RR$:
\[ \sup_{k\in \RR, \ |k|\leq 1} \left| t^{\seff}(k) - t^{q_\eps}(k) \right| \ \leq \ C\frac{\eps^3}{\eps^2+|k|} , \]
where $C=C(M_K)$, and $M_K=\max(1,\sup_{k\in \RR}|t^0(k)|)=1$.
The difference is controlled for $|k|\geq 1$ by Proposition~\ref{prop:large-k1}, and Corollary~\ref{cor:k-real} follows.
\end{proof}

\section{Detailed dispersive time decay of $\exp\left(-iH_{q_\eps}t\right)\psi_0$;\\
 the effect of a pole of $t^{q_\eps}(k)$ near $k=0$}
\label{sec:time-decay}

In this section we use our detailed results on $t^{q_\eps}(k)$ to prove time decay estimates
 of the Schr\"odinger equation:
 \begin{equation}\label{eq:schrod}
i \partial_t \psi \ = \ H_{V}\psi \ \equiv \ -\ \partial_x^2 \psi \ + \ V(x)\psi,\ \ 
\psi(0,x) \ = \ \psi_0\ .
\end{equation}
for initial conditions $\psi_0$, which are orthogonal to the bound states of $H_{q_\eps}$.

Let $V\in \mathcal{L}^1_{1}$. Then, it is known that $H_V$ has no singular-continuous spectrum, no positive ({\it embedded}) eigenvalues and its absolutely-continuous spectrum is $[0, \infty)$; see, for example,~\cite{DeiftTrubowitz79}. In general, $H_V$ may have a finite number of negative eigenvalues that are simple: $E_N<\dots <E_0<0$. We denote by $u_j$ the eigenvector associated to the eigenvalue $E_j$, normalized to have $L^2$ norm equal to one. By the spectral theorem, the solution of~\eqref{eq:schrod} can be decomposed as follows:
\[\psi(x,t) \ = \ e^{-itH_V} \psi_0 \ = \ \sum_{j=0}^N e^{-itE_j} (\psi_0,u_j) u_j + e^{-itH_V} P_c\psi_0,\]
where $P_c$ denotes the projection onto the continuous spectral subspace of $H$.
 $\exp(-itH_V) P_c\psi_0$ is a {\it scattering state} which spatially spreads and temporally decays: 
$ \big\vert e^{-itH_V} P_c\psi_0 \big\vert_{L^\infty_x}\to0$ as $t\to\infty$. 
In the case $V(x)\equiv 0$, we have $\psi(x,t)=\exp(it\partial_x^2 )\psi_0=K_t\star\psi_0,$ where $\left| K_t(x) \right| \le (4\pi t)^{-1/2}$.
From this decay estimate it follows immediately that
$\big\vert e^{-itH_0} P_c\psi_0 \big\vert_{L^\infty_x} \ \le\ C\ |t|^{-1/2}\ \big\vert\psi_0\big\vert_{L^1}.$
This $|t|^{-1/2}$ decay-rate is associated with the potential $V\equiv0$ being non-generic. 
For generic potentials the decay estimate is more rapid:
$ \big\vert e^{-itH_V} P_c\psi_0 \big\vert_{L^\infty_x} \ = \ \O(t^{-3/2});$ 
see~\cite{Goldberg07},~\cite{Schlag07}. In~\cite{Weder00,ArtbazarYajima00} the time-decay of spatially weighted 
 $L^2$ norms is studied. 
\smallskip 

\noindent {\it Question:\ What is the precise behavior of the 
 $e^{-itH_{q_\eps}} P_c\psi_0$, when $q_\eps$ is a highly oscillatory potential: $q_\eps(x) \equiv q(x,x/\eps)$?\
In particular, what is the influence of the low-energy bound state induced by the effective potential well (equivalently, the complex pole of $t^{q_\eps}(k)$ near $k=0$) on the dispersive decay properties?}
\medskip

\noindent Using the preceding analysis we can prove:
\begin{theorem}[Dispersive decay estimate for $\exp(-iH_{q_\eps}t)$] \label{cor:schrod} ~\\
 Let $V_\eps=q_\eps(x)=q(x,x/\eps)$ satisfy Hypotheses~{\bf (V')} with ${q_{\rm av} \equiv 0}$, and $\psi_0\in \mathcal{L}^1_{3}$. There exists constants $C=C(\big\bracevert\hspace{-6pt}\big\bracevert V \big\bracevert\hspace{-6pt}\big\bracevert)>0$ and $\eps_0>0$ such that for $0<\eps<\eps_0$, 
\begin{align}
\big\vert (1+|x|)^{-3} \left(e^{-itH_{q_\eps}} P_c\psi_0\right)(t,x)\big\vert \ &\le\ C\ \frac{1}{t^{1/2}}\ 
\frac{1}{1+\eps^4\ \left(\int_\mathbb{R}\Lambda_{\rm eff}\right)^2\ t}\ \big|\psi_0\big|_{\mathcal{L}^1_3}\ .
\label{decay-eps} \end{align}
\end{theorem}
\begin{Remark} We expect that an analogous result holds with $V_\eps=q_{\rm av}(x)+q(x,x/\eps)$, where $q_{\rm av}$ is any non-generic potential.
\end{Remark}
\begin{Remark}
As a consequence of our proof, a decay estimate like~\eqref{decay-eps} holds in the case of small potentials: $V\equiv \lambda Q$, with $\int Q\neq0$ and $\lambda$ sufficiently small:
\[
\big\vert (1+|x|)^{-3} \left(e^{-itH_{\lambda Q}} P_c\psi_0\right)(t,x)\big\vert \ \le\ C\ \frac{1}{t^{1/2}}\ 
 \frac{1}{1+\lambda^2\ \left(\int_\mathbb{R}Q\right)^2\ t}\ \big|\psi_0\big|_{\mathcal{L}^1_3}.
 \]
\end{Remark}

\noindent {\em Proof of Theorem~\ref{cor:schrod}.} We follow the method of~\cite{Goldberg07,Schlag07}. In particular, the starting point of our analysis is the spectral theorem for $H$:
$
P_c \phi = \mathcal{F}^\star \mathcal{F} \phi,
$
with $\mathcal{F}$ and $\mathcal{F}^\star$ the distorted Fourier transform and its adjoint, bounded operators on $L^2$:
\[\begin{array}{cl}
 \mathcal{F} & :\ \phi \mapsto \mathcal{F}[\phi](k) \equiv \displaystyle \int_\mathbb{R} \phi (x) \overline{\Psi (x,k)}\ \dd x,\\
 \mathcal{F}^\star & :\ \Phi \mapsto \displaystyle \int_{-\infty}^{+\infty} \Phi(k) \Psi (x,k)\ \dd k
 \end{array}\]
and
\[
\Psi (x;k) \ \equiv \ \frac{1}{\sqrt{2 \pi}} \left\{ \begin{array}{cc}
t(k) f^{q_\eps}_+(x;k) & k \geq 0, \\
t(-k) f^{q_\eps}_-(x;-k) & k < 0.
\end{array} \right. 
\]
The role of the transmission coefficient, $t^{q_\eps}(k)$ on the time-evolution on the continuous spectral part of $H_{q_\eps}$ is made explicit via the representation of ${\psi_c(x,t)=P_c\psi(x,t)}$:
\begin{align*}\psi_c(t,x) \ &\equiv \ e^{-itH_{q_\eps}} P_c\psi_0 \ = \ \mathcal{F}^\star e^{-itk^2} \mathcal{F} \psi_0\\
& =\frac1{2\pi} \int_0^\infty e^{-ik^2t} |t^{q_\eps}(k)|^2 F(x;k) \dd k,\end{align*}
with
\[F(x;k) \ = \ \int_{-\infty}^\infty \big[ f^{q_\eps}_+(x;k) \overline{f^{q_\eps}_+ (y,k)}+f^{q_\eps}_-(x;k) \overline{f^{q_\eps}_- (y,k)}\big] \psi_0(y) \dd y.\]
We next decompose $\psi_{c}(x,t)$ into its high and low frequency components, respectively, {\em i.e.} components respectively near and far away from the edge of the continuous spectrum. Introduce the smooth cutoff function
 $\chi$ defined by
\[\chi(k) \ \equiv \ 0 \quad \text{ for } \quad |k|\geq 2k_0 \ ,\qquad \qquad \chi(k) \ \equiv \ 1\quad \text{ for } \quad |k|\leq k_0.\]
Here, we set $k_0=1+\big\bracevert\hspace{-6pt}\big\bracevert V \big\bracevert\hspace{-6pt}\big\bracevert$, motivated by the high frequency analysis of~\cite{Schlag07}. 
Using $\chi(k)$, we decompose into high and low energy components $\psi_{\rm high}$ and $\psi_{\rm low}$:
\begin{align} 
\psi_c(t,x)  &= \ \psi_{\rm low}(t,x)+\psi_{\rm high}(t,x) \nn \\
& = \  \int_0^\infty \chi\ e^{-ik^2t} |t^{q_\eps}(k)|^2\ F(x;k) \frac{\dd k}{2\pi}  +   \int_0^\infty (1-\chi )e^{-ik^2t}\ |t^{q_\eps}(k)|^2\ F(x;k) \frac{\dd k}{2\pi}  .\label{lh-frequencies} 
\end{align}

\noindent $\psi_{\rm high}$, can be estimated without regard to whether or not $V$ is generic. We refer to Proposition 3 of~\cite{Goldberg07} and Theorem 3.1 of~\cite{Schlag07}, for the following estimate:
\begin{equation}\label{high-est}
\big\vert (1+|x|)^{-1}\psi_{\rm high} \big\vert_{L^\infty_x} \ = \ \big\vert (1+|x|)^{-1} e^{-itH_{q_\eps}} (1-\chi(H) ) P_c\psi_0\big\vert_{L^\infty_x} \ \leq \ C\ |t|^{-3/2}\big\vert\psi_0\big\vert_{\mathcal{L}^1_{1}},
\end{equation} 
where $C$ depends on $\big| q_\eps \big|_{L^1_1}$ and is bounded, independent of $\eps$.

To estimate the low energy component, $\psi_{\rm low}$, we make use of estimates on the Jost solutions, $f^{q_\eps}_\pm(x;k)$ and use the precise behavior of $t^{q_\eps}(k)$ obtained in Corollary~\ref{cor:k-real}. 
 We first obtain $\O(t^{-1/2})$- decay, uniformly for $\eps$. In a second step, we obtain the precise behavior in the statement of Theorem~\ref{cor:schrod}, for $\eps$ small.

 Let us decompose $\psi_{\rm low}$ into contributions from frequencies in the ranges:
 \begin{equation} 
 0\le k\le \frac{k_0}{\sqrt{t}}\ \ {\rm and}\ \ \frac{k_0}{\sqrt{t}}\le k\ \le 2k_0. 
 \nn\end{equation}
 In terms of the cutoff function, $\chi$, we have:
 \begin{align}
 \psi_{\rm low} \ 
 &=\ \frac1{2\pi} \int_0^\infty \chi(k\sqrt t) \chi(k)\ e^{-ik^2t} |t^{q_\eps}(k)|^2\ F(x;k)\ \dd k \nn\\
 &\quad + \ \frac1{2\pi} \int_0^\infty (1-\chi(k\sqrt t)) \chi(k)\ e^{-ik^2t} |t^{q_\eps}(k)|^2\ F(x;k)\ \dd k \label{decompose1-low}\nn\\
 &=\ \psi_{\rm low}^{(i)}(x,t)\ +\ \psi_{\rm low}^{(ii)}(x,t)
\end{align}

 Straightforward estimate of $\psi_{\rm low}^{(i)}$ gives:
\begin{equation}\label{low1}
\left| \psi_{\rm low}^{(i)}(x,t) \right| \le \frac{1}{2\pi } \int_0^{2k_0/\sqrt t} |t^{q_\eps}(k)|^2\ F(x;k) \dd k \ \leq \ \frac{k_0}{\pi }\ \frac{1}{ t^{1/2}}\ \sup_{k\in\RR} \left| F(x,k)\right| .\end{equation}
To estimate $\psi_{\rm low}^{(ii)}$, we integrate by parts:
 \[ \psi_{\rm low}^{(ii)}(x,t)\ = \ \frac{-1}{4\pi i t} \int_0^\infty e^{-ik^2t} \partial_k\left((1-\chi(k\sqrt t))\chi(k)k^{-1} |t^{q_\eps}(k)|^2 F(x;k) \right)\dd k. \]
 Note that there is no boundary contribution from $k=\infty$, since $\chi(k)$ is compactly supported, and no boundary contribution from $k=0$, since $|t^{q_\eps}(0)|=0$; $q_\eps$ is generic if $\eps$ is small enough, by Corollary~\ref{cor:qeps-is-generic}.
 
 Since $\chi(x,k)\equiv0$ for $k\geq 2 k_0$ and $1-\chi(k\sqrt t)\equiv0$ for $k\leq k_0/\sqrt t$, it follows that
 \begin{align*}
\left|\psi_{\rm low}^{(ii)}(x,t)\right| &\leq  \frac{C}{ t}\ \int_{k_0/\sqrt t}^{2 k_0} \left|\ |t^{q_\eps}(k)|^2 F(x;k) {\partial_k}\left[\chi(k)\frac{1-\chi(k\sqrt t)}{2ik} \right]\ \right| \ + \ \left| \frac{\partial_k \left[ |t^{q_\eps}(k)|^2 F(x;k)\right]}{k}\right| \dd k\\
 & \leq  \frac{C}{t} \sup_{k\in\RR} \left| F(x,k)\right| \int_{k_0/\sqrt t}^{2 k_0}
 \sqrt{t} \frac{|\chi'(k\sqrt{t})|}{k}+\frac{1}{k^2}\dd k  +  \frac{C}{t} \int_{k_0/\sqrt t}^{2 k_0}\left| \frac{\partial_k \left[ |t^{q_\eps}(k)|^2 F(x;k)\right]}{k} \right| \dd k .
 \end{align*} 
 Note that 
 \[ \sqrt{t}\ \int_{k_0/\sqrt t}^{2 k_0}\ \frac{|\chi'(k\sqrt{t})|}{k} \dd k = \sqrt{t}\ \int_{k_0}^{2k_0\sqrt{t}} \frac{|\chi'(z)|}{z}\ dz\ = \mathcal{O}(\sqrt{t}), \]
 since $\chi'(z)$ vanishes near $0$ and is of compact support. Therefore,
 \begin{align}
 \left|\psi_{\rm low}^{(ii)}(x,t)\right|\ &\leq \ \frac{C(1+k_0^{-1})}{t^{1/2}} \sup_{k\in\RR} \left| F(x,k)\right|
 \ +\ \frac{C}{t} \int_{k_0/\sqrt t}^{2 k_0}\left|\ \frac{\partial_k \left[ |t^{q_\eps}(k)|^2 F(x;k)\right]}{k}\ \right| \dd k.
 \label{psi-low2-est1}
 \end{align}
 
 The estimates~\eqref{low1} and~\eqref{psi-low2-est1} are bounded thanks to uniform (in $\eps$) control of $t^{q_\eps}(k)$, $F(x;k)$ and their $k-$derivatives, which are given in~\eqref{eq:est-F} and Lemma~\ref{t-bound}, below. It follows then from~\eqref{decompose1-low} that 
\begin{equation}\label{uniform-estimate}  \big\vert (1+|x|)^{-3} \psi_{\rm low}(x,t) \big\vert \ \le\ C(\big\bracevert\hspace{-6pt}\big\bracevert V \big\bracevert\hspace{-6pt}\big\bracevert)\ \frac{1}{t^{1/2}}\ \big|\psi_0\big|_{\mathcal{L}^1_3}.\end{equation}

We now refine~\eqref{uniform-estimate} by carefully considering the $\eps$- dependence for $\eps$ small at $t\gg1$. In order to achieve a $\O(t^{-3/2})$ estimate, we first integrate by parts:
 \[
 \psi_{\rm low} \ = \ \frac{-1}{4\pi i t} \int_0^\infty e^{-ik^2t} \partial_k\left(\chi(k)k^{-1} |t^{q_\eps}(k)|^2 F(x;k) \right)\dd k\ \equiv \ \frac{-1}{4\pi i t} \int_0^\infty e^{-ik^2t} G(x;k) \dd k.
\]
 Note again, as above, that there are no boundary contributions from $k=\infty$ or, for $\eps$ small, from $k=0$, by genericity of $q_\eps$.
 We now decompose $\psi_{\rm low}$ further into contributions from frequencies in the ranges:
 $
 0\le k\le \frac{k_0}{\sqrt{t}}\ \ {\rm and}\ \ \frac{k_0}{\sqrt{t}}\le k\ \le 2k_0. 
 $
 In terms of the cutoff function, $\chi$, we have:
 \begin{align}
 \psi_{\rm low} \ 
 &=\ \frac{-1}{4\pi i t} \int_0^\infty \chi(k\sqrt t) e^{-ik^2t} G(x;k) \dd k \ + \ \frac{-1}{4\pi i t} \int_0^\infty (1-\chi(k\sqrt t)) e^{-ik^2t} G(x;k) \dd k \label{decompose-low}\\
 &=\ \psi_{\rm low}^{(1)}(x,t)\ +\ \psi_{\rm low}^{(2)}(x,t)\nn
\end{align}
Estimation of $\psi_{\rm low}^{(1)}$ gives:
\begin{equation}\label{low}
\left| \psi_{\rm low}^{(1)}(x,t) \right| \le \frac{1}{4\pi t} \int_0^{2k_0/\sqrt t} \left| G(x;k)\right| \dd k \ \leq \ \frac{k_0}{\pi }\ \frac{1}{ t^{3/2}}\ \sup_{k\in\RR} \left| G(x;k)\right| .\end{equation}

To estimate $\psi_{\rm low}^{(2)}$, we subject it to one further integration by parts:
 \[ \psi_{\rm low}^{(2)}(x,t)\ = \ \frac{1}{4\pi t^2}\ \int_0^\infty e^{-ik^2t} \frac{\partial}{\partial k}\left[\frac{1-\chi(k\sqrt t)}{2ik} G(x;k)\right] \dd k.\]
 Since $G(x;k)\equiv0$ for $k\geq 2 k_0$, it follows that
 \begin{align*}
\left|\psi_{\rm low}^{(2)}(x,t)\right|\ &\leq \ \frac{C}{ t^2}\ \int_{k_0/\sqrt t}^{2 k_0} \left|\ G(x;k) \frac{\partial}{\partial k}\left[\frac{1-\chi(k\sqrt t)}{2ik} \right]\ \right| \ + \ \left| \frac{\partial_k G(x;k)}{k}\right| \dd k \nn \\
 & \leq \ \frac{C}{t^{2}} \sup_{k\in\RR} \left|\ G(x;k)\right| \int_{k_0/\sqrt t}^{2 k_0}\
 \sqrt{t}\ \frac{|\chi'(k\sqrt{t})|}{k}+\frac{1}{k^2}\dd k \ + \ \frac{C}{t^2} \int_{k_0/\sqrt t}^{2 k_0}\left|\ \frac{\partial_k G(x;k)}{k}\ \right| \dd k
 \end{align*}
 
 Note again that 
 \[ \sqrt{t}\ \int_{k_0/\sqrt t}^{2 k_0}\ \frac{|\chi'(k\sqrt{t})|}{k} \dd k = \sqrt{t}\ \int_{k_0}^{2k_0\sqrt{t}} \frac{|\chi'(z)|}{z}\ dz\ = \mathcal{O}(\sqrt{t}), \]
 since $\chi'(z)$ vanishes near $0$ and is of compact support. Therefore,
 \begin{align}
 \left|\psi_{\rm low}^{(2)}(x,t)\right|\ &\leq \ \frac{C(1+k_0^{-1})}{t^{3/2}} \sup_{k\in\RR} \left|\ G(x;k)\right|
 \ +\ \frac{C}{t^2} \int_{k_0/\sqrt t}^{2 k_0}\left|\ \frac{\partial_k G(x;k)}{k}\ \right| \dd k
 \label{psi-low2-est}
 \end{align}

 We now use the following two bounds, proved below, to complete our estimation of 
 $\psi_{\rm low}^{(1)}(x,t)$ and $\psi_{\rm low}^{(2)}(x,t)$:
 \begin{gather}
 \left| G(x;k)\right|  \leq  C(\big\bracevert\hspace{-6pt}\big\bracevert V \big\bracevert\hspace{-6pt}\big\bracevert)\frac{1+|x|^2}{k^2+\eps^4(\int\Lambda_{\rm eff})^2}  \leq  C(\big\bracevert\hspace{-6pt}\big\bracevert V \big\bracevert\hspace{-6pt}\big\bracevert)\frac{1+|x|^2}{\eps^4(\int\Lambda_{\rm eff})^2}\big\vert\psi_0\big\vert_{\mathcal{L}^1_{2}}, \label{est-F}\\
 \left|\partial_k G(x;k)\right| \ \leq \ C(\big\bracevert\hspace{-6pt}\big\bracevert V \big\bracevert\hspace{-6pt}\big\bracevert)\frac{1+|x|^3}{k(k^2+\eps^4(\int\Lambda_{\rm eff})^2)}\big\vert\psi_0\big\vert_{\mathcal{L}^1_{3}}.\label{est-dF}
 \end{gather}
 
Using these bounds in~\eqref{low} and~\eqref{psi-low2-est}, we obtain:
\begin{align}\label{low-ld-first}
&(1+|x|^2)^{-1}\ \left| \psi^{(1)}_{\rm low}(x,t) \right|\ \le\ C(\big\bracevert\hspace{-6pt}\big\bracevert V \big\bracevert\hspace{-6pt}\big\bracevert)\ t^{-{3/2}}\ \frac{1}{\eps^4\left(\int_\mathbb{R}\Lambda_{\rm eff}\right)^2}\ \left|\psi_0\right|_{\mathcal{L}^1_2} \ ;
\end{align}
and
\begin{align}
(1+|x|^3)^{-1}\ \left| \psi^{(2)}_{\rm low}(x,t) \right|\ &\le\ C(\big\bracevert\hspace{-6pt}\big\bracevert V \big\bracevert\hspace{-6pt}\big\bracevert)\ t^{-2}\ \int_{k_0/\sqrt{t}}^{2k_0}\frac{1}{k^2(k^2+\eps^4\left(\int\Lambda_{\rm eff}\right)^2 }\ \dd k \left|\psi_0\right|_{\mathcal{L}^1_3}\nn\\
&\le\ C(\big\bracevert\hspace{-6pt}\big\bracevert V \big\bracevert\hspace{-6pt}\big\bracevert)\ 
\frac{1}{k_0\ t^{1/2}}\ \int_1^{2\sqrt{t}}\ \frac{1}{l^2}\ \frac{dl}{k_0^2 l^2+
 \eps^4\left(\int\Lambda_{\rm eff}\right)^2\ t}\
 \left|\psi_0\right|_{\mathcal{L}^1_3}\nn\\
&\le\ C(\big\bracevert\hspace{-6pt}\big\bracevert V \big\bracevert\hspace{-6pt}\big\bracevert)\ 
\frac{1}{k_0\ t^{1/2}}\ \frac{1}{k_0^2+ \eps^4\left(\int\Lambda_{\rm eff}\right)^2\ t}\ \int_1^{2\sqrt{t}}\ \frac{1}{l^2}\ \dd l
\ \left|\psi_0\right|_{\mathcal{L}^1_3}\nn\\
&\le\ C(\big\bracevert\hspace{-6pt}\big\bracevert V \big\bracevert\hspace{-6pt}\big\bracevert)\ 
\frac{1}{k_0\ t^{1/2}}\ \frac{1}{k_0^2+ \eps^4\left(\int\Lambda_{\rm eff}\right)^2\ t}\ \left|\psi_0\right|_{\mathcal{L}^1_3} \ .
\label{low-ld-final}\end{align}

Finally, one has from~\eqref{decompose-low},~\eqref{low-ld-first} and~\eqref{low-ld-final} the estimate 
\begin{equation}\label{nonuniform-estimate} \big\vert (1+|x|)^{-3} \psi_{\rm low}(x,t) \big\vert \ \le\ C(\big\bracevert\hspace{-6pt}\big\bracevert V \big\bracevert\hspace{-6pt}\big\bracevert)\ \frac{t^{-3/2}}{\eps^4\left(\int\Lambda_{\rm eff}\right)^2}\ \big|\psi_0\big|_{\mathcal{L}^1_3}.\end{equation}
Theorem~\ref{cor:schrod} is a consequence of~\eqref{high-est},~\eqref{uniform-estimate} and~\eqref{nonuniform-estimate}.

We conclude the proof by establishing~\eqref{est-F}-\eqref{est-dF}. This requires sharp estimates on the transmission coefficient and the Jost solutions, as well as their derivatives. These estimates are given in Lemmata~3.6 and~3.9 of~\cite{ArtbazarYajima00} for any generic $V$ sufficiently decreasing at infinity. We shall adapt the estimates to $V_\eps\equiv V(x,x/\eps)$. 

The estimates concerning the Jost solutions are uniform with respect to $\eps$. In particular, one has from Lemma~3.6 of~\cite{ArtbazarYajima00}:
 \begin{align}
& \sup_{k\in\RR} \left| \partial_k^j \left(e^{-ikx} f_+^{V_\eps}(x;k)\right) \right| \ \leq \ C\big(\big| V_\eps \big|_{\L^1_3}\big)(1+\max(0,-x))^{j+1},\nn \\
 & \sup_{k\in\RR} \left| \partial_k^j \left(e^{ikx} f_-^{V_\eps}(x;k)\right) \right| \ \leq \ C\big(\big| V_\eps \big|_{\L^1_3}\big)(1+\max(0,x))^{j+1}, \ \ j=0,1,2. \label{Jost-derivatives}
 \end{align}
 Therefore, 
 \begin{equation}\label{eq:est-F}
|\partial_k^j F(x;k)| \ \leq \ C(\big| V_\eps \big|_{\L^1_3})(1+|x|^{j+1})\big\vert\psi_0\big\vert_{\mathcal{L}^1_{j+1}}, \qquad j=0,1,2.\end{equation}
Estimates~\eqref{est-F}-\eqref{est-dF} are now a direct consequence of the following Lemma, together with~\eqref{eq:est-F}.
\begin{Lemma} \label{t-bound} Let $V_\eps=V(x,x/\eps)$ satisfy Hypotheses {\bf (V')}, with ${q_{\rm av} \equiv 0}$. Then for $\eps$ small enough, one has
\[ \left| \partial_k^j t^{V_\eps}(k) \right| \ \leq \ C(\big\bracevert\hspace{-6pt}\big\bracevert V \big\bracevert\hspace{-6pt}\big\bracevert)\left|\frac{k^{1-j}}{k+\eps^2\int\Lambda_{\rm eff}}\right| \ ,\]
with $j=0,1,2$.
\end{Lemma}
\begin{proof}[Proof of the Lemma.]
The estimate for $j=0$ is a consequence of Corollary~\ref{cor:k-real} with the estimate~\eqref{eqnfork-app}. Estimates on the derivatives are obtained by deriving identity~\eqref{tsig-Isig} with respect to $k$. We recall
\[ t^{V_\eps}(k) \ =\ \frac{2ik}{2ik-I^{V_\eps}(k)}, \qquad \text{where} \quad I^{V_\eps}(k) \ \equiv \ \int_{-\infty}^\infty V_\eps(y)e^{-iky}f^{V_\eps}_+(y;k)\dd y ,\]
so that
\[ \partial_k t^{V_\eps}(k) \ =\ \frac{2i}{2ik-I^{V_\eps}(k)}  -  \frac{2ik(2i-\partial_k I^{V_\eps}(k))}{(2ik-I^{V_\eps}(k))^2} \ = \ \frac{t^{V_\eps}(k)}{k}  -  \frac{(t^{V_\eps}(k))^2(2i-\partial_k I^{V_\eps}(k))}{2ik}.\]
 Using~\eqref{Jost-derivatives}, one controls uniformly $\partial_k I^{V_\eps}(k)$, so that 
 \[ \left| \partial_k t^{V_\eps}(k) \right| \ \leq \ \ \frac{|t^{V_\eps}(k)|}{k}(1+C|t^{V_\eps}(k)|) \ \leq \ C(\big\bracevert\hspace{-6pt}\big\bracevert V \big\bracevert\hspace{-6pt}\big\bracevert)\left|\frac{1}{k+\eps^2\int\Lambda_{\rm eff}}\right|.\]
 The second derivative in $k$ follows in the same way. \end{proof}

\section{The effective potential, $\seff(x)$; proof of Theorem~\ref{thm:t-conv} }\label{sec:hard}

As discussed in the introduction, for small $|k|$, $t^{q_{\rm av}+q_\eps}(k)$ is \underline{not} uniformly approximated by the transmission coefficient of the homogenized (averaged) potential $q^{\rm av}(x)=\int_0^1V(x,y) \dd y$, for $\eps$ small.
 In this section we prove for $k$ bounded that a uniform approximation can be achieved comparing $t^{q_{\rm av}+q_\eps}(k)$ to the transmission coefficient of an appropriate {\it effective potential well}:
 \begin{align}
 V^{\rm eff}_\eps(x)\ &=\ q_{\rm av}(x)\ +\ \seff(x),\ \ {\rm where}\ 
 \nn\\
 \seff(x)\ &\equiv\ -\eps^2\ \Lambda_{\rm eff}(x)\ \equiv\ -\frac{\eps^2}{(2\pi)^2} \sum_{j\neq 0} \frac{|q_j(x)|^2}{{\lambda_j}^2}.
 \label{seff-def1}
 \end{align}
 
The point of departure for the analysis is the identity~\eqref{IVW-def}, with the choices 
$V=q_{\rm av}+q_\eps$ and $W=q_{\rm av}+\sigma$:
\begin{align}
&\frac{k}{t^{q_{\rm av}+q_\eps}(k)} \ - \ \frac{k}{t^{q_{\rm av}+\sigma}(k)}\ =\ \ 
\frac{i}{2}\ I^{[q_{\rm av}+q_\eps,q_{\rm av}+\sigma]}(k),\quad \text{with }\label{trans-diff}\\
& \ I^{[q_{\rm av}+q_\eps,q_{\rm av}+\sigma]}(k) \equiv \int_{-\infty}^\infty f^{q_{\rm av}+\sigma}_-(y;k)\left(\ q_\eps(y)-\sigma(y)\right)\ f^{q_{\rm av}+q_\eps}_+(y;k)\dd y.
\label{Iqs}
\end{align}
 Here, $\sigma(x)$ is unspecified and to be chosen so that $I^{[q_{\rm av}+q_\eps,q_{\rm av}+\sigma]}$ is sufficiently high order in $\eps$. 
 The main step in the proof is:
 
\begin{Proposition}\label{prop:k-small}
Let $V_\eps \ \equiv \ q_{\rm av}(x)+q(x,x/\eps)$ satisfy Hypotheses {\bf (V)}, and $k\in K$ satisfy Hypotheses {\bf (K)}.
Define the effective potential $\seff \in L^\infty_\beta$, by the expression in ~\eqref{seff-def1}. Then, there exists $\eps_0>0$ such that the following bound holds uniformly for $(\eps,k)\in[0,\eps_0)\times K$:
\begin{equation}
I^{[q_{\rm av}+\seff,q_{\rm av}+q_\eps]}(k)\ \leq \ \eps^3 \ C\left(\big\bracevert V \big\bracevert,\sup_{k\in K}|k|\right) \ \max\left(1 , \sup_{k\in K}|t^{q_{\rm av}}(k)|\right)
\label{Isw-bound}
\end{equation}

\end{Proposition}

Theorem~\ref{thm:t-conv} is then a consequence of the bound~\eqref{Isw-bound}, applied to the right hand side of~\eqref{trans-diff}.
 We now turn to derivation of the {\em effective potential well} $\seff$, and the proof of Proposition~\ref{prop:k-small}.

\subsection{The heart of the matter; derivation of the effective potential well, $\seff(x)$, 
 and the proof of Proposition~\ref{prop:k-small} }\label{sec:potential-well}

To prove Proposition~\ref{prop:k-small} we need to bound $I^{[q_{\rm av}+\seff,q_{\rm av}+q_\eps]}$,
given by the integral expression in~\eqref{Iqs}. 
 We seek a decomposition of the integrand into oscillatory and non-oscillatory terms.
 Oscillatory terms can be integrated by parts to obtain bounds of high order in $\eps$. Non-oscillatory terms are removed by appropriate choice of~$\sigma(x)$.

We begin with $f^{q_{\rm av}+q_\eps }_+$. Using the Volterra equation~\eqref{eq:Volterra1} with $V=q_{\rm av}+q_\eps$ and $W=q_{\rm av}$, one has 
\begin{align}
f^{q_{\rm av}+q_\eps }_+(x;k) \ 
 &= \ f^{q_{\rm av}}_+(x;k) \ + \ J[q_{\rm av},q_\eps](x;k)\ ,\label{eq:Volterra1b}
\end{align}
where 
\begin{equation}
\label{eq:defI} J[q_{\rm av},q_\eps](\zeta;k) \ \equiv \ \int_\zeta^\infty q_\eps(y) \frac{f^{q_{\rm av}}_+(\zeta;k)f^{q_{\rm av}}_-(y;k)-f^{q_{\rm av}}_-(\zeta;k)f^{q_{\rm av}}_+(y;k)}{\mathcal{W}[f^{q_{\rm av}}_+ ,f^{q_{\rm av}}_-]} f^{q_{\rm av}+q_\eps}_+(y;k) \dd y,
\end{equation}
Therefore,
\[
\left(\ q_\eps(\zeta) -\sigma(\zeta)\ \right)\ f^{q_{\rm av}+q_\eps}_+(\zeta;k) \ = \ q_\eps(\zeta) f^{q_{\rm av}}_+(\zeta;k) \ - \ \sigma(\zeta) f^{q_{\rm av}+q_\eps}_+(\zeta;k) + q_\eps(\zeta) J[q_{\rm av},q_\eps](\zeta;k) ,
\]
implying that $I^{[q_{\rm av}+\sigma,q_{\rm av}+q_\eps]}$, given by~\eqref{Iqs}, can be written as
\begin{equation}\label{eq:Iqsigma2}
I^{[q_{\rm av}+\sigma,q_{\rm av}+q_\eps]} \ = \ \int_{-\infty}^\infty f^{q_{\rm av}+\sigma}_-(\zeta;k)\Big(q_\eps(\zeta) f^{q_{\rm av}}_+(\zeta;k) \ - \ \sigma(\zeta) f^{q_{\rm av}+q_\eps}_+(\zeta;k) + q_\eps(\zeta) J[q_{\rm av},q_\eps](\zeta;k)\Big)\dd \zeta.
\end{equation}
 We next show that there exists a natural choice, $\sigma=\seff(x)=\mathcal{O}(\eps^2)$ such that the contribution of 
\[- \sigma(\zeta) f^{q_{\rm av}+q_\eps}_+(\zeta;k) + q_\eps(\zeta) J[q_{\rm av},q_\eps](\zeta;k)\]
 to the integral~\eqref{eq:Iqsigma2} is of order $\mathcal{O}(\eps^3)$, for $\eps$ sufficiently small.
\begin{Lemma}[Cancellation Lemma] \label{Lem:cancellation}
Let $V(x,y)$ satisfy Hypotheses {\bf (V)}, and $k\in K$ satisfy Hypotheses {\bf (K)}. Define
\begin{equation}
\seff(x) \ = \ -\frac{\eps^2}{(2\pi)^2} \sum_{j\neq 0} \frac{|q_j(x)|^2}{{\lambda_j}^2}\ =\ -\eps^2 \Lambda_{\rm eff}(x).
\label{sigma-def}
\end{equation}
Then, there exists $\eps_0>0$ and $C(V,K)=C(\big\bracevert V \big\bracevert,\sup_{k\in K}|k|)$ such that 
\begin{align*}
&-\seff(\zeta) f^{q_{\rm av}+q_\eps}_+(\zeta;k) + q_\eps(\zeta) J[q_{\rm av},q_\eps](\zeta;k) \\
& \qquad\qquad = \ \eps^2\ \sum_{j\neq0} \t q_j(\zeta)e^{2i\pi {\lambda_j}\zeta/\eps}\ +\ \eps^2\sum_{{\substack{j,l\neq0\\ j+l\ne0 } }} \t q_{j,l}(\zeta)e^{2i\pi {(\lambda_j+\lambda_l)}\zeta/\eps} \ + \ \eps^3 q_\eps(\zeta) R^\eps(\zeta;k),
\end{align*} 
where the following estimate holds for any ${(\eps,k)\in [0,\eps_0)\times K}$:
\begin{align*}
\sum_{{\substack{j,l\neq0\\ j+l\ne0 } }} \big( |\t q_{j,l}(\zeta) e^{\beta|\zeta|}|+|\t q_{j,l}'(\zeta) e^{\beta|\zeta|}|+|\t q_{j,l}''(\zeta) e^{\beta|\zeta|}|\big)  &\leq  C(V,K) ,  \\
|R^\eps(\zeta;k)| \ + \ \sum_{j\neq0}\big( |\t q_j(\zeta) e^{\beta|\zeta|}|+|\t q_j'(\zeta) e^{\beta|\zeta|}|+|\t q_j''(\zeta) e^{\beta|\zeta|}|\big)
&\leq  C(V,K) M_K (1+|\zeta|^2) e^{\alpha |\zeta|},& 
\end{align*} 
\end{Lemma}
for $\beta>2\alpha$. Therefore, one has
\begin{align}
 I^{[q_{\rm av}+\seff,q_{\rm av}+q_\eps]}(k) \ = \ \int_{-\infty}^\infty f^{q_{\rm av}+\seff}_-(\zeta;k) \Big(q_\eps(\zeta) 
f_+^{q_{\rm av}}+ \eps^2\sum_{j\neq0} \t q_j(\zeta)e^{2i\pi {\lambda_j}\zeta/\eps} &\nonumber \\
 \qquad +\eps^2\sum_{{\substack{j,l\neq0\\ j+l\ne0 } }} \t q_{j,l}(\zeta)e^{2i\pi {(\lambda_j+\lambda_l)}\zeta/\eps} + \eps^3 q_\eps(\zeta) R^\eps(\zeta;k)\Big) \dd y. & \label{eq:Iqsigma3} 
\end{align}
Lemma~\ref{Lem:cancellation} is proved in the next section. We first apply it to conclude the proof of
Theorem~\ref{thm:t-conv}. In succession, each term in~\eqref{eq:Iqsigma3} is controlled, for $k\in K$, by the bounds of the following: 
\begin{Lemma}\label{Lem:oscillations}
Let $V(x,y)$ satisfy Hypotheses {\bf (V)}, and $k\in K$ satisfy Hypotheses {\bf (K)}, then one has
\begin{align*}
 \left|\ \int_{-\infty}^\infty f^{q_{\rm av}+\seff}_-(\zeta;k)\ q_\eps(\zeta)\ f^{q_{\rm av}}_+(\zeta;k) \dd \zeta \ \right| \ &\leq \ \eps^3\ C(\big\bracevert V \big\bracevert,\sup_{k\in K}|k|),\\
 \sum_{j\neq0} \left|\ \int_{-\infty}^\infty f^{q_{\rm av}+\seff}_-(\zeta;k)\ \t q_j(\zeta)\ e^{2i\pi \lambda_j \zeta/\eps} \dd \zeta\ \right| \ &\leq \ \eps^2\ M_K\ C(\big\bracevert V \big\bracevert,\sup_{k\in K}|k|) ,\\
\sum_{\substack{j,l\neq0\\ j+l\ne0 } } \left|\ \int_{-\infty}^\infty f^{q_{\rm av}+\seff}_-(\zeta;k)\ \t q_{j,l}(\zeta)\ e^{2i\pi {(\lambda_j+\lambda_l)}\zeta/\eps} \dd \zeta\ \right| \ &\leq \ \eps^2\ C(\big\bracevert V \big\bracevert,\sup_{k\in K}|k|), \\ 
 \left|\ \int_{-\infty}^\infty f^{q_{\rm av}+\seff}_-(\zeta;k)\ q_\eps(\zeta)\ R^\eps(\zeta;k) \dd \zeta \ \right| \ &\leq \ M_K\ C(\big\bracevert V \big\bracevert,\sup_{k\in K}|k|)\ ,
\end{align*}
where $C(\big\bracevert V \big\bracevert,\sup_{k\in K}|k|)$ and $M_K=\max(1,\sup_{k\in K}|t^{q_{\rm av}}(k)|)$ are independent of $\eps\in [0,\eps_0)$.
\end{Lemma}

Applying Lemma~\ref{Lem:oscillations} to~\eqref{eq:Iqsigma3} yields the desired $\O(\eps^3)$ bound on $I^{[q_{\rm av}+\seff,q_{\rm av}+q_\eps]}(k)$. Proposition~\ref{prop:k-small} and therefore Theorem~\ref{thm:t-conv} follow. We now turn to the proofs of Lemmata~\ref{Lem:cancellation} and~\ref{Lem:oscillations}, in Sections~\ref{sec:ProofLemmaCancellation} and~\ref{sec:ProofLemmaOscillation}.

\subsection{Proof of Lemma~\ref{Lem:cancellation}}\label{sec:ProofLemmaCancellation}
For ease of presentation, we will use the simplified notation for the expression in~\eqref{eq:defI}:
 \begin{align} J[q_{\rm av},q_\eps](\zeta;k)  
 &\equiv \sum_{j\ne0}\ \int_\zeta^\infty \m(\zeta,y;k)\ q_j(y)\ e^{c{\lambda_j}y/\eps} \ f(y)\ \dd z,\label{Jrep}
 \end{align}
where  $c=2\pi i$, $f(y)=f^{q_{\rm av}+q_\eps}_+(y;k)$ and
\[ \m(\zeta,y;k) \ =\ \frac{f^{q_{\rm av}}_+(\zeta;k)f^{q_{\rm av}}_-(y;k)-f^{q_{\rm av}}_-(\zeta;k)f^{q_{\rm av}}_+(y;k)}{\mathcal{W}[f^{q_{\rm av}}_+ ,f^{q_{\rm av}}_-]}.\]
To make explicit the smallness of certain terms due to cancellations, we shall integrate by parts, keeping in mind that we do not control more than two derivatives of $f\equiv f^{q_{\rm av}+q_\eps}_+$. To evaluate boundary terms which arise, we shall use that 
\[\left.\{\m(\zeta,y;k),\partial_y \m(\zeta,y;k),\partial_y^2\m(\zeta,y;k)\}\right|_{y=\zeta}=\{0,1,0\}.\]

We now embark on the detailed expansion. From~\eqref{Jrep}, using integration by parts, one has
 \[ J[q_{\rm av},q_\eps](\zeta;k) \ \equiv \ \sum_j \Big(\frac\eps{c{\lambda_j}}\Big)^2\Big[q_j\ f\ e^{c{\lambda_j}\zeta/\eps} \ + \ \int_\zeta^\infty \partial_y^2(\m\ q_j\ f)\ e^{c{\lambda_j}y/\eps} \dd y\Big] .\]
 Decompose the integrand by using: $\partial_y^2(\m\ q_j\ f)=\partial_y^2(\m\ q_j)\ f  +  2\partial_y(\m\ q_j)\ \partial_y f  +  \m\ q_j\ \partial_y^2 f$.
The first two terms can be integrated by parts once more. This gives for~${j\ne0}$:
\begin{align*}
\int_\zeta^\infty \partial_y^2(\m\ q_j)\ f e^{c{\lambda_j}y/\eps} \dd y \ &= \ -\frac\eps{c{\lambda_j}}\int_\zeta^\infty \partial_y\big(\partial_y^2(\m\ q_j)\ f \big)\ e^{c{\lambda_j}y/\eps} \dd y \ - \ 2\frac\eps{c{\lambda_j}} q_j'(\zeta) f(\zeta) e^{c{\lambda_j}\zeta/\eps}, \\
\int_\zeta^\infty \partial_y(\m\ q_j)\ \partial_y f e^{c{\lambda_j}y/\eps} \dd y \ &= \ -\frac\eps{c{\lambda_j}}\int_\zeta^\infty \partial_y\big(\partial_y(\m\ q_j)\ \partial_y f \big)\ e^{c{\lambda_j}y/\eps} \dd y \ - \ \frac\eps{c{\lambda_j}} q_j(\zeta) \ f'(\zeta) e^{c{\lambda_j}\zeta/\eps}.
\end{align*}
As for the last term, we use the equation for the Jost solution, $f$, to express $\partial_y^2 f $ in terms of $f$:
$ \partial_y^2 f = \partial_y^2
 f^{ q_{\rm av}+q_\eps }_+ = (q_{\rm av}+q_\eps-k^2) f^{q_{\rm av}+q_\eps}_+$.
 Thus we eventually obtain:
\begin{align}\label{eq:devI}
 J[q_{\rm av},q_\eps](\zeta;k) \ &= \ \sum_{j\ne0} \big(\frac\eps{c{\lambda_j}}\big)^2\Big[q_j\ f \ e^{c{\lambda_j}\zeta/\eps} \ + \ \int_\zeta^\infty \m q_j (q_{\rm av}+q_\eps-k^2) f\ e^{c{\lambda_j}y/\eps} \dd y \nn\\
&\quad \ + \ \frac\eps{c{\lambda_j}}\big\{\sum_{l,m,n}c_{lmn}\int_\zeta^\infty \left(\ \partial^l \m\ \partial^m q_j\partial^n f\ \right) e^{c{\lambda_j}y/\eps} \dd y \ -2 \ ( q_j f)' \ e^{c{\lambda_j}\zeta/\eps}\big\} \Big] ,\end{align}
with $0\leq l,m\leq 3,\ 0\leq n\leq 2$, and $c_{lmn}\in\NN$. 

We now study each of the terms of~\eqref{eq:devI} separately, beginning with an $\mathcal{O}(\eps^3)$ bound on the curly bracket terms in~\eqref{eq:devI}. Using the estimates of Lemmata~\ref{Lem:f-gen} and~\ref{Lem:m}, one has for any $0\leq l,m\leq 3,\ 0\leq n\leq 2$,
%
\begin{align*} 
 \left|\partial_y^l \m(\zeta,y;k)\partial_y^m q_j(y) \partial_y^n f^{q_{\rm av}+q_\eps}_+(y;k)\right| \ \leq \ M_K\ C (1+|k|^l)\big(1+ |y-\zeta|(1+|y|)(1+|\zeta|) e^{\alpha|\zeta|}e^{\alpha|y|}\big) &\\
\times (1+|k|^n)(1+|y|)e^{\alpha|y|}\big\vert\partial_y^m q_j(y)\big\vert. &
\end{align*}
Therefore, the contribution to $J[q_{\rm av},q_\eps]$ of the sum over all integrals in curly brackets in~\eqref{eq:devI} is bounded by $\eps^3 M_K C\left(\big\bracevert V \big\bracevert,\sup_{k\in K}|k|\right)(1+|\zeta|)^2e^{\alpha|\zeta|}$,
%
uniformly for $k\in K$. The boundary term in curly brackets satisfy a similar bound. Its contribution is bounded by 
 $\eps^3 M_K C\left(\big\bracevert V \big\bracevert,\sup_{k\in K}|k|\right)$.

We now turn to the first two terms, in square brackets, of~\eqref{eq:devI}. Using the Fourier decomposition of $q_\eps(x)$,~\eqref{q-eps-per}, one sees that there are two types of terms: (a) terms where $\lambda_l=-\lambda_j$ ($l=-j$), $q_{-j} e^{-2i\pi {\lambda_j}y/\eps}$, where no oscillations remain due to phase-cancellation,
 and (b) contributions from terms where $\lambda_l+\lambda_j\ne0$, which are highly oscillatory for $\eps$ small.
In these latter terms, an additional factor of $\eps$ is gained via one more integration by parts.
 Precisely, one has 
\begin{multline*}
\int_\zeta^\infty \m q_j (q_{\rm av}+q_\eps-k^2) f\ e^{c{\lambda_j}y/\eps} \dd y \ = \ \int_\zeta^\infty \m q_j q_{-j} f\ \dd y \\  + \ \int_\zeta^\infty \m q_j f\ \big( (q_{\rm av}-k^2) e^{c{\lambda_j}y/\eps}+\sum_{l\notin \{0,-j\}}q_i e^{c{({\lambda_l}+{\lambda_j})}y/\eps} \big)\dd y.
\end{multline*}
The last terms can be integrated by parts; the resulting integral and boundary terms are estimated as above. Finally, recalling that $f=f^{q_{\rm av}+q_\eps}$, we obtain
 \begin{align}\label{eq:Ij2}
 J[q_{\rm av},q_\eps](\zeta;k) \ &= \ \sum_{j\ne0} \big(\frac\eps{c{\lambda_j}}\big)^2\left[q_j\ f^{q_{\rm av}+q_\eps}(\zeta;k)\ e^{c{\lambda_j}\zeta/\eps} \right. \nn\\
 & + \left. \int_\zeta^\infty \m(\zeta,y;k)\ q_j(y) q_{-j}(y) f^{q_{\rm av}+q_\eps}(y;k) \dd y\right]\ 
 + \ \eps^3 R^\eps(\zeta;k),
 \end{align}
 with $\big\vert R^\eps(\zeta;k)\big\vert \ \leq \ M_K\ C\Big(\big|q\big|_{W^{3,\infty}_\beta},\sup_{k\in K}|k|\Big)\ (1+|\zeta|^2)\ e^{\alpha|\zeta|}$. 
 
Now multiply~\eqref{eq:Ij2} by $q_\eps(\zeta)=\sum_{l\ne0}q_l(\zeta)\exp(2\pi i {\lambda_l}\zeta/\eps)$ and then add the result to $-\sigma f_+^{q_{\rm av}+q_\eps}$ to obtain (decomposing again into non-oscillatory and highly oscillatory terms and using the notation $c=2\pi i$):
\begin{align}
&-\sigma(\zeta) f^{q_{\rm av}+q_\eps}_+(\zeta;k) + q_\eps(\zeta)\ J[q_{\rm av},q_\eps](\zeta;k)\label{cancel} \\ 
&\quad\quad= \left( -\sigma(\zeta)\ +\ \sum_{j\ne0} \Big(\frac{\eps}{c{\lambda_j}}\Big)^2\ q_j(\zeta)q_{-j}(\zeta)\ \right)\ f^{q_{\rm av}+q_\eps}_+(\zeta;k) \nn\\
&\quad\qquad +\ \sum_{l\notin \{0,-j\}}\sum_{j\ne0}\ \left(\frac{\eps}{c{\lambda_j}}\right)^2\ \left[ q_lq_j e^{c({\lambda_l}+{\lambda_j})\zeta/\eps} f^{q_{\rm av}+q_\eps}_+\right] \nn \\
&\quad\qquad +\ 
\sum_{l\neq0}\sum_{j\ne0}\left(\frac{\eps}{c{\lambda_j}}\right)^2 \left[ q_le^{c{\lambda_l}\zeta/\eps}\int_\zeta^\infty \m(\zeta,y)\ q_j(y) q_{-j}(y) f^{q_{\rm av}+q_\eps}(y;k) \dd y
\right]\nn \\
&\quad\qquad  + \ \eps^3 q_\eps(\zeta) R^\eps(\zeta;k) .
\nn\end{align}
The first term on the right hand side of~\eqref{cancel} is non-oscillatory in $\zeta$ for small~$\eps$.
We remove it by choosing 
\begin{equation}\label{eq:defsigma}
\sigma(\zeta) \ =\ \seff(\zeta)\ \equiv \ \sum_{j\neq 0}\ \Big(\frac{\eps}{2i\pi{\lambda_j}}\Big)^2 q_{-j}(\zeta) q_j(\zeta) \ = \ -\frac{\eps^2 }{4\pi^2}\sum_{j\neq 0}\frac{|q_{j}(\zeta)|^2}{{\lambda_j}^2}\ .
\end{equation}
Then 
\begin{multline*}
-\seff(\zeta) f^{q_{\rm av}+q_\eps}_+(\zeta;k) + q_\eps(\zeta) J[q_{\rm av},q_\eps](\zeta;k) \\
 \ = \ \eps^2\ \sum_{l\neq0} \t q_l(\zeta)e^{2i\pi {\lambda_l}\zeta/\eps}\ +\ \eps^2\sum_{{\substack{j,l\neq0\\ j+l\ne0 } }} \t q_{j,l}(\zeta)e^{2i\pi {(\lambda_j+\lambda_l)}\zeta/\eps} \ + \ \eps^3 q_\eps(\zeta) R^\eps(\zeta;k),\end{multline*}
which we've written in the form of the statement of Lemma~\ref{Lem:cancellation}. 
 Here, $\t q_j(\zeta)$ and $\t q_{j,l}(\zeta)$ are given by
\begin{align}
\t q_l(\zeta) \ & \equiv \ q_l(\zeta)\sum_{j\neq 0}\Big(\frac{1}{2i\pi {\lambda_j}}\Big)^2\int_\zeta^\infty \m(\zeta,y;k) q_j q_{-j}(y) f^{q_{\rm av}+q_\eps}_+(y;k) \dd y,
\label{first}\\
\t q_{j,l}(\zeta) \ &\equiv \ \Big(\frac{1}{2i\pi {\lambda_j}}\Big)^2q_{l}(\zeta)q_{j}(\zeta) 
f^{q_{\rm av}+q_\eps}_+(\zeta;k). \label{second}
\end{align}

To conclude, we verify the necessary estimates on $\t q_j$ and $\t q_{j,l}(\zeta)$, and their first and second derivatives. 

As for~\eqref{first}, we use Lemmata~\ref{Lem:f-gen} and~\ref{Lem:m}, and obtain
\[\left| \int_\zeta^\infty \m(\zeta,y;k) q_j q_{-j}(y) f^{q_{\rm av}+q_\eps}_+(y;k) \dd y \right| \ \leq \ M_K C(\big\bracevert V \big\bracevert,\sup_{k\in K}|k|)(1+|\zeta|^2)e^{\alpha|\zeta|}.\]
For the derivatives, we use
\begin{align*}
 \partial_\zeta \int_\zeta^\infty\!\! \m(\zeta,y;k) q_j q_{-j}(y) f^{q_{\rm av}+q_\eps}_+(y;k) \dd y & =  \int_\zeta^\infty\!\! \partial_\zeta^2 \m(\zeta,y;k) q_j q_{-j}(y) f^{q_{\rm av}+q_\eps}_+(y;k) \dd y , \\
\partial_\zeta^2 \int_\zeta^\infty \m(\zeta,y;k) q_j q_{-j}(y) f^{q_{\rm av}+q_\eps}_+(y;k) \dd y & 
 =  \int_\zeta^\infty\!\! \partial_\zeta^2\m(\zeta,y;k) q_j q_{-j}(y) f^{q_{\rm av}+q_\eps}_+(y;k) \dd y \\
 & \quad -q_j q_{-j}(\zeta) f^{q_{\rm av}+q_\eps}_+(\zeta;k) ,
\end{align*}
so that the integrals are uniformly bounded in the same way.
As these objects are multiplied by $q_l,\ q_l'$ or $q_l''$, and since $q_l\in W^{2,\infty}_{\beta}$, it follows
\[|\t q_l(\zeta) e^{\beta|\zeta|}|+|\t q_l'(\zeta) e^{\beta|\zeta|}|+|\t q_l''(\zeta) e^{\beta|\zeta|}| \ \leq \ M_K C\big(|q_l|_{W^{2,\infty}_\beta},\sup_{k\in K}|k|\big)(1+|\zeta|^2) e^{\alpha |\zeta|}, \]
uniformly for $k\in K$.

As for~\eqref{second}, one has 
\begin{align*} \left|q_{l}(\zeta)q_{j}(\zeta) f^{q_{\rm av}+q_\eps}_+(\zeta;k)\right| &\leq  |q_{l}(\zeta)||q_{j}f^{q_{\rm av}+q_\eps}_+(\zeta;k)|  \leq  e^{-\beta|\zeta|}|q_{l}|_{L^\infty_\beta} |q_{j}f^{q_{\rm av}+q_\eps}_+(\cdot;k)|_{L^\infty} \\ & \leq  C(\big\bracevert V \big\bracevert,\sup_{k\in K}|k|)|q_{j}|_{L^\infty_\beta} |q_{l}|_{L^\infty_\beta} e^{-\beta|\zeta|},
\end{align*}
where we used Lemma~\ref{Lem:f-gen} to estimate $f^{q_{\rm av}+q_\eps}_+$. The first and second derivatives are bounded in the same way, and the double series converge.

 This concludes the proof of the Cancellation Lemma~\ref{Lem:cancellation}.

\subsection{Proof of Lemma~\ref{Lem:oscillations}}\label{sec:ProofLemmaOscillation}

The last estimate of Lemma~\ref{Lem:oscillations} follows from bounds on $R^\eps$ (see Lemma~\ref{Lem:cancellation}) and $f^{q_{\rm av}+\seff}_-(y;k)$ (see Lemma~\ref{Lem:f-gen}), and the decay Hypotheses {\bf (V)} on $q_\eps$. One has 
\begin{align*}
&\Big| \int_{-\infty}^\infty \ f_-^{q_{\rm av}+\seff}(y;k) q_\eps(y) R^\eps(y;k)\ \dd y \Big| \\
 &\le\ M_K \ C(\big\bracevert V \big\bracevert,\sup_{k\in K}|k|)\ \int_{-\infty}^\infty (1+|y|)^3 e^{2\alpha|y|}|q_\eps(y)|\ \dd y
\ \ \le\ M_K\ C(\big\bracevert V \big\bracevert,\sup_{k\in K}|k|\big).
\end{align*}

To prove the $\eps^2$-smallness of the second estimate of Lemma~\ref{Lem:oscillations}, we integrate by parts: 
\[ \int_{-\infty}^\infty f^{q_{\rm av}+\seff}_-(y;k) \t q_j e^{2i\pi {\lambda_j}/\eps} \dd y = \ \left(\frac{\eps}{2i\pi {\lambda_j}}\right)^2\int_{-\infty}^\infty (f^{q_{\rm av}+\seff}_-(\cdot ;k)\t q_j)''(y)e^{2i\pi {\lambda_j}y/\eps}\dd y.
\]
The estimate follows as previously from the bounds on $\t q_j$ (Lemma~\ref{Lem:cancellation}) and the ones on $f^{q_{\rm av}+\seff}_-(y;k)$ (Lemma~\ref{Lem:f-gen}), as well as the hypotheses on $\lambda_j$:~\eqref{hyp-lambda} in Hypotheses {\bf (V)}.

The third estimate follows as previously, as 
\begin{multline*}
 \int_{-\infty}^\infty f^{q_{\rm av}+\seff}_-(y;k) \t q_{j,l} e^{2i\pi ({\lambda_j+\lambda_l})/\eps} \dd y\\
= \ \left(\frac{\eps}{2i\pi ({\lambda_j+\lambda_l})}\right)^2\int_{-\infty}^\infty (f^{q_{\rm av}+\seff}_-(\cdot ;k)\t q_{j,l})''(y)e^{2i\pi {\lambda_j}y/\eps}\dd y.
\end{multline*}
 The estimate follows, using now the bounds on $\t q_{j,l}$ (Lemma~\ref{Lem:cancellation}).
Finally, we use three integration by parts for the first estimate of Lemma~\ref{Lem:oscillations}:
\begin{multline*}
\int_{-\infty}^\infty f^{q_{\rm av}+\seff}_-(y;k) q_j(y) f^{q_{\rm av}}_+(\cdot ;k) e^{\frac{2i\pi {\lambda_j}}{\eps}} \dd y\\
= 
\left(\frac{i\eps}{2\pi {\lambda_j}}\right)^3\int_{-\infty}^\infty (f^{q_{\rm av}+\seff}_-(\cdot ;k)q_j f^{q_{\rm av}}_+(\cdot ;k))'''(y)e^{\frac{2i\pi {\lambda_j}y}{\eps}}\dd y,\end{multline*}
which is estimated using the third item of Lemma~\ref{Lem:f-gen}, and Hypotheses {\bf (V)}.

\appendix

\section{Some useful estimates used throughout the paper}\label{sec:jost-tools}
We recall that the Jost solution is defined through the Volterra equation
\begin{equation}\label{eq:Volterra-in-appendix}
f^{V}_+(x;k) -e^{ikx} \ = \ \int_x^\infty \frac{\sin(k(y-x))}{2ik}V(y)f^{V}_+(y;k) \dd y.
\end{equation}
 A detailed discussion of Jost solutions, $f_\pm(x;k)$, applying to $\Im(k)\ge0$ can be found in~\cite{DeiftTrubowitz79}, where it is assumed that $V\in \mathcal{L}^1_{2}$. We present in the following Lemma the results holding when $k\in \RR$, and deal with the analytic continuation in a complex strip around the real axis afterwards.
 \begin{Lemma} \label{Lem:f-gen-k-real}
If $k\in\RR$ and $V\in \mathcal{L}^1_{2}$, then one has
\begin{align}
| f^V_\pm(x;k) | \ &\leq \ C(1+|k|)^{-1}(1+|x|),\label{fplus-allx-k-real} \\
|\partial_x f^V_\pm(x;k)| \ &\leq \ C\frac{1+|k|(1+|x|)}{1+|k|} \ \leq \ C(1+|x|),\label{Dfplus-allx-k-real} \\ 
|\partial_x^2 f^V_\pm(x;k)| \ &\leq \ |V(x)-k^2| |f^V_+(x;k)| \ \leq \ \ C(1+|k|)(1+|x|),\label{D2fplus-allx-k-real}
\end{align}
where $C=C\big(\big|V\big|_{\mathcal{L}^1_{2}}\big)$. Moreover, if $\partial_x V\in \mathcal{L}^1_{2}$, then 
\[ \left| \partial_x^3 f^V_\pm(x;k) \right| \ \leq \ C(1+|k|^2)(1+|x|),
\ \ 
{\rm with}\ C=C(\big|V\big|_{\mathcal{W}^{1,1}_2}\big).\]


 \end{Lemma}
 \begin{proof}
As for the first two estimates, equivalent bounds are given in~\cite{DeiftTrubowitz79}, Lemma 1, for the function $m_\pm(x;k) \ \equiv \ f_\pm(x;k)e^{\pm ikx}$ . The results for $f_\pm(x;k)$ follow straightforwardly. The last two estimates are a direct consequence of~\eqref{eq:Volterra-in-appendix}.
 \end{proof}

 If $e^{2\alpha|x|}V\in L^1$, then $f_\pm(x;k)$ has an analytic continuation to $\Im(k) > -\alpha$. Some results are presented in~\cite{ReedSimonIII}.
In this section we review and obtain the required extensions of these results. In order to simplify the results, we also restrict $k$ to the complex strip $|\Im(k)| < \alpha$.
\begin{Lemma}\label{Lem:f-gen}
If $|\Im(k)|<\alpha$ and $V\in L^\infty_\beta$, with $\beta>2\alpha\geq0$, then one has
\begin{align}
| f^V_\pm(x;k) | \ &\leq \ C (1+|x|)e^{\alpha |x|},\label{fplus-allx} \\
|\partial_x f^V_\pm(x;k)| \ &\leq \ C(1+|k|)(1+|x|)e^{\alpha |x|},\label{Dfplus-allx} \\ 
|\partial_x^2 f^V_\pm(x;k)| \ &\leq \ |V(x)-k^2| |f^V_+(x;k)| \ \leq \ \ C(1+|k|^2)(1+|x|)e^{\alpha|x|},\label{D2fplus-allx}
\end{align}
where $C=C\big(\big|V\big|_{L^\infty_\beta}\big)$. Moreover, if $V\in W^{1,\infty}_\beta$, then 
\[ \left| \partial_x^3 f^V_\pm(x;k) \right| \ \leq \ C(1+|k|^3)(1+|x|)e^{\alpha |x|},\ \ 
{\rm with}\ C=C(\big|V\big|_{W^{1,\infty}_\beta}\big)\ .\]
\end{Lemma}
\begin{proof} We prove bounds for $f_+^V$. Analogous bounds $f^V_-(x;k)$ are similarly proved and are obtained from the above by replacing $x$ by $-x$, and $x\ge0$ by $-x\ge0$ {\em etc}.

The estimates follow from the Volterra equation~\eqref{eq:Volterra-in-appendix} satisfied by the Jost solutions, and make use of the following bounds: for $k\in\CC$, and for $y\geq x$, one has
\begin{gather}
|\cos(k(y-x))| + |\sin(k(y-x))|\  \leq \ C e^{|\Im(k)|(y-x)} \ \leq \ C e^{\alpha|x|}e^{\alpha |y|} \ ,\\
\label{eq:sin} \frac{|\sin(k(y-x))|}{|k|}\  \leq \ C\frac{y-x}{1+|k|(y-x)}e^{|\Im(k)|(y-x)} \ \leq \ C(y-x)e^{\alpha|x|}e^{\alpha |y|}\ .
\end{gather}
By Theorem XI.57 of~\cite{ReedSimonIII}, one deduces from a careful study of the iterates of the Volterra equation~\eqref{eq:Volterra-in-appendix}, that for $x\geq 0$, one has
\begin{equation}
|f^V_+(x;k)-e^{ikx}| \ \leq \ e^{\alpha|x|}|e^{Q_k(x)}-1| \ \leq \ Ce^{\alpha|x|}, \label{eq:xpos}\
\end{equation}
with $Q_k(x) \ \equiv \ \int_x^\infty \frac{4y}{1+|k|y}\ |V(y)|\ e^{2\alpha|y|} \dd y$. Equation~\eqref{fplus-allx} follows for $x\geq 0$.

As for the case $x\leq 0$,~\eqref{eq:Volterra-in-appendix} yields
\begin{align*}
|f^V_+(x;k)| \ &= \ \left|e^{ikx}+ \int_x^\infty \frac{\sin(k(y-x))}{k}V(y)f^V_+(y;k) \dd y \right| \\
& \leq \ e^{\alpha|x|} \ + \ \int_x^\infty (y-x)e^{\alpha|x|}e^{\alpha |y|} |V(y)||f^V_+(y;k)| \dd y \\
& \leq \ e^{\alpha|x|}\big[1 \ + \ \int_0^\infty y e^{\alpha |y|}|V(y)||f^V_+(y;k)| \dd y \ + \ (-x)\int_x^\infty e^{\alpha |y|}|V(y)||f^V_+(y;k)| \dd y \big]\\
& \leq \ e^{\alpha|x|}\big[C_0 \ + \ (-x)\int_x^\infty e^{\alpha |y|}|V(y)||f^V_+(y;k)| \dd y \big].
\end{align*}
We used~\eqref{eq:sin} for the first inequality; the last inequality follows from~\eqref{eq:xpos}, with $x=0$.
 Therefore, one has with $g(x) \ \equiv\ \frac{|f^V_+(x;k)|}{(C_0+(-x)) e^{\alpha|x|}}$,
\[
|g(x)| \ \leq \ 1 \ + \ \int_x^\infty e^{\alpha|y|}|V(y)||g(y;k)|(C_0+(-y)) e^{\alpha|y|} \dd y\ .
\]
By Gronwall's inequality 
\[ g(x) \ \leq \ \exp\big(\int_x^\infty (C_0+(-y))e^{2\alpha|y|}|V(y)|\dd y\big) \ \leq \ C\big(\big|V\big|_{L^\infty_\beta}\big|\big).\]
Finally, one has
\[f(x;k) \ \leq \ C\big(\big|V\big|_{L^\infty_\beta}\big|\big) (C_0+(-x)) e^{\alpha|x|} \ \leq \ C(1+|x|)e^{\alpha |x|},\]
with $C=C\big(\big|V\big|_{L^\infty_\beta}\big|\big)$. This completes the proof of~\eqref{fplus-allx}.

The proof of~\eqref{Dfplus-allx} is similar, and obtained by differentiation and estimation of the Volterra integral equation~\eqref{eq:Volterra-in-appendix}. The bound~\eqref{D2fplus-allx} is a direct consequence of $\partial_x^2f^V_+=(V -k^2)f^V_+$ and the above bounds.
\end{proof}

\begin{Lemma}\label{Lem:m}
Let $q_{\rm av}\in W^{1,\infty}_\beta$ and $k\in K$, satisfy Hypotheses {\bf (K)}. Define 
\[\m(x,y;k) \ \equiv \ \frac{f^{q_{\rm av}}_+(x;k)f^{q_{\rm av}}_-(y;k)-f^{q_{\rm av}}_-(x;k)f^{q_{\rm av}}_+(y;k)}{W[f^{q_{\rm av}}_+ ,f^{q_{\rm av}}_-]}. \]
Then one has, for $0\leq l\leq 3$,
\begin{align}
| \partial_y^l \m(x,y;k) | + | \partial_x^l \m(x,y;k) | & \leq C\ M_K \ (1+|k|)^l\left(1+ |y-x|(1+|y|)(1+|x|)e^{\alpha |x|}e^{\alpha |y|}\right), \label{est-diffm} 
\end{align}
where $C=C\big(\big|q_{\rm av}\big|_{W^{1,\infty}_\beta}\big)$, and $M_K=\max(1,\sup_{k\in K}|t^{q_{\rm av}}(k)|)<\infty$.

Restricting to $k\in\RR$, and assuming only $q_{\rm av}\in \mathcal{W}^{1,1}_{2}$, one has for $0\leq l\leq 3$
\begin{align*}
 | \partial_y^l \m(x,y;k) | \ +\ | \partial_x^l \m(x,y;k) |\ & \leq \ C(1+|k|)^{l-2}\Big(1+ |y-x|(1+|y|)(1+|x|)\Big),
\end{align*}
where $C=C\big(\big|q_{\rm av}\big|_{\mathcal{W}^{1,1}_2}\big)$.
\end{Lemma}
\begin{proof} Let us start with the estimate~\eqref{est-diffm} when $l=0$. One can always assume that $y>x$, since $\m(x,y;k)=-\m(y,x;k)$.
Using Taylor's theorem with remainder in the integral form, one has
\[f^{q_{\rm av}}_\pm(y;k) \ = \ f^{q_{\rm av}}_\pm(x;k) \ + \ (y-x) \big(\partial_y f^{q_{\rm av}}_\pm(y;k)\big)\big\vert_{y=x} \ + \frac12\ \int_x^y \big(\partial_y^2 f^{q_{\rm av}}_\pm(y;k)\big)\big\vert_{y=t} (y-t) \dd t.\]
It follows that
\begin{align*}\m(x,y;k) \ &= \ (y-x) + \frac12\int_x^y \frac{f^{q_{\rm av}}_+(x;k)f^{q_{\rm av}}_-(t;k)-f^{q_{\rm av}}_-(x;k)f^{q_{\rm av}}_+(t;k)}{W[f^{q_{\rm av}}_+ ,f^{q_{\rm av}}_-]} (q_{\rm av}(t)-k^2) (y-t) \dd t\\
&= \ (y-x) + \frac12\int_x^y \m(x,t;k) (q_{\rm av}(t)-k^2) (y-t) \dd t. \end{align*}
 Therefore, one has with $g_x(y) \ \equiv\ \frac{|\m(x,y;k)|}{|x-y|}$,
\[
g_x(y) \ \leq \ 1 \ + \ \frac1{2|x-y|} \int_x^y g_x(t) |x-t| |q_{\rm av}(t)-k^2| |y-t| \dd t\ \leq \ 1 \ + \ \frac1{2} \int_x^y g_x(t) |x-t| |q_{\rm av}(t)-k^2| \dd t,
\]
since $|y-t|\leq|y-x|$ for $t\in[x,y]$.
By Gronwall's inequality, one has
\[ g_x(y) \ \leq \ \exp\big(\frac12\int_x^y |x-t| |q_{\rm av}(t)-k^2|\big) \dd t \ \leq \ C\big(\big|q_{\rm av}\big|_{L^\infty_\beta}\big)e^{\frac14 k^2 (y-x)^2}.\]
Therefore, we have an estimate on $|\m(x,y;k)|$, uniformly for $k$ such that ${|k||x-y|\leq 1}$.

When $|k||x-y|\geq 1$, one has from Lemma~\ref{Lem:f-gen}
\begin{align*}
|\m(x,y;k)|  &\leq  C\frac{(1+|x|)e^{\alpha|x|}(1+|y|)e^{\alpha|y|}}{W[f^{q_{\rm av}}_+ ,f^{q_{\rm av}}_-]} \\
& \leq  CM_K(1+|x|)(1+|y|)\frac{e^{\alpha|x|}e^{\alpha|y|}}{|k|}  \leq  CM_K(1+|x|)(1+|y|)|x-y| e^{\alpha|x|}e^{\alpha|y|}, 
\end{align*}
where we used that $\frac{1}{W[f^{q_{\rm av}}_+ ,f^{q_{\rm av}}_-](k)} \ = \ \frac{t^{q_{\rm av}}(k)}{-2ik}$ from~\eqref{eq:Wronskian-vs-transmission}, and $|t^{q_{\rm av}}(k)|\ \leq M_K$, from Hypotheses~{\bf (K)}. The estimate~\eqref{est-diffm}, when $l=0$, is now straightforward.

Let us now look at $\partial_y \m(x,y;k)$. Using 
\[\partial_y f^{q_{\rm av}}_\pm(y;k) \ = \ \big(\partial_y f^{q_{\rm av}}_\pm(y;k)\big)\big\vert_{y=x} \ + \ \int_x^y \big(\partial_y^2 f^{q_{\rm av}}_\pm(y;k)\big)\big\vert_{y=t} \dd t,\]
one has the identity
\[\partial_y \m(x,y;k) \ = \ 1+ \int_x^y \m(x,t;k) (q_{\rm av}(t)-k^2) \dd t. \]
If $|k||x-y|\leq 1$, we use that $\m(x,y;k)$ is uniformly bounded, and obtain
\[\big\vert \partial_y \m(x,y;k) \big\vert  \leq  1+ \int_x^y |\m(x,t;k)| |q_{\rm av}(t)-k^2| \dd t  \leq  C(1+ |x-y|+ |k|^2|x-y|)  \leq  C(1+ |x-y|)(1+|k|) . \]
When $|k||x-y|\geq 1$, one uses the definition of $\m$ with Lemma~\ref{Lem:f-gen}, and one obtains as previously
\begin{align*}
|\partial_y \m(x,y;k)|  \leq  CM_K (1+|k|)(1+|x|)(1+|y|)|x-y| e^{\alpha|x|}e^{\alpha|y|}. 
\end{align*}
Estimate~\eqref{est-diffm} follows for $l=1$, using the symmetry $\m(x,y;k)=-\m(y,x;k)$.

Estimate~\eqref{est-diffm} for $l=2$ is straightforward when remarking that
\[ \partial_y^2 \m(x,y;k) \ = \ (q_{\rm av}(y)-k^2)\m(x,y;k),\]
and the case $l=3$ follows in the same way.
\medskip

The proof when $k\in\RR$ and $q_{\rm av},\ \partial_x q_{\rm av}\in \mathcal{L}^1_{2}$ is identical, using the estimates of Lemma~\ref{Lem:f-gen-k-real} instead of Lemma~\ref{Lem:f-gen}. Note that $M_K=1$ for $k\in\RR$, using~\eqref{tr-energy}.
\end{proof}

\section{Transmission coefficient of $\sigma(x) \ \equiv \ -\eps^2\Lambda(x)$}\label{sec:seff-tools}
In this section, we study the transmission coefficient of potentials of the form $\sigma(x) \ \equiv \ -\eps^2 \Lambda(x)$, where $\Lambda\in L^{\infty}_\beta$, is independent of $\eps$. We are particularly interested in the special case where $\sigma(x)$ is the effective potential 
\[ \seff(x) \ \equiv \ \ -\frac{\eps^2 }{4\pi^2}\sum_{j\neq 0}\frac{|q_{j}(x)|^2}{{\lambda_j}^2},\]
derived earlier.


\begin{Lemma}[Transmission coefficient $t^{q_{\rm av}-\eps^2\Lambda}(k)$] \label{Lem:tsigma} 
Let $q_{\rm av}$ and $ \Lambda$ be any functions in $L^\infty_\beta$. Then, for $k\in K$ satisfying Hypotheses {\bf (K)}, one has
\begin{equation}
\frac{k}{t^{q_{\rm av}-\eps^2\Lambda}(k)} \ = \ \left(\ \frac{k}{t^{q_{\rm av}}(k)} \ -\ \frac{i\eps^2}{2}\int_{-\infty}^\infty f^{q_{\rm av}}_-(y;k)\Lambda(y)f^{q_{\rm av}}_+(y;k)\dd y \ \right)\ + \ \O\big(\eps^4\big).
\label{tavLam}
\end{equation}
\end{Lemma}
\begin{proof}

We recall the identity~\eqref{IVW-def}, satisfied by the transmission coefficient related to {\em any potential} $V,W\in L^\infty_\beta$:
\[\frac{k}{t^V(k)} \ = \ \frac{k}{t^W(k)}  -  \frac{ I^{[V,W]}(k)}{2i},\ \  \text{with } \ I^{[V,W]}(k) \equiv \int_{-\infty}^\infty f^W_-(y;k)(V-W)(y)f^V_+(y;k)\dd y.\]

Now, in the case where $W\equiv q_{\rm av}$ and $V \ \equiv \ q_{\rm av}-\eps^2\Lambda(x)$, one has
\[\frac{k}{t^{q_{\rm av}-\eps^2\Lambda}(k)} \ - \ \frac{k}{t^{q_{\rm av}}(k)} \ =\ - \frac{i\eps^2}{2} I^{\eps}(k), \quad I^{\eps}(k) \ \equiv \ \int_{-\infty}^\infty f^{q_{\rm av}}_-(y;k)\Lambda(y)f^{q_{\rm av}-\eps^2\Lambda}_+(y;k)\dd y.\]

Then, the Volterra equation~\eqref{eq:Volterra1} with $V=q_{\rm av}-\eps^2\Lambda$ and $W=q_{\rm av}$, leads to 
\[
f^{q_{\rm av}-\eps^2\Lambda}_+(x;k)  =  f^{q_{\rm av}}_+(x;k)  - \eps^2 \ \int_x^\infty \Lambda(y) \frac{f^{q_{\rm av}}_+(x;k)f^{q_{\rm av}}_-(y;k)-f^{q_{\rm av}}_-(x;k)f^{q_{\rm av}}_+(y;k)}{W[f^{q_{\rm av}}_+ ,f^{q_{\rm av}}_-]} f^{q_{\rm av}-\eps^2\Lambda}_+(y;k) \dd y.\]
We can then use the estimates of Lemmata~\ref{Lem:f-gen} and~\ref{Lem:m}, so that 
\begin{align*}
& \left|\ I^{\eps}(k)-\int_{-\infty}^\infty f^{q_{\rm av}}_-(y;k)\Lambda(y)f^{q_{\rm av}}_+(y;k)\dd y\ \right| \\
 &\qquad \leq \ 
C\eps^2 \int_{-\infty}^\infty f^{q_{\rm av}}_-(y;k) \Lambda(y)\int_y^\infty \Lambda(z) \m(y,z;k) f^{q_{\rm av}-\eps^2\Lambda}_+(z;k) \dd z \dd y\\
 &\qquad \le \eps^2 M_K C , \qquad \text{ uniformly for } k\in K.\end{align*}
 This concludes the proof. \end{proof}

A simple consequence is the following
\begin{Corollary}\label{cor:tseff} Let $q_{\rm av}$ and $ \Lambda$ be functions in $L^\infty_\beta$. Then,
\begin{enumerate}[(1)]
\item If $q_{\rm av}$ is generic, in the sense of Definition~\ref{def:generic}, then $q_{\rm av}-\eps^2\Lambda$ is generic for $\eps$ sufficiently small.
\item If $q_{\rm av}$ is non-generic, and $\int_{-\infty}^\infty \Lambda(y)(f^{q_{\rm av}}_+(y;0))^2\dd y\neq 0$, then $q_{\rm av}-\eps^2\Lambda$ is generic for $\eps$ sufficiently small.
\item
 If $q_{\rm av}\equiv0$, and $k\in K$ satisfy Hypotheses {\bf (K)}. Then, 
\begin{equation}
\frac{k}{t^{-\eps^2\Lambda}(k)} \ = \ \ k\ -\ 
\frac{i\eps^2}{2}\int_{-\infty}^\infty\ \Lambda(y)\ \dd y\ + \ \O\big(\eps^4\big),
\label{eqnfork-app}
\end{equation}
uniformly in $k\in K$. 
  It follows that if
 \[
 \left| k - \frac{i\eps^2}{2}\int_{-\infty}^\infty \Lambda\right|\ \geq\ C\ \max( \eps^\tau, |k|), \qquad \text{for}\ \ \ \tau<4,\ \ k\in K, 
 \]
then one has
 \begin{equation}
\left| t^{-\eps^2\Lambda}(k) \ - \ \frac{k}{k-\frac{i\eps^2}{2}\int_{-\infty}^\infty \Lambda} \right| \ = \ \O\big(\eps^{4-\tau}\big).
\label{compare-tseffq0iszero}
\end{equation}
\end{enumerate}
\end{Corollary}

\begin{proof}
As discussed in section~\ref{sec:generic}, a potential, $V$, is generic, if and only if its transmission coefficient satisfies $t^V(0)=0$ or, equivalently, if 
$\lim_{k\to 0} \frac{k}{t^V(k)}\neq 0$. Items~(1) and~(2) are therefore a straightforward consequence of~\eqref{tavLam}. As for item~(3),
since $q_{\rm av}(x)\equiv0$, we have $t^{q_{\rm av}}\equiv1$ and 
$f^{q_{\rm av}}_\pm(x;k)=e^{\pm ikx}$. The result follows by substitution into~\eqref{tavLam}, and straightforward computations. \end{proof}


\noindent{\bf Acknowledgements:} I.V. was supported in part by U.S. NSF EMSW21- RTG: Numerical Mathematics for Scientific Computing and U.S. NSF DGE IGERT-1069240. M.I.W. was supported in part by U.S. NSF Grant DMS-10-08855.

The authors wish to thank Percy Deift, Jared C. Bronski and Vadim Zharnitsky for stimulating discussions.


\frenchspacing
\bibliographystyle{plain}
\def\cprime{$'$}

\end{document}